\documentclass[10pt]{article}

\usepackage{graphicx,verbatim}
\usepackage{amsmath,amssymb,amsthm}
\usepackage{algorithm,algorithmic}
\usepackage[small,bf]{caption}
\usepackage{subcaption} 
\usepackage{cases}
\usepackage[table,dvipsnames]{xcolor}
\usepackage[shortlabels]{enumitem}
\usepackage{hyperref,url}
\usepackage{cite}
\hypersetup{
	colorlinks=true,
	linkcolor=red,     
	urlcolor=magenta,
	citecolor={blue},
}
\usepackage{booktabs,adjustbox,multirow}  
\usepackage{import} 
\usepackage{aligned-overset} 
\usepackage[toc,page]{appendix}
\usepackage{dsfont,bm}
\usepackage{lmodern}
\usepackage{upgreek} 

\newcommand{\grad}{{\nabla}}   
\newcommand{\zero}{\mathbf{0}}  
\newcommand{\one}{\mathbf{1}}   
\newcommand{\real}{\mathbb{R}}  

\newcommand{\diag}{{\rm diag}}  
\newcommand{\bdiag}{{\rm blkdiag}}  
\newcommand{\col}{\mathrm{col}}     

\newcommand{\Tr}{\mathop{\rm Tr}}    
\newcommand{\prox}{\operatorname{prox}}    

\newcommand{\Ex}{\mathop{\mathbb{E}{}}}

\newcommand{\eg}{{\it e.g.}}
\newcommand{\ie}{{\it i.e.}}

\newcommand{\tran}{^{\textit{\footnotesize \texttt{T}}}} 
 
\newcommand{\define}{\triangleq} 

\newcommand{\qd}{\hfill{$\blacksquare$}}

\newcommand{\alg}[1]{\textsf{#1}} 

\def\bxi{\boldsymbol{\xi}}

\def\A{{\mathbf{A}}}
\def\B{{\mathbf{B}}}

\def\D{{\mathbf{D}}}

\def\F{{\mathbf{F}}}

\def\I{{\mathbf{I}}}

\def\P{{\mathbf{P}}}
\def\Q{{\mathbf{Q}}}

\def\V{{\mathbf{V}}}
\def\W{{\mathbf{W}}}

\def\d{{\mathbf{d}}}

\def\f{{\mathbf{f}}}
\def\g{{\mathbf{g}}}
\def\h{{\mathbf{h}}}

\def\s{{\mathbf{s}}}

\def\u{{\mathbf{u}}}
\def\v{{\mathbf{v}}}

\def\x{{\mathbf{x}}}
\def\y{{\mathbf{y}}}
\def\z{{\mathbf{z}}}

\newcommand{\cE}{{\mathcal{E}}}

\newcommand{\cN}{{\mathcal{N}}}
\newcommand{\cO}{{\mathcal{O}}}

\newcommand{\cU}{{\mathcal{U}}}

\def\evdots{\vbox{\baselineskip=2pt \lineskiplimit=0pt 
		\kern6pt \hbox{$.$}\hbox{$.$}\hbox{$.$}}}  
\newtheorem{lemma}{{Lemma}}
\newtheorem{theorem}{{Theorem}}
\newtheorem{corollary}{{Corollary}}
\newtheorem{proposition}{{Proposition}}
\newtheorem{assumption}{{Assumption}}
\newtheorem{remark}{{Remark}}

\usepackage[margin=1in,
includefoot,headsep=20pt 
]{geometry}
\usepackage{sectsty}
\allsectionsfont{\sffamily}
\usepackage{fancyhdr} 
\begin{document}

	\title{\textsf{\bfseries  Local Exact-Diffusion for Decentralized Optimization and Learning}}
	\author{Sulaiman A. Alghunaim \\
		\small Kuwait University \\
		\texttt{\small sulaiman.alghunaim@ku.edu.kw}
 \vspace{5mm}\\
	}
	\maketitle
	
	

	\begin{abstract}
	Distributed optimization methods with local updates have recently attracted a lot of attention due to their potential to reduce the communication cost of distributed methods. In these algorithms, a collection of nodes performs several local updates based on their local data, and then they communicate with each other to exchange estimate information. While there have been many studies on distributed local methods with centralized network connections, there has been less work on decentralized networks.

	In this work, we propose and investigate a locally updated decentralized method called Local Exact-Diffusion (LED). We establish the convergence of LED in both convex and nonconvex settings for the stochastic online setting.   Our convergence rate  improves over the rate of existing decentralized methods. When we specialize the network to the centralized case, we recover the state-of-the-art bound for centralized methods. We also link LED to several other  independently studied  distributed methods, including  Scaffnew, FedGate, and VRL-SGD.  Additionally, we numerically investigate the benefits of local updates for decentralized networks and demonstrate the effectiveness of the proposed method.
	
	\end{abstract}
	
	\section{Introduction}
This work examines the distributed consensus optimization problem, as formally presented in \eqref{main_prob}. In this setup, a network of nodes (also referred to as agents, workers, or clients) collaboratively seeks to minimize the average of the nodes' objectives. This formulation is appealing for large scale data problems because it is more efficient to use distributed solution methods to reduce the computational burden for large data sets.   In terms of communication protocol, distributed methods  can be classified as either {\em centralized} or {\em decentralized}\footnote{In this work, the term ``distributed methods''  refers to the class of methods that includes both centralized (server-workers) and decentralized approaches.}.   Centralized distributed methods require all nodes to communicate with a central server (\eg, server-workers connection) without sharing private data, as seen in parallel optimization \cite{boyd2011admm} and federated learning \cite{konevcny2016federated,mcmahan2017communication}. In this setup, there is a central node that is responsible for aggregating local variables and updating model estimates. In contrast, decentralized distributed methods are ``fully distributed'' that are designed for arbitrary connected network topologies such as line, ring, grid, and random graphs. These methods require nodes to communicate only with their immediate neighbors \cite{nedic2009distributed,lopes2008diffusion}. It's important to note that decentralized methods can adapt to a centralized setting when the network is fully connected.

In this paper, we consider  a group of $N$ nodes, connected via an undirected decentralized network, collaborating to solve the optimization problem:
	\begin{align} \label{main_prob}
		\begin{array}{ll}
			\min\limits_{x \in \real^m} &  f(x) \define \displaystyle \frac{1}{N}\sum\limits_{i=1}^N f_i(x), \quad f_i(x) \define \Ex [F_i(x;\xi_i)],
		\end{array}
	\end{align}
	where  $f_i: \real^m \rightarrow \real$ ($i=1,\dots,N$) represents a smooth function known only to node $i$. This function is defined as the expected value of some loss function $F_i(\cdot ; \xi_i)$ over the random variable or data $\xi_i$. We focus on the stochastic online setting, in which each node has access only to random samples of its data $\{\xi_i\}$. Problems of the form \eqref{main_prob} have received a lot of attention in control and engineering communities \cite{nedic2009distributed,alghunaim2019decentralized,cattivelli2010diffusion,chen2013distributed}, as well as in the machine learning community \cite{lian2017can,karimireddy2019scaffold,koloskova2020unified} .

Our main contribution is the proposal and study of a local variant of the {\em decentralized \alg{Exact-Diffusion}} algorithm \cite{yuan2019exactdiffI,yuan2020influence}  (see also \cite{alghunaim2019decentralized}), where nodes employ multiple local updates between communication rounds.	  We establish  the algorithm's convergence in  both convex and nonconvex settings. . Our bounds improve upon those of decentralized methods and match the best known results for centralized methods. Before we formally state our contributions, we will first discuss some related works.

\subsection{Related works}
We  begin by discussing relevant centralized methods, which require a central server for implementation. One of the most  popular centralized methods is \alg{FedAvg},  which involves  a random subset of nodes performing multiple local stochastic gradient descent (\alg{SGD}) updates at each round; They then send their estimates (parameters) to the central server, which averages these estimates and sends them back to replace the local estimates \cite{mcmahan2017communication}.  It should be noted that \alg{FedAvg} is often called \alg{Local-SGD}\footnote{In this paper, we refer to the case where all of the nodes participate in each round as \alg{Local-SGD}.}. Several works have analyzed  \alg{FedAvg} and \alg{Local-SGD} \cite{stich2018local,khaled2019tighter,wang2019adaptive,wang2021cooperative,li2020on,woodworth2020minibatch}. It has been observed that the performance of \alg{FedAvg} and \alg{Local-SGD} is suboptimal for heterogeneous data and that an increased number of local steps can lead to worse performance \cite{zhao2018federated}. One major reason is that node estimates drift toward their local solutions due to local updates, resulting in a biased solution \cite{li2020on,woodworth2020minibatch,karimireddy2019scaffold}. To correct this drift in \alg{FedAvg} and \alg{Local-SGD}, several algorithms have been proposed,   including   \alg{SCAFFOLD} \cite{karimireddy2019scaffold}, \alg{FedDyn}  \cite{zhang2021fedpd}, \alg{FedPD} \cite{durmus2021federated}, \alg{VRL-SGD} \cite{liang2019variance}, and \alg{FedGATE} \cite{haddadpour2021federated}. These methods, however, are only applicable to centralized connections.

In this work, we focus on decentralized setups as presented in \cite{lopes2008diffusion,nedic2009distributed}.  The most extensively studied method for this setup is the decentralized stochastic gradient descent method (\alg{DSGD})  \cite{ram2010distributed,chen2012diffusion,chen2013distributed,lian2017can,wang2021cooperative,koloskova2020unified}\footnote{\alg{DSGD} has two main implementations depending on the combination step: the adapt-then-combine (ATC) implementation (aka diffusion) and the non-ATC implementation (aka consensus) \cite{cattivelli2010diffusion,sayed2014nowbook}.  Both implementations are termed \alg{DSGD} in this paper.}.   The study in \cite{lian2017can} showcased that \textsf{\alg{DSGD}} achieves the centralized \alg{SGD} rate asymptotically, a distributed trait termed as {\em linear speedup} \cite{dekel2012optimal}.  However,  \alg{DSGD}  converges to a biased solution, and this bias is further negatively influenced by network sparsity \cite{yuan2020influence,koloskova2020unified}.  The bias arises from the heterogeneity among the local functions ${f_i}$, as characterized by $(1/N) \sum_{i=1}^N \|\grad f_i(x) - \grad f (x)\|^2$. This value, which is required to be bounded for the analysis of \alg{DSGD}  \cite{lian2017can,koloskova2020unified}, can become quite large when the functions are heterogeneous, thereby slowing down the convergence of \alg{DSGD} \cite{yuan2020influence,alghunaim2021unified}.  Multiple studies have introduced bias-correction algorithms resistant to local function heterogeneity,  such as the alternating direction method of multipliers (\alg{ADMM}) based methods \cite{chang2015multi,ling2015dlm}, \alg{EXTRA} \cite{shi2015extra}, \alg{Exact-Diffusion} (\alg{ED}) \cite{yuan2019exactdiffI,yuan2020influence} (also known as \alg{NIDS} \cite{li2017nids} and \alg{D$^2$} \cite{tang2018d}), and  \alg{Gradient-Tracking} (\alg{GT})  methods \cite{di2016next,xu2015augmented,nedic2017achieving,qu2017harnessing}. It has been established that these bias-correction methods outperform \alg{DSGD} \cite{alghunaim2021unified}. Yet, all these methods necessitate communication at every iteration.

	Locally updated stochastic decentralized methods have received less attention than centralized methods and are more challenging to study. Federated learning can be viewed as a subset of decentralized optimization and learning under time-varying and asynchronous updates \cite{zhao2015asynchronous}. For instance, \alg{DSGD} with local steps, termed \alg{Local-DSGD}, has been studied in \cite{li2019communication,haddadpour2019convergence,koloskova2020unified,wang2021cooperative} and is analogous to \alg{Local-SGD} when the network is fully connected. However, just as \alg{DSGD} suffers from bias, \alg{Local-DSGD} does too. Furthermore, similar to \alg{Local-SGD}, it also experiences drift. Increasing the number of local steps exacerbates this drift in the solution, requiring the use of very small stepsizes, which significantly slows down convergence.

	Only a few works have studied decentralized methods with local updates and bias correction. The work in \cite{nguyen2022performance} studied \alg{Local Gradient-Tracking} (\alg{LGT}) under nonconvex costs, but it focused solely on deterministic settings. Similarly, the research presented in \cite{liu2023decentralized} explored a locally updated stochastic variant of gradient-tracking, namely $K$-\alg{GT}, but it too addressed only nonconvex settings. In this paper, we propose and investigate a different algorithm inspired by \cite{yuan2019exactdiffI,alghunaim2019decentralized}. Our results improve upon the rates of local \alg{GT}-based methods and require just half the communication cost of existing approaches. Furthermore, our rates surpass those of \alg{Local-DSGD} \cite{koloskova2020unified,wang2021cooperative}. Importantly, when employing a constant step size, \alg{LED} achieves precise convergence in the deterministic (noiseless) case. In contrast, \alg{Local-DSGD} does not, due to the bias/drift introduced by the heterogeneity in local functions, as discussed earlier. We will next formally outline our contributions.

	\subsection{Contribution}
	\begin{itemize}
		
		\item We propose \alg{Local Exact-Diffusion} (\alg{LED}) method for distributed optimization with local updates. An advantage over previous methods is that \alg{LED} is a decentralized method that only requires one single vector communication per link and is robust to the functions heterogeneity  -- See Table \ref{table_comparison}. Numerical results are provided to demonstrate the effectiveness of \alg{LED} over other methods.

		\item 		 We provide insights and draw connections between our proposed method  and the following state-of-the-art algorithms: \alg{Exact-Diffusion} \cite{yuan2019exactdiffI}, \alg{NIDS} \cite{li2017nids},  \alg{D$^2$} \cite{tang2018d}, \alg{ProxSkip/Scaffnew} \cite{mishchenko2022proxskip}, 	\alg{VRL-SGD} \cite{liang2019variance}, and \alg{FedGATE} \cite{haddadpour2021federated}. For instance, we demonstrate that \alg{LED} can be interpreted as  \alg{Scaffnew}  \cite{mishchenko2022proxskip} with fixed local updates instead of random ones. We also show that \alg{LED}  is a decentralized variant of the centralized methods \alg{FedGATE} \cite{haddadpour2021federated} and \alg{VRL-SGD} \cite{liang2019variance}. Furthermore, we highlight that all these methods can be traced back to the primal-dual method  \alg{PDFP2O} \cite{chen2013aprimal} (also known as \alg{PAPC} \cite{drori2015simple}) in the case of a single local update.

		\item      We establish the convergence of \alg{LED} in both (strongly-)convex and nonconvex environments for online stochastic learning settings. Our rates improve upon existing bounds for local decentralized methods—see Table \ref{table_rates}.  It is worth noting that the analysis of \alg{LED} is more challenging than that of \alg{Local-DSGD}, even in the single local update scenario. Additionally, in contrast to \alg{Local-DSGD}, \alg{LED} is robust to heterogeneity in local functions and converges exactly in the deterministic case (no noise), as discussed earlier.

		A byproduct of our result is that, when adapting our analysis to centralized networks, we achieve new and tighter analyses for the methods \alg{VRL-SGD} \cite{liang2019variance} and \alg{FedGATE} \cite{haddadpour2021federated}.

	\end{itemize}

	\begin{table*}[t]  \small
		\renewcommand{\arraystretch}{1.5}
		\begin{center}
			\caption{\small Differences with existing methods employing local steps.   }
			\begin{adjustbox}{max width=1\columnwidth}
				\begin{tabular}{lccc} \toprule
					{Method}  & {Decentralized} &
				 Single  communication  & {Robust to functions heterogeneity}
				\\ \midrule
			\alg{SCAFFOLD} \cite{karimireddy2019scaffold}   
			&
		{\color{red}	$\bm{\times}$ }
			&
			{\color{red}	$\bm{\times}$ }
			&
			\checkmark 
			\\
			\alg{VRL-SGD} \cite{liang2019variance}/\alg{FedGate} \cite{haddadpour2021federated}   
			&
			{\color{red}	$\bm{\times}$ }
			&
			\checkmark 
			&
			\checkmark 
			\\
			\alg{Local-DSGD} \cite{koloskova2020unified}   
			&
			\checkmark 
			&
			\checkmark
			&
		{\color{red}	$\bm{\times}$ }
			\\
			\alg{LGT}   \cite{nguyen2022performance} and	\alg{$K$-GT}   \cite{liu2023decentralized} 
			&
			\checkmark
			&
			{\color{red}	$\bm{\times}$ }
			&
			\checkmark
			\\
	\rowcolor[gray]{0.9}	\alg{LED}	\textbf{[this work]}     
			&
			\checkmark
			&
			\checkmark
			&
			\checkmark
			\\  \bottomrule
				\end{tabular}
			\end{adjustbox}
			\label{table_comparison}
		\end{center} 
	\end{table*}

{\bf Notation.} Lowercase letters represent vectors and scalars, while uppercase letters denote matrices. The notation $\col\{a_1,\ldots,a_n\}$ (or $\col\{a_i\}_{i=1}^n$) stands for the vector that stacks the vectors (or scalars) $a_i$ on top of each other. We use $\diag\{d_1,\ldots,d_n\}$ (or $\diag\{d_i\}_{i=1}^n$) to represent a diagonal matrix with the diagonal elements $d_i$. Additionally, the symbol $\bdiag\{D_1,\ldots,D_n\}$ (or $\bdiag\{D_i\}_{i=1}^n$) denotes a block diagonal matrix with diagonal blocks $D_i$. The notations $\one$ and $\zero$ represent vectors of all ones and zeros, respectively. The dimension is determined from the context, or we use notation like $\one_n$. The inner product of two vectors $a$ and $b$ is given by $\langle a,b \rangle$. The symbol $\otimes$ indicates the Kronecker product operation. We use upright bold symbols, such as $\x,\f,\W$, to represent augmented network quantities.

{\bf Outline.} This paper is organized as follows: Section \ref{sec_led} introduces our algorithm and its motivation. In Section \ref{sec_connection}, we compare our method with other leading approaches. Section \ref{sec_results} presents our core assumptions and convergence findings, and a discussion comparing our results with prior works. Section \ref{sec-sim} provides simulation outcomes, and conclusions are drawn in Section \ref{sec_conc}. Detailed proofs are reserved for the appendix.

	\section{Local Exact-Diffusion} \label{sec_led}
In this section, we start by describing the proposed algorithm in its decentralized implementation. We then rewrite it in network notation for reasons of analysis and interpretation.
	
	\subsection{Algorithm description}
The method under study is described in \texttt{Alg.~\ref{alg:led}} and is named \alg{Local Exact-Diffusion} (\alg{LED}). In \texttt{step 1}, each node \(i\) employs \(\uptau\) local updates, starting from the initialization \(x_i^r\), which is its local estimate of the solution after the communication round \(r\). \texttt{Step 2} is the communication round during which each node \(i\) sends its local intermediate estimate \(\phi_{i,\uptau}^{r}\) to its neighbors \(j \in \cN_i\), where the symbol \(\cN_i\) denotes the set of neighbors of node \(i\) (including node \(i\)); in this step, \(w_{ij}\) is a nonnegative scalar weight that node \(i\) uses to scale the information received from node \(j \in \cN_i\). The final step, \texttt{step 3}, is where each node \(i\) updates its (dual) estimate \(y_i^r \in \real^m\).
	
	\begin{algorithm}[h] 
		\caption{\alg{Local Exact-Diffusion (LED)}}
		\textbf{node $i$ input:} $x_i^0$,   $\alpha>0$, $\beta>0$, and $\uptau$.
		
		\textbf{initialize} $y_i^0=x_i^0-\sum_{j \in \cN_i} w_{ij} x_j^{0}$ (or  $y_i^0=0$).
		
		\textbf{repeat for} $r=0,1,2,\dots$
	\begin{enumerate}\itemsep=-10pt
		\item 	 \texttt{Local primal updates:} set $\phi_{i,0}^r=x_{i}^{r}$ and do $\uptau$ local updates:  
		\begin{subequations} \label{ed_updates_dec}
			\begin{align}  \label{local_ed_per_node}
				\phi_{i,t+1}^r &=    \phi_{i,t}^r-\alpha  \grad F_i(\phi_{i,t}^r;\xi_{i,t}^r)   -  \beta  y_i^{r} , \quad t=0,\dots,\uptau-1 .
			\end{align}
			
			\item  \texttt{Diffusion:} 
			\begin{align} 
				x_i^{r+1} &= \sum_{j \in \cN_i} w_{ij} \phi_{j,\uptau}^{r}. \label{local_ed_comm_round_per_node} 
			\end{align}
		
		\item \texttt{Local dual update:} 
		\begin{align} \label{local_dual_per_node}
				y_i^{r+1} &= y_i^{r} +  \phi_{i,\uptau}^{r} -x_i^{r+1}.
		\end{align}
		
		\end{subequations}
	\end{enumerate}
		\label{alg:led} 
	\end{algorithm}

	\subsection{Networked  description}
The \alg{LED} method, as listed in \ref{alg:led}, is described at the node level. For the analysis and interpretation of the method, we will present it in a networked form. To do this, we introduce the following network weight matrix notation:
 
	\begin{align} \label{weight_matrix}
	W \define [w_{ij}] \in \real^{N \times N}, \quad 	\W \define W \otimes I_m \in \real^{mN \times mN}.
	\end{align}
Using the above notation, we have $\W \u=\col\{\sum_{j \in \cN_i} w_{ij} u_{j}\}_{i=1}^N$ for any vector $\u$ with the structure $\u=\col\{u_{1},\dots,u_{N}\}$, where $u_i \in \real^m$ \cite{sayed2014nowbook}. Thus, if we introduce the augmented network quantities:
	\begin{subequations} \label{notation:network_main}
		\begin{align}
					\x^r & \define \col\{x_1^r,\dots,x_N^r\} \in \real^{mN} \\
			\bm{\Phi}^{r}_{t} & \define \col\{\phi_{1,t}^{r},\dots,\phi_{N,t}^{r}\} \in \real^{mN} \\
			\y^r & \define \col\{y_1^r,\dots,y_N^r\} \in \real^{mN} \\
			\f(\x)& \define \sum_{i=1}^N f_i(x_i) \\
			\grad \f(\x)& \define \col\{\grad f_1(x_1),\dots,\grad F_N(x_N)\} \in \real^{mN} \\
			\grad \F(\x; \bxi)& \define \col\{\grad F_1(x_1;\xi_1),\dots,\grad F_N(x_N;\xi_N)\} \in \real^{mN}.
		\end{align}
	\end{subequations}
Then, Algorithm \ref{alg:led} can be represented in a compact networked form as follows: Given $\x^0$, set $\y^0=(\I-\W)\x^0$ (or $\y^0=\zero$) and update for $r=0,1,2,\dots$
	\begin{enumerate}
		
		\item {\em Local primal updates:}  set $\bm{\Phi}_{0}^r=\x^{r}$, for $t=0,\dots,\uptau-1$: 
		\begin{subequations}  \label{ed_updates_network}
			\begin{align}  \label{local_ed_updates}
				\bm{\Phi}^{r}_{t+1} &=    \bm{\Phi}^{r}_{t}-\alpha  \grad \F(\bm{\Phi}^{r}_{t};\bxi_t^r)   -  \beta  \y^{r} . 
			\end{align}

			\item {\em Diffusion round:}
			\begin{align} 
				\x^{r+1} &=  \W \bm{\Phi}^r_{\uptau} . \label{local_ed_comm_round}
			\end{align}

		\item {\em Local dual update:}
		\begin{align}
				\y^{r+1} &= \y^{r}+ (\I-\W) \bm{\Phi}^r_{\uptau}.
		\end{align}
		\end{subequations}
	\end{enumerate}
The networked description \eqref{ed_updates_network} will be used for analysis purposes.

	\subsection{Motivation and relation with \alg{Exact-Diffusion}}
\alg{Exact-Diffusion} (\alg{ED}) was derived in \cite{yuan2019exactdiffI} and takes the following form:
\begin{align} \label{ed_eliminated}
	\x^{r+1} &= \W \big(2 \x^r -\x^{r-1} - \alpha ( \grad \f(\x^{r})- \grad \f(\x^{r-1})) \big).
\end{align}
The unified decentralized algorithm (\alg{UDA}) from \cite{alghunaim2019decentralized} demonstrated that the iterates $\x^{r}$ of \alg{ED} in \eqref{ed_eliminated} can be equivalently described by:
\begin{subnumcases}{ \alg{UDA-ED} ~ \label{uda}}
	\bm{\Phi}^{r} = \x^{r}-\alpha \grad \f(\x^{r}) - \B^{1/2} \z^{r} \\
	\x^{r+1} = \W \bm{\Phi}^r \\
	\z^{r+1} = \z^{r}+ \B^{1/2} \bm{\Phi}^r,
\end{subnumcases}
where $\B= \I-\W$. The above form is convenient for analytical purposes; however, it cannot be implemented in a decentralized fashion due to $\B^{1/2}= (\I-\W)^{1/2}$.

To derive our method, we set $\B=\beta(\I-\W)$ in \eqref{uda} and introduce the change of variable $\y^r=\frac{1}{\beta} \B^{1/2} \z^{r}$. This leads to the following description:
\begin{subnumcases}{\alg{LED-1} ~\label{uda_ed}}
	\bm{\Phi}^{r} = \x^{r}-\alpha \grad \f(\x^{r}) - \beta \y^{r} \\
	\x^{r+1} = \W \bm{\Phi}^r \\
	\y^{r+1} = \y^{r}+ (\I-\W) \bm{\Phi}^r.
\end{subnumcases}
It can be observed that the update \alg{LED-1} \eqref{uda_ed} is equivalent to \alg{LED} \eqref{ed_updates_network} when $\uptau=1$. In other words, \alg{LED} \eqref{ed_updates_network} is an extension of \alg{LED-1} \eqref{uda_ed} that incorporates local updates.

\begin{remark}[\sc  NIDS, and D$^2$] \rm
	The \alg{NIDS} method from \cite{li2017nids} is given by
	\begin{align*}
		\x^{r+1} = \widetilde{\W} \left(2 \x^{r}-\x^{r-1}-\alpha \left(\nabla \f(\x^{r}) - \nabla \f(\x^{r-1}) \right) \right),
	\end{align*}
	where $\widetilde{\W} = (1-\alpha c) \I + \alpha c \W$ and $c$ is a stepsize parameter. When $c=1/\alpha$, \alg{NIDS} reduces to \alg{ED} \eqref{ed_eliminated}. We also note that \alg{ED} (or \alg{NIDS} with $c=1/\alpha$) has been studied under the name \alg{D$^2$} \cite{tang2018d}. Thus, \alg{LED} can be viewed as a modification of \alg{NIDS/D$^2$} that incorporates local updates.
	\qd
\end{remark}
\begin{remark}[\sc Non-Equivalence with Local Updates] \rm
	It is important to note that the updates \eqref{ed_eliminated}, \eqref{uda}, and \eqref{uda_ed} are equivalent only when there are no multiple local updates and $\beta=1$. To understand this, observe that the local updates variants can be modeled as a time-varying graph $\W_r$, where $\W_r=\W$ when $r=\uptau,2 \uptau, 3 \uptau, \dots$, and $\W_k=\I$ otherwise. In this scenario, the updates $\x^r$ differ for all these methods. Indeed, for this case, the updates \eqref{ed_eliminated} and \eqref{uda} are not guaranteed to converge and often diverge in simulations.
	\qd
\end{remark}

\section{Connection with existing algorithms}	\label{sec_connection}
In this section, we discuss and highlight the connections of \alg{LED} to the following algorithms: \alg{Scaffnew/ProxSkip} \cite{mishchenko2022proxskip}, \alg{VRL-SGD} \cite{liang2019variance}, and \alg{FedGate} \cite{haddadpour2021federated}. We also demonstrate that all these methods can be traced back to the primal-dual method \alg{PDFP2O} \cite{chen2013aprimal}, which is also known as \alg{PAPC} \cite{drori2015simple} and was initially proposed in \cite{loris2011generalization} for quadratic objectives.

\subsection{Relation with \alg{PDFP2O}/\alg{PAPC}}
We now  demonstrate that \alg{LED} \eqref{ed_updates_network} with $\uptau=1$ (\ie, \alg{LED-$1$} \eqref{uda_ed})  can be interpreted as the primal-dual algorithm  \alg{PDFP2O} \cite{chen2013aprimal}  applied to the following reformulation of problem \eqref{main_prob}:
\begin{align} \label{reformulation}
	\min_{\x} \quad \f(\x) + \g(\B^{1 \over 2} \x),
	\end{align}
where $\B =\I-\W$ and $\g$ is the indicator function of zero, \ie, $\g(\u)=0$ if $\u=\zero$ and $g(\u)= + \infty$ otherwise. Problem \eqref{reformulation} is equivalent to \eqref{main_prob} because $\B \x=\zero$ if and only if $x_1=x_2=\dots=x_N$ -- see \cite{shi2015extra,alghunaim2019linear}.

The following updates are obtained when \alg{PDFP2O} \cite{chen2013aprimal} is applied to formulation \eqref{reformulation}:
\begin{subnumcases}{\alg{PDFP2O}~ \label{papc}} 
\v^{r+1}= \prox_{\frac{\alpha}{\eta} \g^*} \left(\B^{1 \over 2} (\x^{r}-\alpha \nabla \f(\x^{r}) )+\left(\I-\eta \B \right)\v^{r}\right) \\
\x^{r+1}=\x^{r}-\alpha \nabla \f(\x^{r}) -\eta \B^{1 \over 2} \v^{r+1},
\end{subnumcases} 
where $\prox_{\frac{\alpha}{\eta} \g^*}(\cdot)$ denotes the proximal operator of the conjugate of $g$ and $\alpha,\eta >0$ are stepsize parameters.   The following result relates \alg{LED-1} \eqref{uda_ed}  (\alg{LED} \eqref{ed_updates_network} with $\uptau=1$) with \alg{PDFP2O} \eqref{papc}. 
\begin{proposition}[\sc Relation to \alg{PDFP2O}] \label{pro}
The updates of \alg{PDFP2O} \eqref{papc} can be rewritten as
\begin{subequations} \label{nids}
	\begin{align}
		\bm{\Phi}^{r} &= \mathbf{x}^{r}-\alpha \grad \mathbf{f}(\mathbf{x}^{r}) - \eta \y^{r} \label{nids_a} \\
		\mathbf{y}^{r+1} &= \mathbf{y}^{r}+ (\I-\W) \bm{\Phi}^{r} \label{nids_b} \\
		\mathbf{x}^{r+1}&= \left( (1-\eta)\I+ \eta \W \right) \bm{\Phi}^{r}.  \label{nids_c}
	\end{align}
\end{subequations}
It follows that  \alg{PDFP2O} \eqref{nids} is  equivalent to \alg{LED-1} \eqref{uda_ed}  when $\eta=\beta=1$.
\end{proposition}
\begin{proof}
If we let $	\bm{\Phi}^{r} = \mathbf{x}^{r}-\alpha \grad \mathbf{f}(\mathbf{x}^{r}) - \eta \mathbf{B}^{1 \over 2} \mathbf{v}^{r}$, then we can rewrite equation \eqref{papc} as follows:
\begin{subequations}  \label{papc2}
	\begin{align}
		\bm{\Phi}^{r} &= \mathbf{x}^{r}-\alpha \grad \mathbf{f}(\mathbf{x}^{r}) - \eta \mathbf{B}^{1 \over 2} \mathbf{v}^{r}  \\
		\mathbf{v}^{r+1} &= \prox_{\frac{\alpha}{\eta} \g^*} \left(\mathbf{v}^{r}+ \mathbf{B}^{1 \over 2} \bm{\Phi}^{r} \right) \\
		\mathbf{x}^{r+1}&=\mathbf{x}^{r}-\alpha \grad \mathbf{f}(\mathbf{x}^{r}) - \eta \mathbf{B}^{1 \over 2}\mathbf{v}^{r+1}. 
	\end{align}
\end{subequations}
Since $\g$ is the indicator function of zero, we have $\prox_{\frac{\alpha}{\eta} \g^*}(\z)=\z$; thus $\mathbf{v}^{r+1} = \mathbf{v}^{r}+ \mathbf{B}^{1 \over 2} \bm{\Phi}^{r}$. Moreover, observe that 
\begin{align}
	\mathbf{x}^{r+1}&=\mathbf{x}^{r}-\alpha \grad \mathbf{f}(\mathbf{x}^{r}) - \eta \mathbf{B}^{1 \over 2}\mathbf{v}^{r+1} \nonumber \\
	&=\mathbf{x}^{r}-\alpha \grad \mathbf{f}(\mathbf{x}^{r}) - \eta \mathbf{B}^{1 \over 2} \mathbf{v}^{r}  - \eta \mathbf{B}^{1 \over 2}(\mathbf{v}^{r+1} -  \mathbf{v}^{r} ) \nonumber \\
	&= \bm{\Phi}^{r} - \eta \B \bm{\Phi}^{r}= (\I-\eta \B) \bm{\Phi}^{r} \nonumber .
\end{align}
Therefore, \eqref{papc2} can be represented as:
\begin{subequations} 
	\begin{align*}
		\bm{\Phi}^{r} &= \mathbf{x}^{r}-\alpha \grad \mathbf{f}(\mathbf{x}^{r}) - \eta \mathbf{B}^{1 \over 2} \mathbf{v}^{r} \\
		\mathbf{v}^{r+1} &= \mathbf{v}^{r}+ \mathbf{B}^{1 \over 2} \bm{\Phi}^{r} \\
		\mathbf{x}^{r+1}&= (\I-\eta \B) \bm{\Phi}^{r}. 
	\end{align*}
\end{subequations}
Introducing $\y^{r}=\mathbf{B}^{1 \over 2} \mathbf{v}^{r}$ and using $\B=\I-\W$, the above updates can be rewritten as given in \eqref{nids}. Recall that when $\uptau=1$, we can remove the subscript $t$ from  $\bm{\Phi}^{r}_{t}$ and describe the updates \eqref{ed_updates_network}  as given in \eqref{uda_ed}.
When $\eta=\beta=1$,  the updates \eqref{nids} and \eqref{uda_ed} are identical. 
\end{proof}

\begin{remark}[\sc ED, NIDS, and  D$^2$ interpretations] \rm The above result demonstrates that \alg{LED} \eqref{ed_updates_network} can be interpreted as a locally updated variant of \alg{PDFP2O} \eqref{papc}. It also shows that 	\alg{ED}/\alg{D$^2$} \cite{yuan2019exactdiffI,tang2018d} and \alg{NIDS} \cite{li2017nids}  are  different representations of \alg{PDFP2O} \cite{chen2013aprimal} applied on formulation \eqref{reformulation}. \qd	
\end{remark}

\subsection{Relation with \alg{Scaffnew}}
The work \cite{mishchenko2022proxskip} studied a proximal skipping variant of  \alg{PDFP2O}. The  decentralized \alg{Scaffnew} method  is given by \cite[Alg.~5]{mishchenko2022proxskip}:
\begin{subequations} \label{scaffnew}
	\begin{align} \label{scaffnew_local}
			\bm{\Phi}^{r}=\x^{r}-\alpha \left(\nabla \f\left(\x^{r}\right)+  \z^r\right).
	\end{align}
Generate a random number $\theta_t \in\{0,1\}$  with  $\operatorname{Prob}\left(\theta_t=1\right)=p$ and update:
\begin{align} \label{scaffnew_rare_communication}
\begin{cases}
  \x^{r+1}=	(1-\frac{\alpha \zeta}{p}) \bm{\Phi}^{r} + \frac{\alpha \zeta}{p} \W \bm{\Phi}^{r},  \\
  \z^{r+1}=\z^r + \frac{p}{\alpha} (\bm{\Phi}^{r}- \x^{r+1}) =   \z^r + \zeta (\I-\W) \bm{\Phi}^{r}  & \text { if } \theta_t=1 \\
	  \x^{r+1}=	\bm{\Phi}^{r}, \\
	 \z^{r+1}= \z^r  & \text{otherwise},
\end{cases}
\end{align}
\end{subequations}
where $\alpha,\zeta$ are stepsize parameters.  Observe that \eqref{scaffnew} employs local updates \eqref{scaffnew_local} and communicates only with a small probability $p$  \eqref{scaffnew_rare_communication}. If we let $\y^r=(1/\zeta) \z^r$ and $p=1$ (communicate at each iteration) then \eqref{scaffnew}  reduces to
\begin{subequations} \label{scaffnew_p1}
	\begin{align} 
		\bm{\Phi}^{r}&=\x^{r}-\alpha \nabla \f\left(\x^{r}\right)-  \alpha \zeta \y^r \label{scaffnew_p1a} \\
		\x^{r+1} &=	 \left( (1-\alpha \zeta) \I +  \alpha \zeta \W \right) \bm{\Phi}^{r}, \label{scaffnew_p1b}  \\
		\y^{r+1} &= \y^r +  (\I-\W) \bm{\Phi}^{r}. \label{scaffnew_p1c}
	\end{align}
\end{subequations}
The update \eqref{scaffnew_p1} is the same as \alg{PDFP2O} \eqref{nids} when $\eta=\alpha \zeta$. Consequently, when $p=1$ and $\zeta=1/\alpha$  \eqref{scaffnew_p1} is exactly  \alg{LED-1} \eqref{uda_ed} when  $\beta=1$.

\begin{remark} \rm
	 \alg{LED} \eqref{ed_updates_network} employs a fixed number $\uptau$ of local updates between two communication rounds, whereas \alg{Scaffnew} uses {\em random} number of local updates between two communication rounds.   The use of random communication skipping or fixed local steps differs in analysis, however, in terms of performance they are strikingly similar with $1/p$ playing the role of $\uptau$.

	 The work \cite{mishchenko2022proxskip} analyzes \alg{Scaffnew} and shows that local steps can save communication when the network is well connected. We point out that the analysis techniques in \cite{mishchenko2022proxskip} do not show linear speedup and are only-suited for the probabilistic  implementation with strongly-convex costs.	The techniques we present in this work are distinct and applicable to the locally updated variant with a deterministic number of local updates \eqref{ed_updates_network} for both nonconvex and (strongly-)convex settings.  \qd
\end{remark}

\subsection{Relation with \alg{FedGATE}/\alg{VRL-SGD}}
The work \cite{haddadpour2021federated} introduced and analyzed a federated learning  algorithm (centralized method) named \alg{FedCOMGATE} that employs compression; without compression the method reduces to \alg{FedGATE} \cite[Alg.~3]{haddadpour2021federated}, which is a generalization of \alg{VRL-SGD} \cite{liang2019variance}. We will now show the relationship between  \alg{FedGATE}/\alg{VRL-SGD} and the centralized version of \alg{LED}. As a first step, we will represent \alg{FedGATE}/\alg{VRL-SGD} in a networked form.

 \alg{FedGATE} is described as follows \cite[Alg.~3]{haddadpour2021federated}:
For $r=0,1,2,\dots$,  set $\phi_{i,0}^r=x^{r}$, for $t=0,\dots,\uptau-1$: 
\begin{subequations}
	\begin{align}  
		\phi_{i,t+1}^r &=    \phi_{i,t}^r-\alpha  ( \grad F_i(\phi_{i,t}^r;\xi_{i,t}^r)   -    \delta_i^{r} ) , \quad t=0,\dots,\uptau-1 .
	\end{align}
	Update
	\begin{align}  
		\delta_i^{r+1} &= \delta_i^{r} - \frac{1}{\alpha \uptau} \left(  \phi_{i,\uptau}^{r} - \frac{1}{N} \sum_{j =1}^N  \phi_{j,\uptau}^{r}  \right) \\
			x^{r+1} &= x^{r} - \alpha \gamma \big(x^{r} - \frac{1}{N} \sum_{j =1}^N  \phi_{j,\uptau}^{r} \big),
	\end{align}
\end{subequations}
where $\gamma$ is a global stepsize parameter.  By letting $y_i^r=-\alpha \uptau \delta_i^r$ and employing the network notation defined in \eqref{notation:network_main} with $\x^r=\one \otimes x^r$, the method above can be rewritten as:
	\begin{subequations}  \label{fedgate_network}
	\begin{align} 
		\bm{\Phi}^{r}_{t+1} &=    \bm{\Phi}^{r}_{t}-\alpha  \grad \F(\bm{\Phi}^{r}_{t};\bxi_t^r)   -  \frac{1}{\uptau}  \y^{r},  \quad t=0,\dots,\uptau-1 .  
	\end{align}
	Then update
	\begin{align} 
		\x^{r+1} &= (1-\alpha \gamma) \x^{r} + \alpha \gamma (\tfrac{1}{N} \one \one\tran) \bm{\Phi}^r_{\uptau} \\
		\y^{r+1} &= \y^{r}+ ( \I -  \tfrac{1}{N} \one \one\tran  ) \bm{\Phi}^r_{\uptau}.
	\end{align}
\end{subequations}
It's now evident that when $\alpha \gamma=1$, the update \eqref{fedgate_network} aligns with \alg{LED} \eqref{ed_updates_network} when $\W=\tfrac{1}{N} \one \one\tran$ and $\beta=1/\uptau$. It's worth noting that the updates \eqref{fedgate_network} simplify to \alg{VRL-SGD} \cite{liang2019variance} when $\alpha \gamma=1$ \cite{haddadpour2021federated}. In essence, \alg{FedGATE} with $\alpha \gamma=1$ (or \alg{VRL-SGD}) corresponds to \alg{LED} in the fully connected network scenario. This also suggests that \alg{FedGATE} and \alg{VRL-SGD} are locally updated versions of \alg{PDFP2O} \eqref{papc} with $\B=\I-\tfrac{1}{N} \one \one\tran$.

\begin{remark}[\sc Equivalence] \rm
	All of the derivations in this section require appropriate stepsizes tuning and are based on the assumption that the graph is {\em static} (\ie, $W$ is constant).  When the stepsizes differ or the graph is dynamic (as in the local update variant), these various representation may not necessarily be equivalent.	
	We note that the analysis techniques from \cite{yuan2019exactdiffI,alghunaim2019decentralized} are particularly suited for static graphs and are limited to the strongly-convex case. Moreover, the techniques from \cite{liang2019variance,haddadpour2021federated} are tailored for centralized networks. In contrast, our analyses address the more challenging decentralized connections with local updates; thus, our techniques can be specialized for these methods. In fact, our analysis can provide tighter rates compared to those in \cite{liang2019variance,haddadpour2021federated}. Remark \ref{remark:centralized_case} explains how to adapt our techniques to the centralized scenario.	  	
\end{remark}
	
	\section{Convergence result} \label{sec_results}
In this section, we present our main convergence findings and discuss how they differ from previous results. Before proceeding, we will review the assumptions necessary for our results to hold, which are standard in the literature \cite{alghunaim2021unified,liang2019variance,haddadpour2021federated}.
	
	\subsection{Assumptions}
	\begin{assumption}[\sc \small Weight matrix] \label{assump:network} 
		The weight matrix $W$ is symmetric,  doubly stochastic, and primitive. Moreover, we assume that $W$ is positive definite. \qd
	\end{assumption}
	\noindent   Under Assumption \ref{assump:network}, the eigenvalues of   $W$, denoted by $\{\lambda_i\}_{i=1}^N$, are all strictly less than one (in magnitude for nonpositive definite $W$),  with the exception of a single eigenvalue at one, which we denote by $\lambda_1$. The network's mixing rate is defined as:
	\begin{align} \label{mixing_rate}
		\lambda \define  \left\|W-\tfrac{1}{N} \one \one\tran \right\| =\max_{i \in \{2,\ldots,N\}} |\lambda_i| <1.
	\end{align}
We remark that the positive definiteness assumption can be easily satisfied because, given a symmetric doubly stochastic matrix $\tilde{W}$, we can construct a positive definite weight matrix by $W=0.5(\tilde{W}+I)$.

	\begin{assumption}[\sc \small Bounded variance] \label{assump:noise} 
		Each stochastic gradient $\grad F_i(x_i^k;\xi_i^k)$ is  unbiased  with bounded variance:
		\begin{subequations} \label{noise_bound_eq}
			\begin{align}
				\textstyle \Ex_k \big[\grad F_i(x_i^k;\xi_i^k)-\grad f_i(x_i^k)\big] &=0, \label{noise_bound_eq_mean} \\
				\textstyle \Ex_k \|\grad F_i(x_i^k;\xi_i^k)-\grad f_i(x_i^k)\|^2  &\leq \sigma^2, \label{noise_bound_eq_variance}
			\end{align}
		\end{subequations}
		for some $\sigma^2 \geq 0$ where   $\Ex_{k}$ denotes the expectation conditioned on the all iterates up to iteration $k$,  $\{x_i^0,x_i^1,\dots,x_i^k\}$, for all $i$. Moreover, we assume that the random data  $\{\xi_{i,t}^k\}$ are independent from each other for all $\{i\}_{i=1}^N$ and $\{t \}$.   \qd
	\end{assumption}

	\begin{assumption}[\sc \small Objective function] \label{assump:smoothness} Each function $f_i: \real^m \rightarrow \real$ is $L$-smooth:
		\begin{align} \label{smooth_f_eq}
			\|\grad f_i(y) -\grad f_i(z)\| \leq L \|y -z\|, \quad \forall~y,z \in \real^m, 
		\end{align}
		for some $L>0$. Additionally, the aggregate function $f(x)=\frac{1}{N} \sum_{i=1}^N f_i(x)$ is bounded below, \ie, $f(x) \geq  f^\star > -\infty$ for every $x \in \real^m$, where $f^\star$ denotes the optimal value of $f$.  \qd
		\end{assumption}
		\noindent 	Under the aforementioned assumption, the aggregate function $f(x)=\frac{1}{N} \sum_{i=1}^N f_i(x)$ is also $L$-smooth.

	 Assumptions \ref{assump:network}--\ref{assump:smoothness} are sufficient to establish convergence under nonconvex settings. We will also study convergence under additional convexity assumption given below.    
	
	\begin{assumption}[\sc \small Convexity] \label{assump:cvx} Each function $f_i: \real^m \rightarrow \real$ is ($\mu$-strongly)convex
		for some $0 \leq \mu \leq L $. (When $\mu=0$, then the functions are simply convex.)  \qd
	\end{assumption}

	\subsection{Main results}
We are now ready to present our main findings. The convergence results for nonconvex and convex functions are presented in Theorems \ref{thm:noncvx} and \ref{thm:cvx}, respectively. The final convergence rates derived from these theorems are given in Corollaries \ref{coro:ratesdet} and \ref{coro:rates}. All proofs can be found in the appendices.

	\begin{theorem}[\sc \small Nonconvex convergence]  \label{thm:noncvx}
		Under Assumptions \ref{assump:network}--\ref{assump:smoothness}, and for sufficiently small constant stepsizes $\alpha$ and $\beta = 1/\uptau$, it holds that
	\begin{align} \label{eq:thm1:noncvx}
		\frac{1}{R}	 \sum_{r=0}^{R-1} \cE_r 		 
&\leq \cO \underbrace{\left( \frac{    f(\bar{x}^{0})-f^\star}{\alpha \uptau R}	  
	+ \frac{ \alpha^2 \uptau^2 L^2    \varsigma_0^2 }{ (1-\lambda)^2  R} \right)}_{\text{Deterministic part}} 
+  \cO \underbrace{\left( \frac{  \alpha  L \sigma^2 }{ N} 
	+ \frac{ \alpha^2 \uptau  L^2  \sigma^2}{1-\lambda}
	\right)}_{\text{Stochastic part}},
	\end{align}
	where  $\cE_r \define  \textstyle   \Ex  \| \grad f(\bar{x}^{r})  \|^2  +  \frac{1}{\uptau} \sum_{t=0}^{\uptau-1} \|  \tfrac{1}{N}  \sum_{i=1}^N \nabla f_i(\phi^r_{i,t})\|^2$ and  $\varsigma_0^2 \define \frac{1}{N} \sum_{i=1}^N \|\grad f_i(\bar{x}^0) -  \grad f(\bar{x}^0) \|^2$.
 \qd
	\end{theorem}

	\begin{theorem}[\sc \small Convex convergence]  \label{thm:cvx}
			Under Assumptions \ref{assump:network}--\ref{assump:cvx}, and  for sufficiently small constant stepsizes $\alpha$ and $\beta = 1/\uptau$, it holds that for $\mu=0$ (convex case)
			\begin{align}  \label{eq:thm2:cvx}
		\frac{1}{R}	 \sum_{r=0}^{R-1} \cE_r
	&\leq  \cO \underbrace{\left( \frac{    \|\bar{x}^{0}-x^\star \|^2}{\alpha \uptau R}	  
		+ \frac{ \alpha^2 \uptau^2 L^2    \varsigma_0^2 }{ (1-\lambda)^2 R} \right)}_{\text{Deterministic part}}  
	+ 	  \cO \underbrace{\left(  \frac{  \alpha   \sigma^2 }{ N}
		+\frac{ \alpha^2 \uptau  L  \sigma^2}{1-\lambda}  \right)}_{\text{Stochastic part}},
		\end{align}
	where    $\cE_r \define  \Ex [f(\bar{x}^r) - f(x^\star)]$,  $\bar{x}^{0} \define (1/N) \sum_{i=1}^N x_i^0$, and  $\varsigma_0^2 \define \frac{1}{N} \sum_{i=1}^N \|\grad f_i(\bar{x}^0) -  \grad f(\bar{x}^0) \|^2$. Moreover, if $\mu>0$ (strongly-convex case) then
				\begin{align}   \label{eq:thm2:strongcvx}
		\Ex  \|\bar{x}^{r}-x^\star\|^2   
& \leq  \underbrace{\left(1-\tfrac{\alpha \uptau \mu}{4} \right)^r a_0}_{\text{Deterministic part}} 
+ \cO \underbrace{\left(   \frac{ \alpha \sigma^2 }{\mu N} + 
	\frac{     \alpha^2 \uptau L  \sigma^2  }{  \mu (1-\lambda)}\right)}_{\text{Stochastic part}},
		\end{align}
	where $\rho \define 1-\lambda$ and  $a_0$ is a constant that depends on the initialization.
		\qd
	\end{theorem}
For the nonconvex case and when $\sigma >0$, Theorem \ref{thm:noncvx} shows that the algorithm converges to a radius around some stationary point, which can be controlled by the stepsize $\alpha$. Without any additional assumptions, a stationary point is the best guarantee possible and is a satisfactory criterion to measure the performance of distributed methods with nonconvex objectives \cite{karimireddy2019scaffold,koloskova2020unified,alghunaim2021unified}. For the convex case, Theorem \ref{thm:cvx} shows that the algorithm converges around some optimal solution controlled by the stepsize $\alpha$. 
\begin{corollary}[\textsc{Exact Convergence in the Deterministic Case}] \label{coro:ratesdet}
	Suppose the conditions of Theorem \ref{thm:noncvx} are met. Then, in the noiseless deterministic case where $\sigma=0$, substituting this into \eqref{eq:thm1:noncvx} gives the nonconvex rate:
	\[
	\cO \left( \frac{1}{R} + \frac{\varsigma_0^2}{(1-\lambda)^2 R} \right).
	\]
	Therefore, \alg{LED} converges exactly in the deterministic case with a rate of $1/R$. Similar results can be obtained for the convex cases. \qd
\end{corollary}
For the stochastic case, the stepsize $\alpha$  is tuned based on $R$ to obtain the following result.
\begin{corollary}[\sc \small Convergence rates in the stochastic case]
	\label{coro:rates}
	For the nonconvex, convex, and strongly convex cases, there exists a stepsize $\alpha$ that yields the following rates.
	\begin{itemize}
		\item 	Nonconvex rate:		
		\begin{align} \label{rate_noncvx}
			&\frac{1}{R}	 \sum_{r=0}^{R-1} \cE_r 
			\leq \cO\left(\frac{   \sigma}{N \uptau R}\right)^{\frac{1}{2}}
			+ \cO \left(\frac{1}{(1-\lambda)^{1/3}}
			\left(\frac{    \sigma}{ \sqrt{\uptau} R}\right)^{\frac{2}{3}} \right)
			+ \cO\left( \frac{  \tfrac{1}{1-\lambda}+     \varsigma_0^2  }{R} \right),
		\end{align}
		where  $\cE_r \define  \textstyle   \Ex  \| \grad f(\bar{x}^{r})  \|^2  +  \frac{1}{\uptau} \sum\limits_{t} \|  \tfrac{1}{N}  \sum\limits_{i} \nabla f_i(\phi^r_{i,t})\|^2$.
		\item Convex rate:
		\begin{align} \label{rate_cvx}
			&	\frac{1}{R}	 \sum_{r=0}^{R-1} \cE_r
			\leq \cO\left(\frac{   \sigma}{N \uptau R}\right)^{\frac{1}{2}}
			+ \cO \left(\frac{1}{(1-\lambda)^{1/3}}
			\left(\frac{    \sigma}{ \sqrt{\uptau} R}\right)^{\frac{2}{3}} \right)
			+ \cO\left( \frac{  \frac{1}{1-\lambda}+      \varsigma_0^2  }{R} \right),
		\end{align}
		where $\cE_r \define  \Ex [f(\bar{x}^r) - f(x^\star)]$.
		
		\item Strongly-convex rate:
		\begin{align}  \label{rate_strongcvx}
			& \Ex  \|\bar{x}^{R}-x^\star\|^2  
			\leq \tilde{\mathcal{O}}\left(\frac{\sigma^2}{ \uptau N R}\right) 
			+ \tilde{\mathcal{O}}\left(\frac{\sigma^2  }{ (1-\lambda) \uptau  R^2}\right) 
			+	\tilde{\mathcal{O}}\left(   \exp \left[- (1-\lambda) R\right] \left( \|\bar{x}^{0}-x^\star\|^2 +    \varsigma_0   \right)
			\right),
		\end{align}
		where the notation $\tilde{\mathcal{O}}(\cdot)$ ignores logarithmic factors.
	\end{itemize}
	Here, $\bar{x}^{0} \define (1/N) \sum_{i=1}^N x_i^0$ and  $\varsigma_0^2 \define \frac{1}{N} \sum_{i=1}^N \|\grad f_i(\bar{x}^0) -  \grad f(\bar{x}^0) \|^2$.
\end{corollary}
\begin{remark}[\sc \small Practical stepsize]
	\rm 
	The stepsize yielding the results in Corollary \ref{coro:rates} is intricate and not very practical. It's chosen mainly for theoretical reasons, as it provides the optimal convergence rate based on our bounds. We opted for this choice to ensure fair comparisons with \alg{SCAFFOLD}, \alg{Local-DSGD}, and \alg{K-GT}, which also tune the stepsize in a similar manner.
	
	In practice, we  typically set $\alpha=1/\sqrt{R}$ for both the nonconvex and convex cases, and $\alpha=1/R$ for the strongly-convex case. For instance, if we plug in  $\alpha=\frac{1}{L+\sqrt{\uptau R  /N}}$ into \eqref{eq:thm1:noncvx}, we obtain the rate
	\begin{align}
		\frac{1}{R}	 \sum_{r=0}^{R-1} \cE_r 		 
		&\leq \cO \left( \tfrac{  \sigma^2}{N \uptau R} \right)^{1/2}	  
		+ \cO \left(   
		\tfrac{ N \sigma^2}{(1-\lambda) \uptau R}
		+ \tfrac{ N    \varsigma_0^2 }{ (1-\lambda)^2  R^2}  	\right).
	\end{align}
	For large $R$, the above rate is also $1/\sqrt{N \uptau R}$ consistent with \eqref{rate_noncvx}.   \qd
\end{remark}
	\begin{corollary}[\sc \small Centralized rate]
		Suppose Assumptions \ref{assump:noise}--\ref{assump:smoothness} hold. Then, with the appropriate parameters, the \alg{LED} under the server-workers scenario, as listed in Alg.~\ref{alg:fed-ed}, converges at the rate
				\begin{align} 
			\frac{1}{R}	 \sum_{r=0}^{R-1} \cE_r  		 
			&\leq \cO \left(	\frac{ \sigma }{ \sqrt{N R \uptau}} +  \frac{   f(\bar{x}^{0}) - f^\star + \varsigma_{0}^2}{  R }
			\right)
		\end{align}
	    for the nonconvex case. Here,  $\cE_r \define  \textstyle   \Ex  \| \grad f(\bar{x}^{r})  \|^2  +  \frac{1}{\uptau} \sum\limits_{t} \|  \tfrac{1}{N}  \sum\limits_{i} \nabla f_i(\phi^r_{i,t})\|^2$ and  $\varsigma_0^2 \define \frac{1}{N} \sum_{i=1}^N \|\grad f_i(\bar{x}^0) -  \grad f(\bar{x}^0) \|^2$. (The proof is given in Appendix \ref{app_server_proof}.)
	\end{corollary}	
\noindent \textbf{Discussion of our results.} We will discuss our results for the nonconvex case; similar arguments apply for the convex case. For large \(R\), the higher-order terms from \eqref{rate_noncvx} can be neglected, and the dominant part becomes on the order of \(\left(\frac{\sigma}{N \uptau R}\right)^{\frac{1}{2}}\). This suggests that to achieve an \(\epsilon\) accuracy, we need \(R \geq \frac{1}{N \uptau \epsilon^2}\). In this scenario, the advantages of \(N\) and \(\uptau\) are evident. Furthermore, the number of communication rounds required to achieve \(\epsilon\) accuracy decreases linearly with \(N\); this characteristic is termed linear speed-up \cite{lian2017can}. When \(R\) is not sufficiently large, then the higher-order terms, specifically 
$
\left(\frac{1}{(1-\lambda)^{1/3}}(\frac{\sigma}{\sqrt{\uptau} R})^{\frac{2}{3}}\right)$
and 
$
\left( \frac{1/(1-\lambda) + \varsigma_0^2}{R} \right)$,
cannot be ignored as they may slow down the convergence. For instance, when the network is sparse, the quantity \(1-\lambda\) can be extremely small as \(\lambda \approx 0\); in this situation, the rate becomes slower, as will be discussed in Section \ref{sec-sim}. Table \ref{table_rates} lists the convergence rate of \alg{LED} compared to state-of-the-art results, in terms of the number of communication rounds needed to achieve \(\epsilon\) accuracy.

Compared to our results, observe that \alg{Local-DSGD} \cite{koloskova2020unified} introduces an additional term 
 $
\frac{\varsigma}{1-\lambda} \frac{1}{\epsilon^{3 / 2}}
$ (or $\frac{\varsigma}{1-\lambda} \frac{1}{\epsilon^{1 / 2}}$ for the strongly convex case)
where \(\varsigma\) represents the local functions heterogeneity constant, such that 
$
\frac{1}{N} \sum_{i=1}^N \|\nabla f_i(x) - \nabla f(x)\|^2 \leq \varsigma^2$.
This additional term result in suboptimal convergence rates, even in deterministic (\(\sigma=0\)) scenarios, causing a significant slow down in convergence (refer to the Simulation section for more details).  When compared to \alg{$K$-GT} \cite{liu2023decentralized}, the second and third terms are
$
\left( \frac{\sigma}{(1-\lambda)^2 \sqrt{\uptau}} \right) \frac{1}{\epsilon^{3 / 2}} + \frac{1}{(1-\lambda)^2 \epsilon}$, 
whereas in our rate for \alg{LED}, these are
$
( \frac{\sigma}{\sqrt{(1-\lambda) \uptau}} ) \frac{1}{\epsilon^{3 / 2}} + \frac{1}{(1-\lambda) \epsilon}$.
The factor \(1-\lambda\) becomes notably small for sparse networks, which indicates that the performance of \alg{$K$-GT} may degrade significantly relative to \alg{LED} in sparsely connected networks. Note that considering a single local step, with \(\uptau=1\), our rates align with the best-established decentralized rates \cite{alghunaim2021unified}.

The table also enumerates the rate of the centralized method \alg{SCAFFOLD} as cited in \cite{karimireddy2019scaffold}. For centralized networks defined by \(W=(1/N)\one\tran \one\), our rate matches that of \alg{SCAFFOLD} \cite{karimireddy2019scaffold}. In fact, our rate is more refined than \alg{VRL-SGD}, presented in \cite{liang2019variance}. Specifically, our bound allows for setting 
$
\uptau = \cO\left(\frac{1}{N \epsilon}\right)$. In this scenario, the number of communication rounds necessary to achieve $\epsilon$ precision is given by  $
R = \cO\left(\frac{1}{\epsilon}\right)$.
In contrast, \cite{liang2019variance} suggests that achieving \(\epsilon\) precision requires a communication round count worse by a factor of \(N\), specifically 
$
R = \cO\left(\frac{N}{\epsilon}\right)$
(as seen in \cite[Table 6]{haddadpour2021federated}). Moreover, in the convex scenario, our rate $
R = \cO\left(\frac{1}{\epsilon}\right)$
is sharper than that of \alg{FedGate} from \cite{haddadpour2021federated}, which is 
$
R = \cO\left(\frac{1}{\epsilon} \log\left(\frac{1}{\epsilon}\right)\right)
$
(also observed in \cite[Table 6]{haddadpour2021federated}). This implies that when adapting our analysis to centralized networks, we provide new and improved analyses for the methods \alg{VRL-SGD} \cite{liang2019variance} and \alg{FedGATE} \cite{haddadpour2021federated}.

	\begin{table*}[t]  \small
	\renewcommand{\arraystretch}{1.8}
	\begin{center}
		\caption{\small Number of communication rounds needed to achieve $\epsilon$ accuracy with $\uptau$ local updates. The rates for \text{\alg{SCAFFOLD} \cite{karimireddy2019scaffold}} are tailored to the case of full participation. In this table, $\rho=1-\lambda$, where $\lambda$ is the mixing rate of the network (for fully connected network $\rho=1$), $\sigma$ is the stochastic gradient noise, and $\varsigma$ is the functions heterogeneity bound such that $(1/N) \sum_{i=1}^N \|\grad f_i(x) - \grad f(x)\|^2 \leq \varsigma^2$. The convergence rate improves over \alg{Local-DSGD} and \alg{GT} methods. }
		\begin{tabular}{llll}
			\hline \textsc{Method} & \textsc{Nonconvex} & \textsc{Strongly Convex} & \textsc{Communication per-link} \\
			\hline \multicolumn{4}{c}{\textbf{Decentralized Setup}} \\
			\hline \text{\alg{Local-DSGD} \cite{koloskova2020unified}} & $\frac{\sigma^2}{N \uptau \epsilon^2}+ \left( \frac{\sigma}{\sqrt{\rho \uptau}} +\frac{\varsigma }{\rho} \right) \frac{1}{\epsilon^{3 / 2}}+\frac{1 }{\rho \epsilon}$ & $ \frac{\sigma^2}{\mu \uptau N \epsilon} +   \left( \frac{\sigma}{\sqrt{\rho \uptau}} +\frac{\varsigma }{\rho} \right) \frac{1}{\sqrt{\epsilon}}+\frac{ 1}{ \rho} \log \frac{1}{\epsilon}$ & One vector
			\\
			\hline \text{\alg{$K$-GT} \cite{liu2023decentralized}} & $\frac{\sigma^2}{N \uptau \epsilon^2}+ \left( \frac{\sigma}{\rho^2 \sqrt{ \uptau}}  \right) \frac{1}{\epsilon^{3 / 2}}+\frac{1 }{\rho^2 \epsilon}$ & N/A  & Two vectors
			\\
			\hline
			\rowcolor[gray]{0.9}  \alg{LED} \textbf{[this work]} &  $\frac{\sigma^2}{N \uptau \epsilon^2}+ \left( \frac{\sigma}{\sqrt{\rho \uptau}}  \right) \frac{1}{\epsilon^{3 / 2}} + \frac{1 }{\rho \epsilon}$ & $ \frac{\sigma^2}{\mu \uptau N \epsilon}+   \left( \frac{\sigma}{\sqrt{\rho \uptau}}  \right) \frac{1}{\sqrt{\epsilon}}+\frac{ 1}{ \rho} \log \frac{1}{\epsilon}$  & One vector \\
					\hline \multicolumn{4}{c}{\textbf{Centralized Setup}} \\
			\hline \text{\alg{SCAFFOLD} \cite{karimireddy2019scaffold}} & $
			\frac{\sigma^2}{N \uptau \epsilon^2}+\frac{1}{\epsilon}
			$ & $
			\frac{\sigma^2}{ N \uptau \epsilon}+ \log(\frac{1}{\epsilon})
			$ & Two vectors
			\\
			\hline
			\rowcolor[gray]{0.9}  \alg{LED} \textbf{[this work]} &  $
			\frac{\sigma^2}{N \uptau \epsilon^2}+\frac{1}{\epsilon}
			$ & $
			\frac{\sigma^2}{ N \uptau \epsilon}+ \log(\frac{1}{\epsilon})
			$ & One vector  \\
			\hline
		\end{tabular}
		\label{table_rates}
	\end{center}
\end{table*}
\section{Numerical simulations} \label{sec-sim}
In this section, we use numerical simulations to demonstrate and validate our findings on the logistic regression problem with a nonconvex regularizer  given by 
\[
\min_{x \in \real^m} \quad \frac{1}{N}\sum_{i=1}^N f_i(x) + \eta \, r(x),
\]
where
$f_i(x) = \frac{1}{S}\sum_{s=1}^S \ln\big(1 + \exp(-y_{i,s}h_{i,s}\tran x)\big)$ and $r(x) = \sum_{j=1}^m \frac{x(j)^2}{ 1 + x(j)^2}$.
In this problem, $x=\col\{x(j)\}_{j=1}^m \in \real^m$ is the unknown variable to be optimized. The training dataset held by node $i$ is represented as $\{h_{i,s}, y_{i,s}\}_{s=1}^S$, where $h_{i,s}\in \mathbb{R}^m$ is a feature vector and $y_{i,s} \in \{-1,+1\}$ is the corresponding label. The regularization $r(x)$ is a nonconvex and smooth function, and the regularization parameter $\eta > 0$ controls the influence of $r(x)$.

\noindent \textbf{Experimental settings.}   We generate local vectors according to $u^o_{i} = u^o + v_i$ where $u^o\sim \cN(0, \sigma_u^2  I_m)$ is a randomly generated vector, and $v_i \sim \cN(0, \sigma^2_h I_m)$. Given $u^o_i$,  the local feature vectors are generated as $h_{i,s} \sim \cN(0, 25 I_m)$ and the corresponding label $y_{i,s}$ is generated as follows: We first generate a random variable $z_{i,s} \sim \cU(0,1)$, then  if $z_{i,s} \le 1 + \exp(- h_{i,s}\tran u_i^o)$, we set $y_{i,s} = 1$; otherwise $y_{i,s} = -1$.  The stochastic gradients are generated as: $\nabla F_i(x) = {\nabla f}_i(x) + w_i$ where $w_i\sim \cN(0, \sigma^2 I_m)$.  All stochastic results are averaged over $100$ runs. In our simulations, we set  $N=15$, $m=5$, $S=1000$, $\eta = 0.01$, $\sigma_u=6$, $\sigma_h=2$, and $\sigma=10^{-3}$. The error criterion for all results is $\Ex \|\nabla f(\bar{x}^r)\|^2$ where $\bar{x}^r=\frac{1}{N}\sum_{i=1}^N x_i^r$.

\noindent \textbf{Simulation results.} Figure \ref{fig:decentalized_sim}  compares our \alg{LED} method with the decentralized methods \alg{L-DSGD} \cite{wang2021cooperative,koloskova2020unified}, \alg{LGT}   \cite{nguyen2022performance}, and	\alg{$K$-GT}   \cite{liu2023decentralized} for different local steps $\uptau=1$, $\uptau=5$, and $\uptau=10$.  In order to fairly compare the convergence rates of these methods, we individually tune the parameters of each algorithm so that each method reaches a predetermined error of 
$10^{-4}$
as quickly as possible. We used a ring (cycle) network and the weight matrix was generated using the Metropolis rule \cite{sayed2014nowbook} with $\lambda \approx 0.943$. We observe that \alg{LED}  outperforms all the other methods as we increase the number of local steps (rightmost plot). \alg{L-DSGD}  performs poorly because it cannot handle the local functions heterogeneity across the nodes. It's worth noting that increasing the number of local steps reduces the communication required to achieve the same level of accuracy.

Figure \ref{fig:centalized_sim}  shows the results against the centralized methods \alg{SCAFFOLD} \cite{karimireddy2019scaffold} and \alg{Local-SGD}; in this case, the network is fully connected $W=(1/N) \one \one\tran$.   We observe that both \alg{LED} and \alg{SCAFFOLD} perform similarly. Furthermore,  the performance of \alg{Local-DSGD} degrades as the number of local updates increases, as expected.  It's worth noting that  in these figures, the horizontal-axis refers to the number of communication rounds. For \alg{GT} and \alg{SCAFFOLD}, each agent must communicate two vectors to its neighbors per communication round. In contrast, \alg{LED} and \alg{Local-DSGD} requires the communication of only one vector.

	\noindent \textbf{Is local steps always beneficial?}
From Figure \ref{fig:decentalized_sim}, it can be observed that we can save on communication when we increase the number of local steps. However, these results hold for the stochastic case, and the significant benefits primarily arise from using more data per communication (similar to batch training). To further test the benefits of local steps, we consider the deterministic case, in which $\sigma=0$. Figures \ref{fig:led_sim_het}--\ref{fig:led_sim_iid} depict the results of \alg{LED} with different local steps and various topologies under both heterogeneous and similar data regimes. When the data is heterogeneous, Figure \ref{fig:led_sim_het} shows that the benefits of local steps are only visible in the fully connected network. Conversely, when the data is similar across nodes (minimal heterogeneity), there is a significant benefit to local steps, as shown in Figure \ref{fig:led_sim_iid}. However, for a sparse network (ring topology), we don't observe any noticeable advantage. These findings suggest that local steps can be beneficial when the network is well-connected and/or when data is consistent across the network.This can be explained by our theoretical results as follows. Assuming the network is well-connected ($\rho \approx 1$), we can disregard the term $1/\rho$ from the rate expression \eqref{rate_noncvx}. The dominant part of the rate then simplifies to
$
\cO(\frac{   \sigma}{N \uptau R})^{\frac{1}{2}}
+ \cO( \frac{     \varsigma_0^2  }{R} )$.
In this scenario, when data heterogeneity is relatively small (e.g., $ \varsigma_0^2=\cO(1/\sqrt{\uptau})$), the benefits of $\uptau$ become evident. Conversely, if the network is sparse ($\rho \approx 0$), then $1/\rho$ can be quite large. In such a case, the dominant terms in the rate would be
$
\cO(\frac{   \sigma}{N \uptau R})^{\frac{1}{2}}
+ \cO (\frac{1}{\rho^{1/3}}
(\frac{    \sigma}{ \sqrt{\uptau} R})^{\frac{2}{3}} )
+ \cO( \frac{       \varsigma_0^2  }{\rho R} )$.
Notice that $1/\rho$ adversely affects the higher-order terms, thus slowing down convergence.

				\begin{figure}[t]
		\begin{center}
			\includegraphics[scale=0.33]{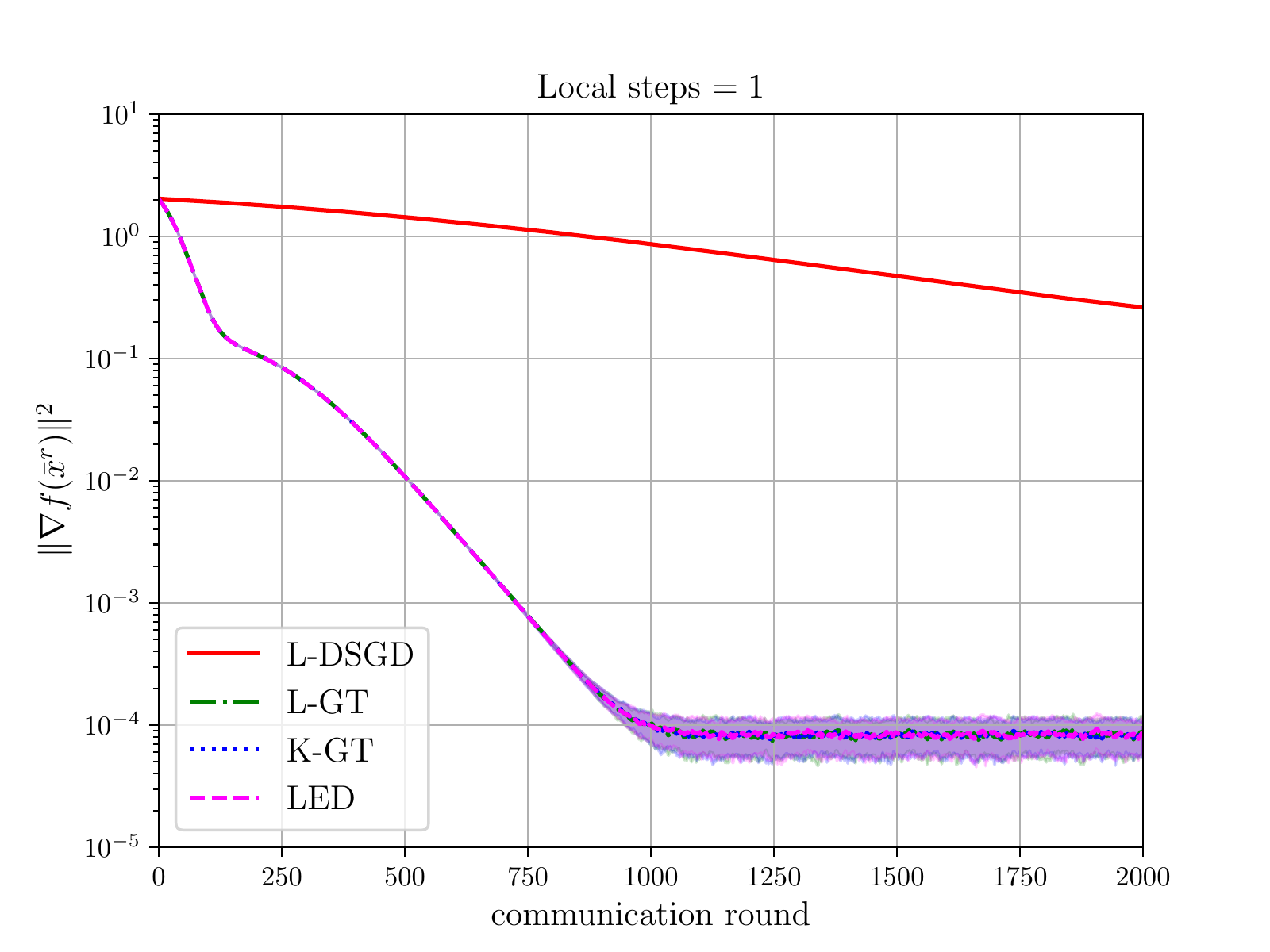}
			\includegraphics[scale=0.33]{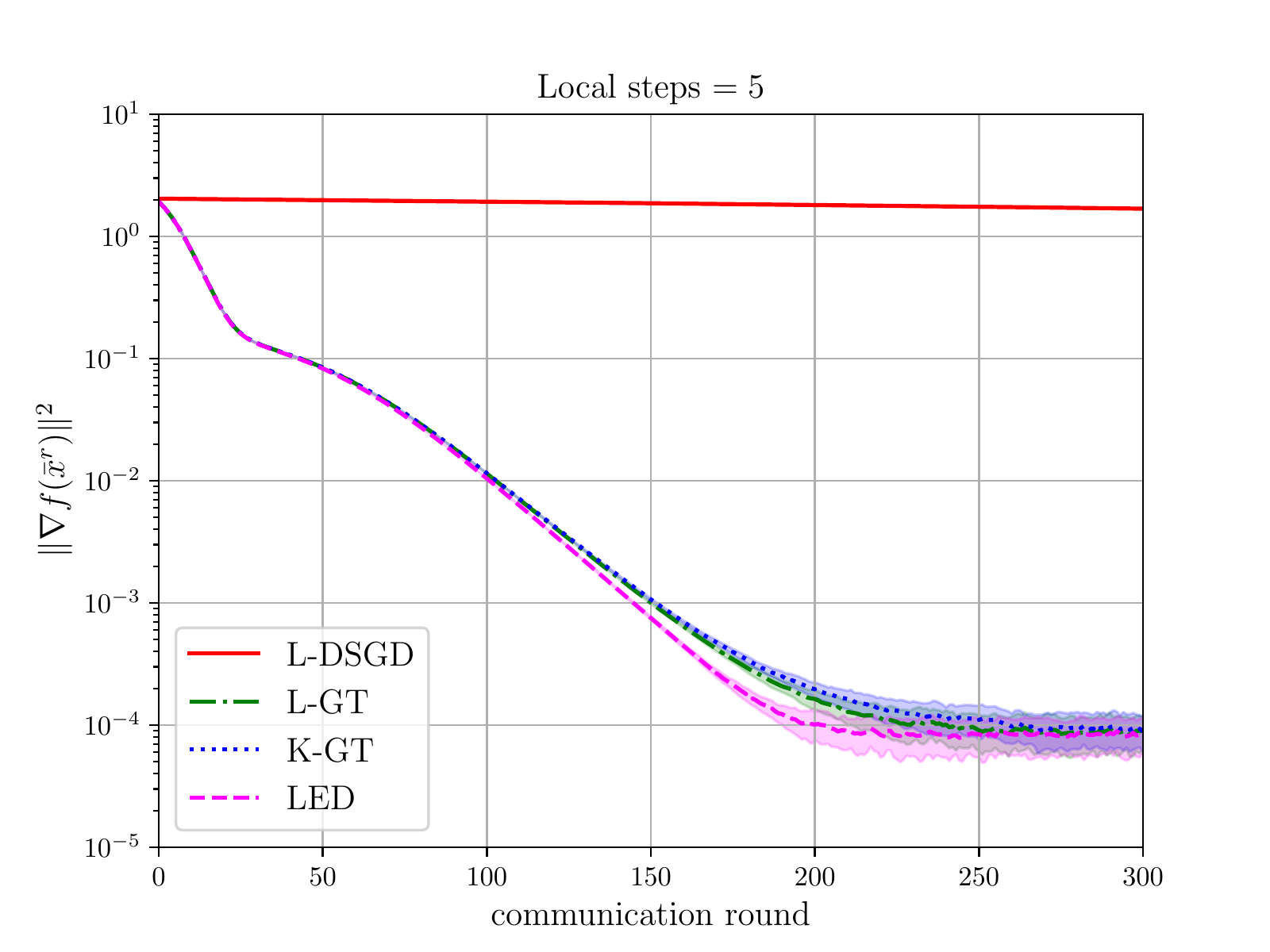}
			\includegraphics[scale=0.33]{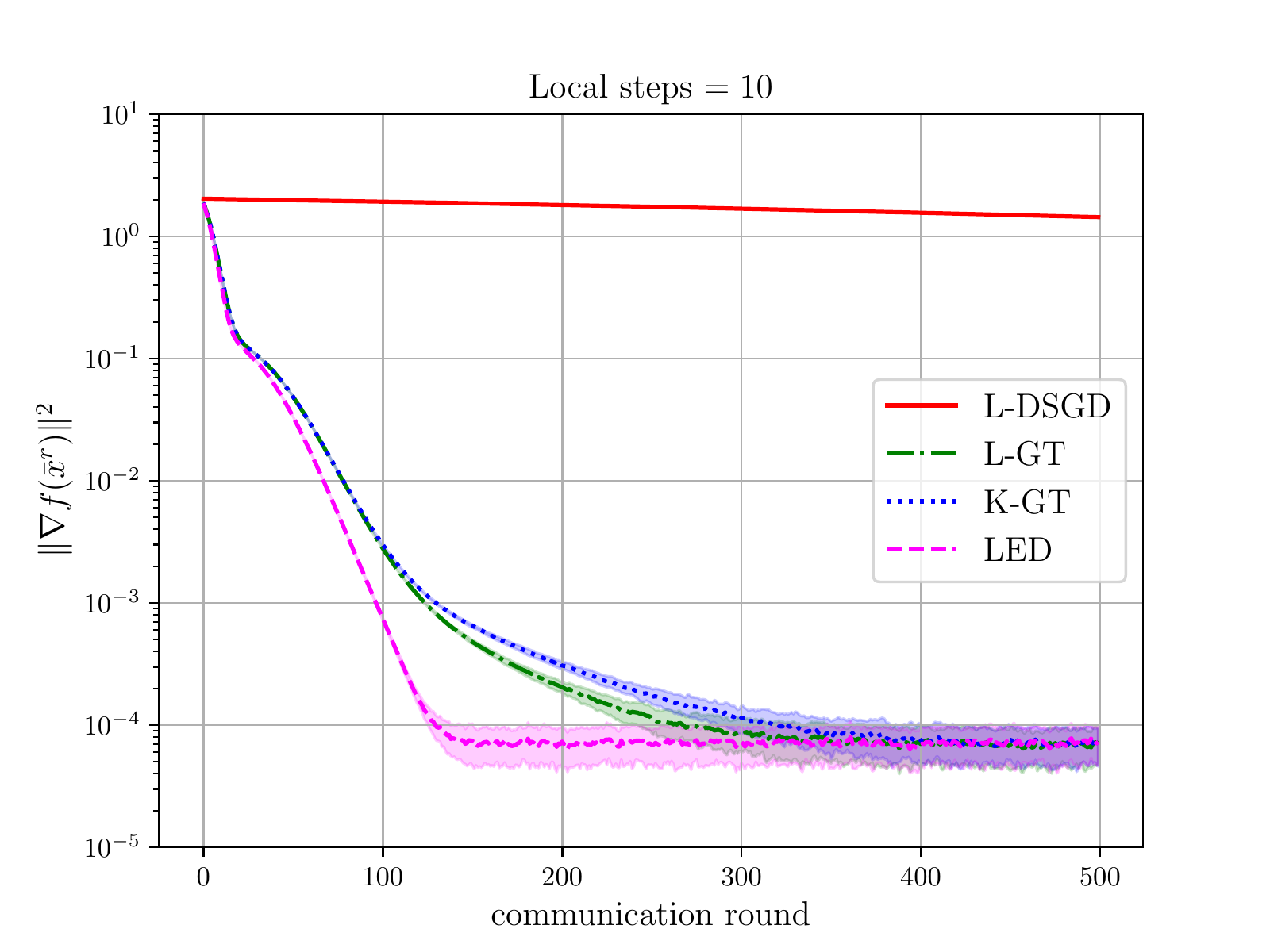}
		\end{center}
		\caption{Simulation results for decentralized methods to achieve an approximate error of $10^{-4}$ under various local steps (\alg{L-DSGD} \cite{wang2021cooperative,koloskova2020unified}, \alg{LGT}   \cite{nguyen2022performance}, and	\alg{$K$-GT}   \cite{liu2023decentralized}). We emphasize that \alg{LED} has half the communication cost of  \alg{GT} methods per communication round. }
		\label{fig:decentalized_sim}
	\end{figure}
	\begin{figure}[t]
		\begin{center}
			\includegraphics[scale=0.33]{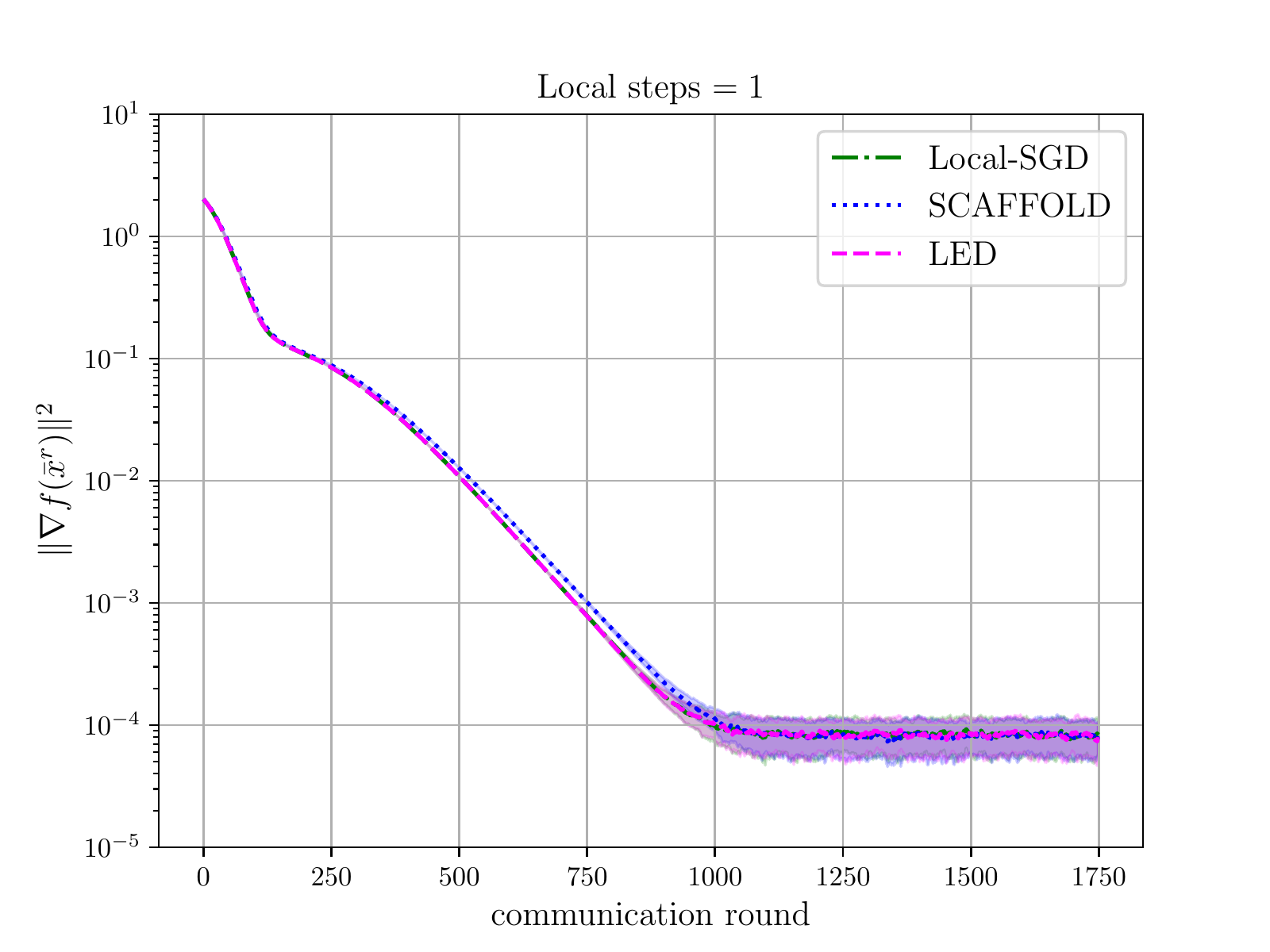}
			\includegraphics[scale=0.33]{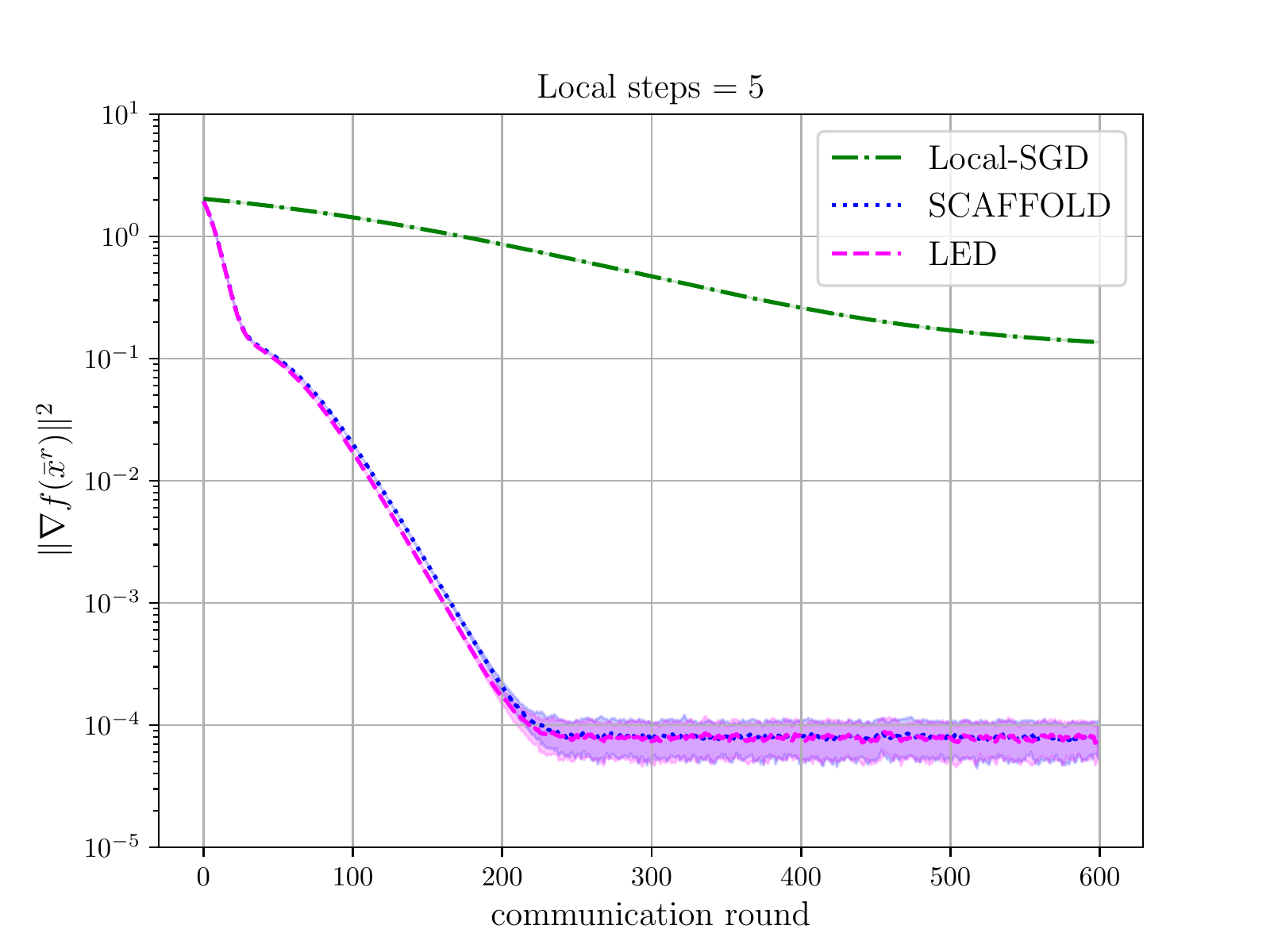}
			\includegraphics[scale=0.33]{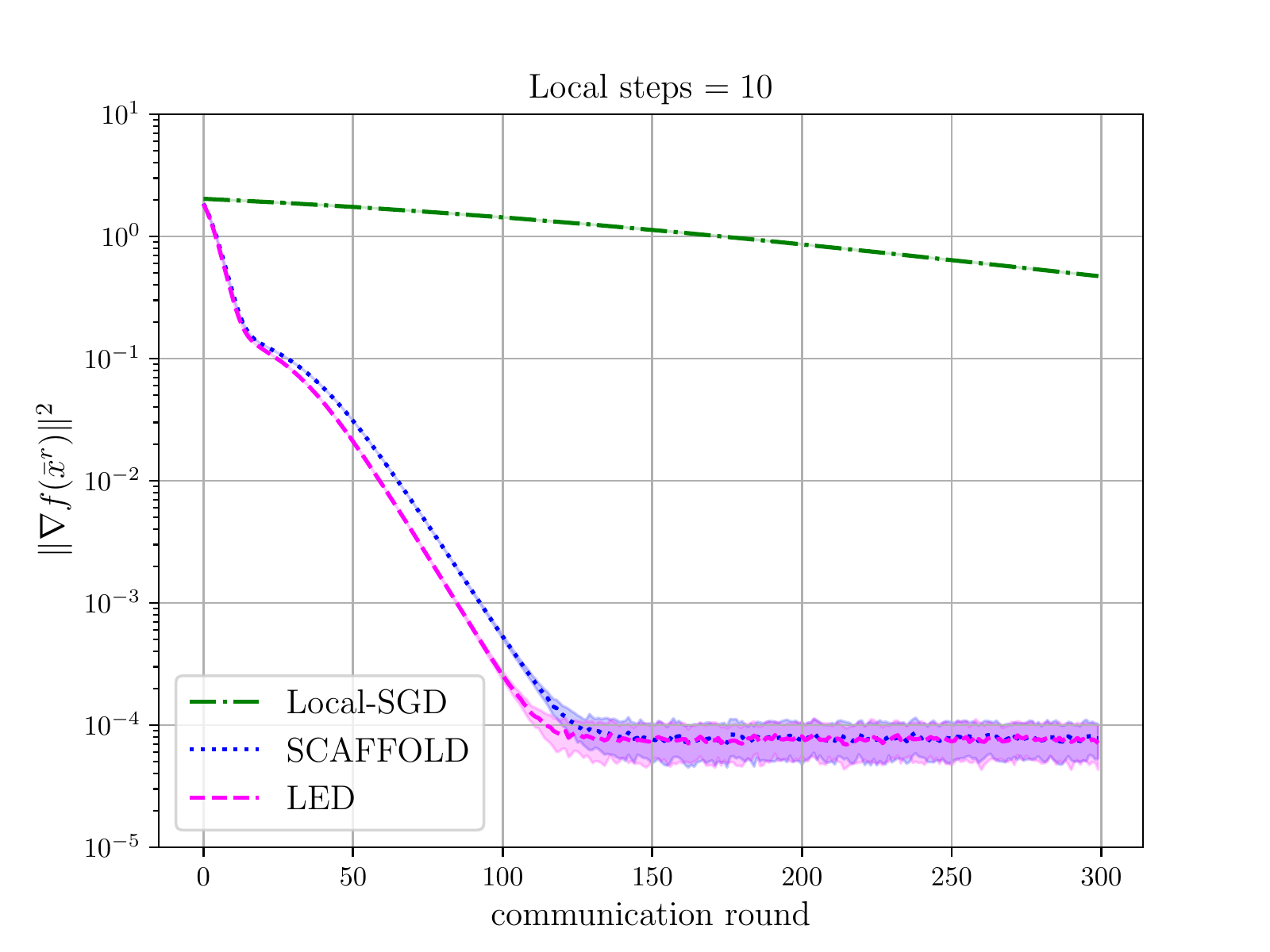}
		\end{center}
		\caption{Simulation results for centralized methods to achieve an approximate error of $10^{-4}$ under various local steps (\alg{SCAFFOLD} \cite{karimireddy2019scaffold}).  We emphasize  that \alg{LED} has half the communication cost of  \alg{SCAFFOLD}  per communication round. }
		\label{fig:centalized_sim}
	\end{figure}
	 
	\begin{figure}[t]
		\begin{subfigure}{\columnwidth}
		\begin{center}
			\includegraphics[scale=0.33]{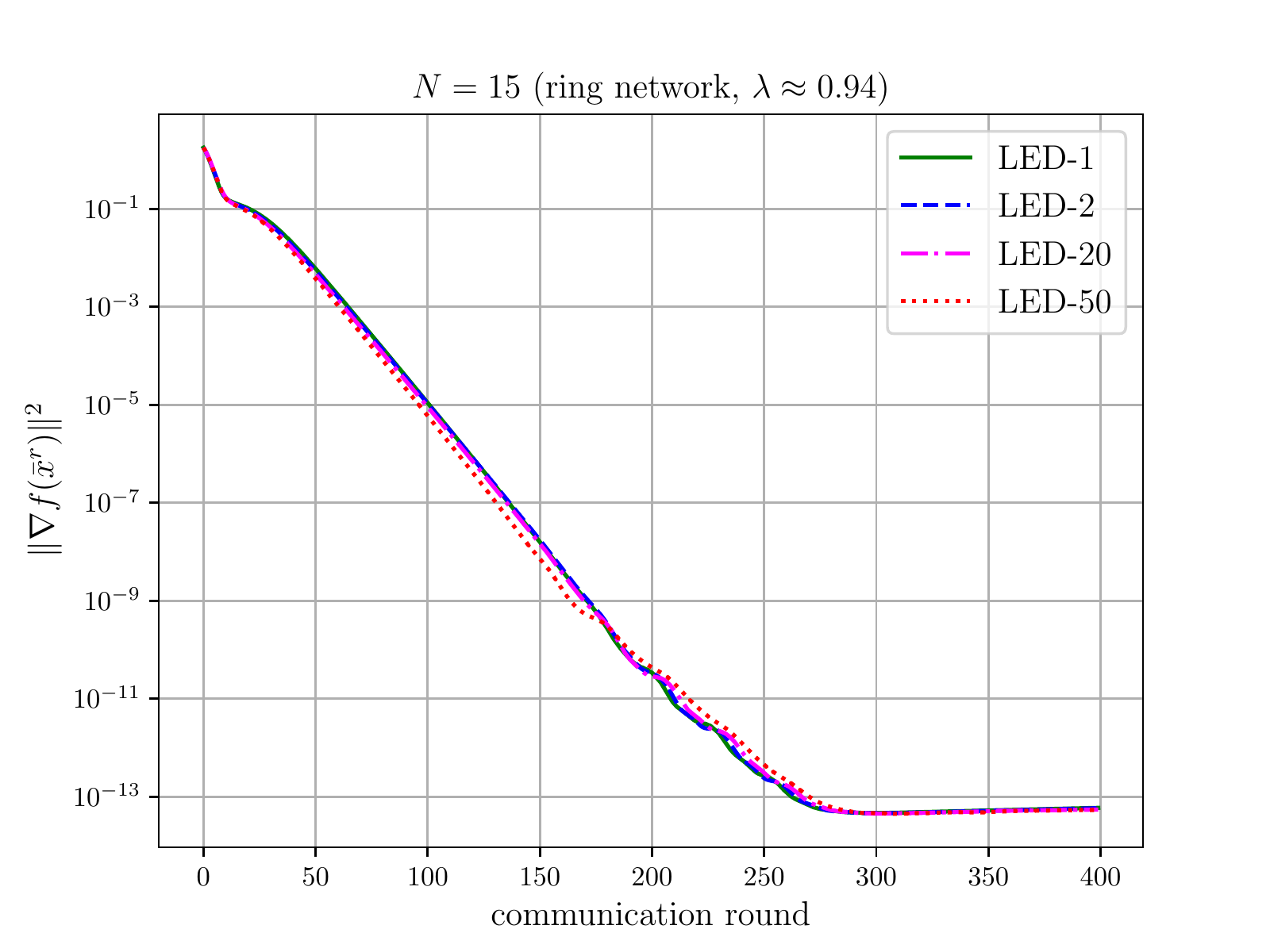}
			\includegraphics[scale=0.33]{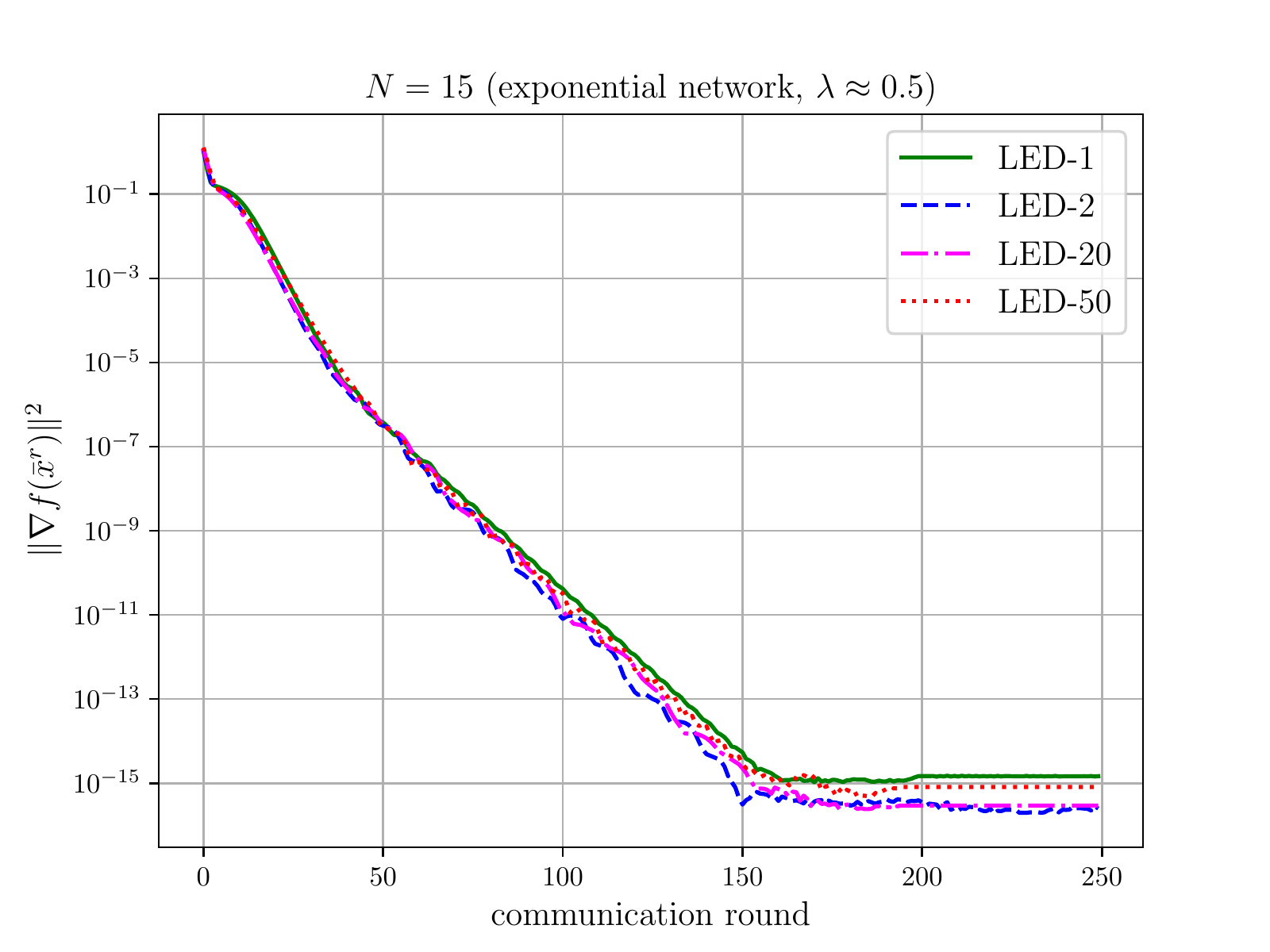}
			\includegraphics[scale=0.33]{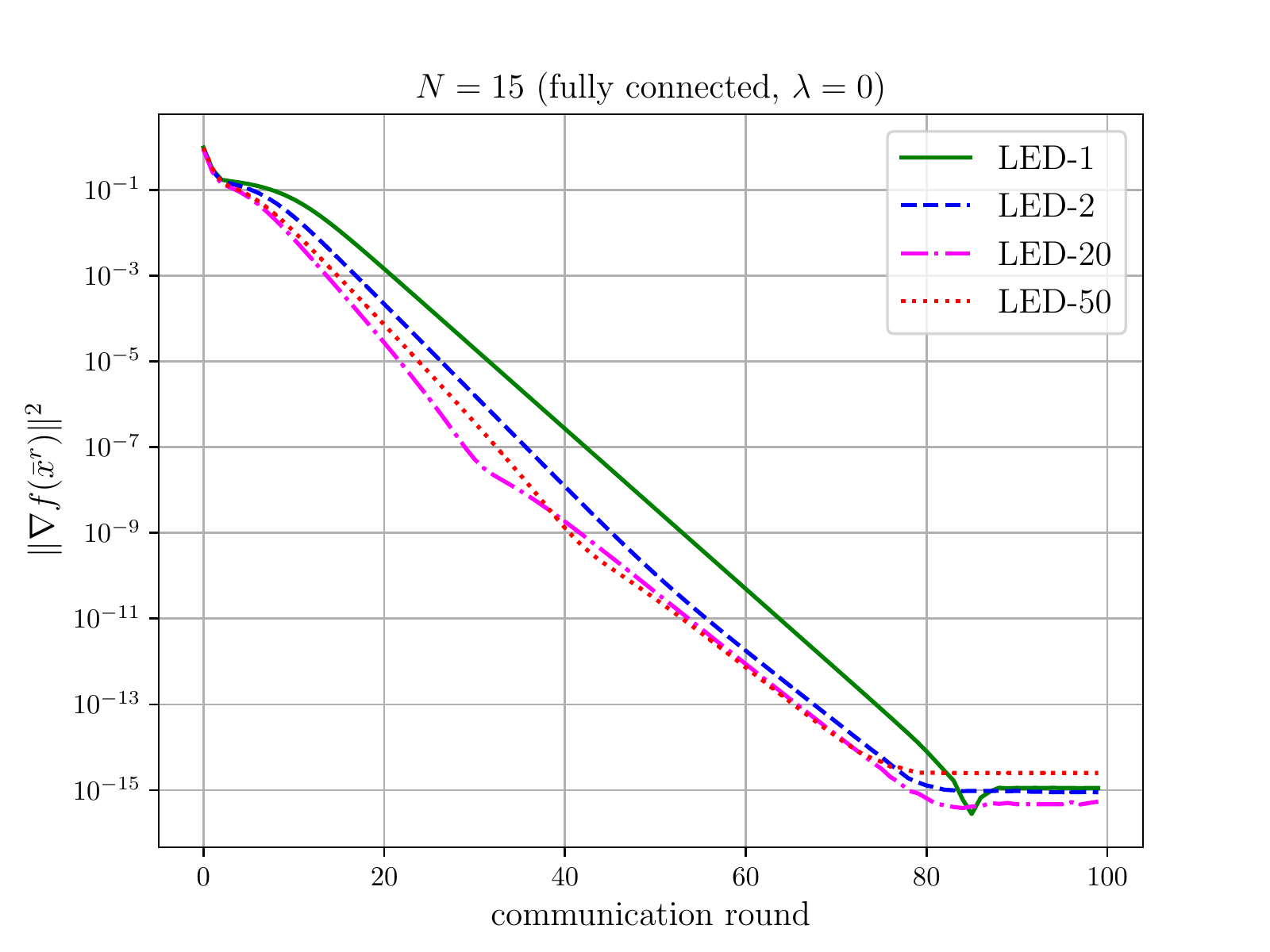}
		\end{center}
	\caption{Large data heterogeneity}
	\label{fig:led_sim_het}
	\end{subfigure}
		\begin{subfigure}{\columnwidth}
	\begin{center}
	\includegraphics[scale=0.33]{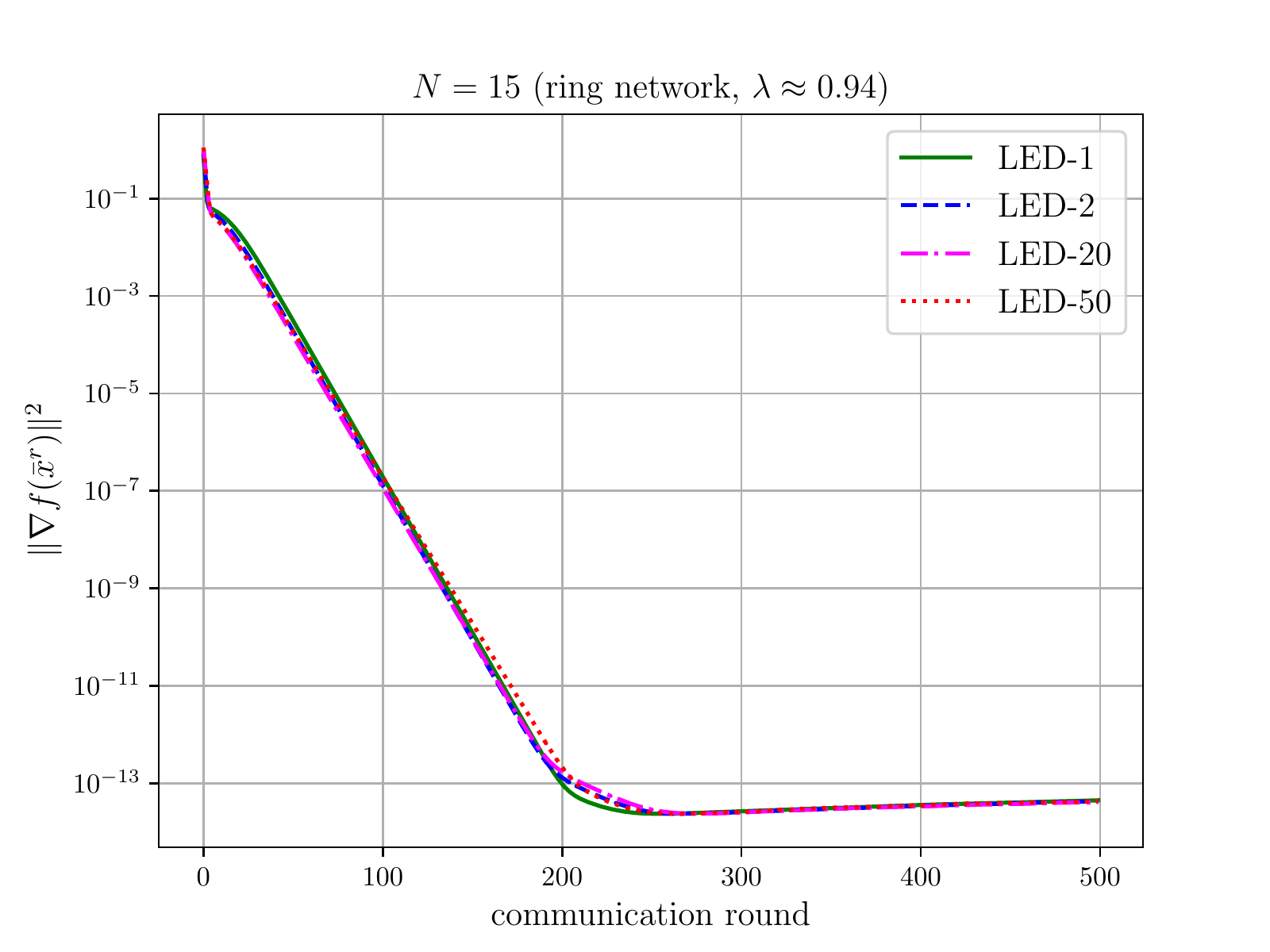}
	\includegraphics[scale=0.33]{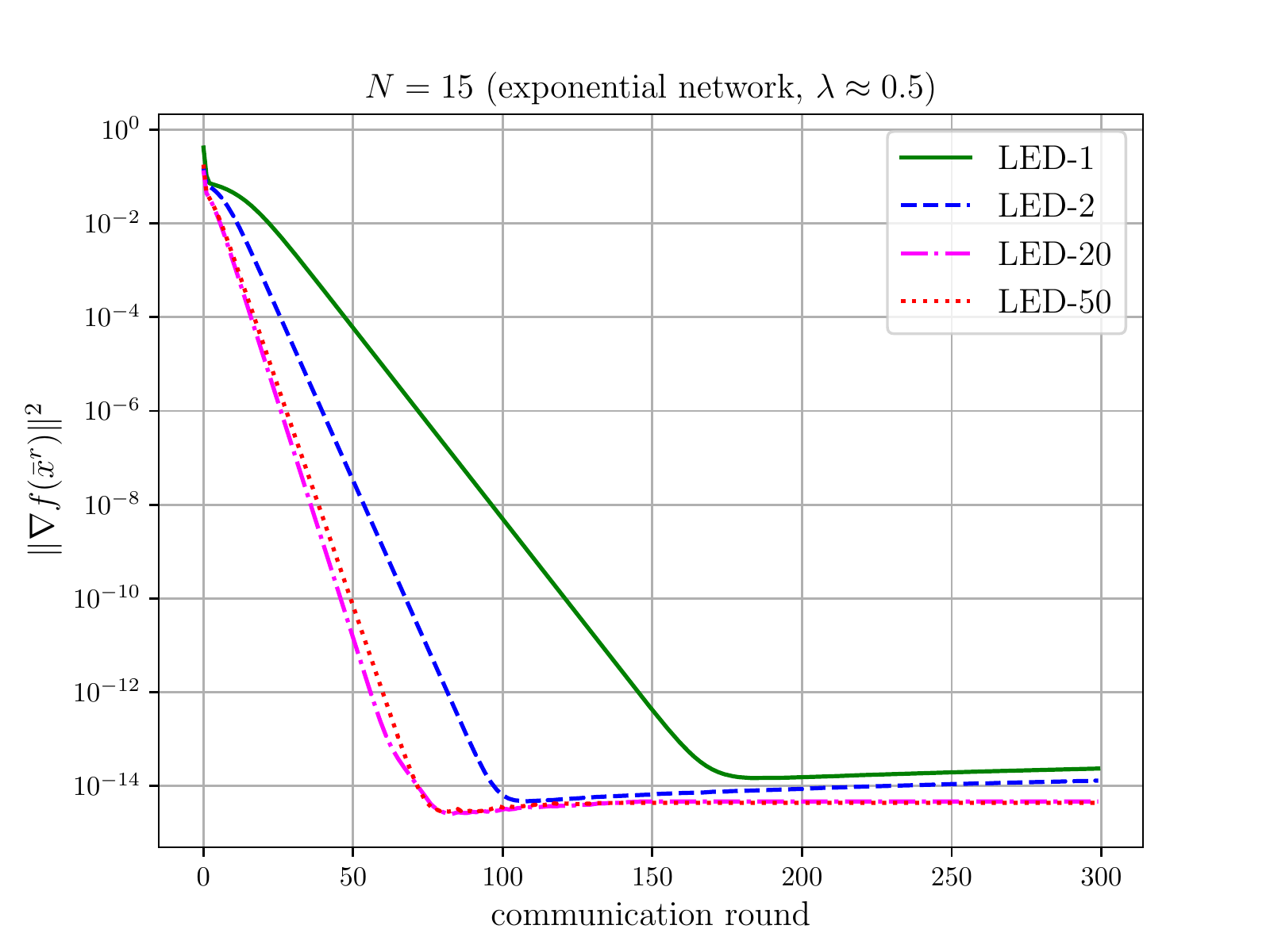}
	\includegraphics[scale=0.33]{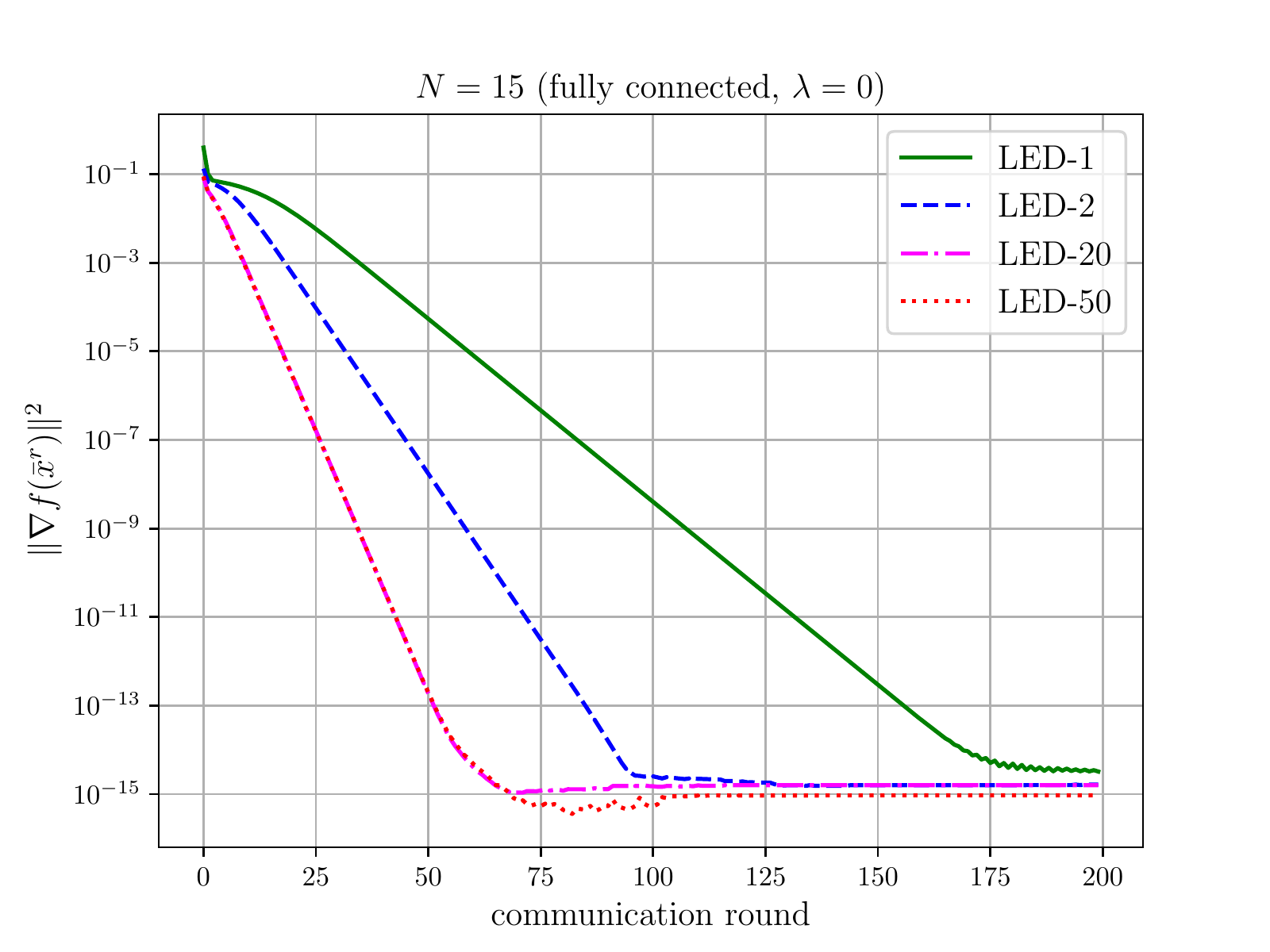}
\end{center}
\caption{Small data heterogeneity (similar data across the network)}
\label{fig:led_sim_iid}
\end{subfigure}
		\caption{\alg{LED} simulation results in the deterministic case (zero noise $\sigma=0$). In this Figure, \alg{LED-$\uptau$}  refers to \alg{LED} with $\uptau$ local updates. }
		\label{fig:led_sim}
	\end{figure}

\section{Concluding remarks} \label{sec_conc}
In this work, we proposed \alg{Local Exact-Diffusion} (\alg{LED}), a locally updated method inspired by the framework from \cite{alghunaim2019decentralized} and the Exact-Diffusion method \cite{yuan2019exactdiffI}. We demonstrated that \alg{LED} can be interpreted as the primal-dual method \alg{PDFP2O}/\alg{PAPC} from \cite{chen2013aprimal,drori2015simple}. We also explored its connection with the following methods: \alg{Exact-Diffusion} (\alg{ED}) \cite{yuan2019exactdiffI}, \alg{NIDS} \cite{li2017nids}, \alg{D$^2$} \cite{tang2018d}, \alg{Scaffnew} \cite{mishchenko2022proxskip}, \alg{VRL-SGD} \cite{liang2019variance}, and \alg{FedGate} \cite{haddadpour2021federated}. We proved the convergence of \alg{LED} in both convex and nonconvex settings and established bounds that offer improvements over existing decentralized methods. Finally, we provided numerical simulations to illustrate the effectiveness of the proposed algorithm.

A promising direction for future research involves examining \alg{LED} in the context of probabilistic local updates. It's worth exploring if the current analysis can be integrated with techniques from \alg{Scaffnew/ProxSkip} \cite{mishchenko2022proxskip} to understand the advantages of local steps in non-strongly-convex scenarios. Another potential avenue is broadening \alg{LED} to accommodate time-varying stochastic graphs. Furthermore, it would be insightful to see if decentralized local methods might also offer advantages for other network coupling constraints, as seen in multitask problems or the distributed feature problem.

\section*{Acknowledgments}
The author would like to thank Kun Yuan for his insightful discussions on parts of the manuscript.


\bibliographystyle{ieeetr}
\bibliography{myref_local_ed}

	
	\newpage
	\appendices
	\section{Preliminary transformation}
In this section, we will convert \alg{LED} updates \eqref{ed_updates_network} into another form more suitable for our analysis. To that end, we will first introduce some notation that complements the notation \eqref{notation:network_main}. 
	\subsection{Notation}
The following quantities will be used in the analysis:
\begin{subequations} \label{notation_proof}
	\begin{align}
		\bar{\x}^r &\define \one_N \otimes \bar{x}^r, \quad 	\bar{x}^r  \define  \frac{1}{N} \sum\limits_{i=1}^N x_i^r, \\ 
		\overline{\grad f } (\x^r)  &\define \frac{1}{N} \sum\limits_{i=1}^N  \grad f_i (x_i^r) \\
		\s_t^r &\define \nabla \F(\bm{\Phi}^{r}_t;\bxi_t) - \nabla \f(\bm{\Phi}^{r}_{t}) \label{noise_def} \\
		\bar{s}_t^r &\define  \frac{1}{N} \sum\limits_{i=1}^N \left( \nabla F_i(\phi^{r}_{i,t};\xi_{i,t}) - \nabla f_i(\phi^{r}_{i,t}) \right),  \\
		\z^r &\define \y^r +\frac{\alpha}{\beta}  \grad \f(\bar{\x}^r),   \label{z_def} \\
		\bar{\z}^r &\define \one_N \otimes \bar{z}^r, \quad \bar{z}^r  \define  \frac{1}{N}\sum\limits_{i=1}^N z_i^r.
	\end{align}
\end{subequations}
\subsection{Weight matrix decomposition}
We will use the structure of the weight matrix $W$ in our analysis, which is a requirement  for our proof. As a result, we will now go over some facts about the matrix $W$. When  Assumption \ref{assump:network} holds, then the weight matrix $W$ can be decomposed as follows \cite{horn2012matrix,sayed2014nowbook}:
	\begin{align*}
		W= \begin{bmatrix}
			\frac{1}{\sqrt{N}}  \one & \widehat{Q}
		\end{bmatrix} \begin{bmatrix}
			1 & 0 \\
			0 & \widehat{\Lambda}
		\end{bmatrix} \begin{bmatrix}
			\frac{1}{\sqrt{N}} \one\tran \vspace{0.6mm} \\  \widehat{Q}\tran
		\end{bmatrix},
	\end{align*}
	where $\widehat{\Lambda}=\diag\{\lambda_i\}_{i=2}^N$ and the matrix $\widehat{Q}$ has size ${N \times (N-1)}$, and satisfies $\widehat{Q}\widehat{Q}\tran=I_N-\tfrac{1}{N} \one \one\tran$  and  $\one\tran \widehat{Q}=0$. It follows that $\W \define W \otimes I_m$ can be decomposed as
	\begin{subequations}  \label{W_QQtran}
		\begin{align}
			\W&=  \begin{bmatrix}
				\frac{1}{\sqrt{N}}   \one \otimes I_m & \widehat{\Q}
			\end{bmatrix} \begin{bmatrix}
				I_m & 0 \\
				0 & \widehat{\mathbf{\Lambda}}
			\end{bmatrix} \begin{bmatrix}
				\frac{1}{\sqrt{N}}  \one\tran \otimes I_m \\ \widehat{\Q}\tran
			\end{bmatrix}=\tfrac{1}{N} \one \one\tran \otimes I_m+\widehat{\Q} \widehat{\mathbf{\Lambda}} \widehat{\Q}\tran,
		\end{align}
	where $\widehat{\mathbf{\Lambda}} \define \widehat{\Lambda} \otimes I_m \in \real^{m(N-1)\times m(N-1)}$ and $\widehat{\Q}\define \widehat{Q} \otimes I_m \in \real^{m N \times m(N-1)}$ satisfies:
	\begin{align}
		\widehat{\Q}\tran \widehat{\Q}=\I, \quad \widehat{\Q}\widehat{\Q}\tran=\I-\tfrac{1}{N} \one \one\tran \otimes I_m, \quad (\one\tran \otimes I_m) \widehat{\Q}=0.
	\end{align}
\end{subequations}
Below, we provide a list of relevant facts and properties concerning the aforementioned decomposition.
\begin{itemize}

\item  It holds that
\begin{align} \label{Q_matrix_bound}
		\|\widehat{\Q}\tran \x^r \|^2= \|\widehat{\Q} \widehat{\Q}\tran \x^r \|^2= \|\x^r - \bar{\x}^r\|^2,  \quad \text{and} \quad \|\widehat{\Q}\|=1.
\end{align}

	\item The matrix $\widehat{\B} \define \I-\widehat{\mathbf{\Lambda}}$	satisfies
	\begin{align} \label{B_matrix_bound}
		\|\widehat{\B}\| = 1-\underline{\lambda}  \leq 1 \quad \text{and} \quad  \|\widehat{\B}^{-1}\|=\frac{1}{1-\lambda},
	\end{align}
where $\underline{\lambda}$ is the smallest nonzero eigenvalue of $W$ and $\lambda$ is the network's mixing rate introduced in \eqref{mixing_rate}.
\end{itemize}

	\subsection{Transformed recursion}
To obtain our result, we will perform a series of transformations that will eventually  lead us to the critical result specified in Lemma \ref{lemma:transf_avg_hat}. In order to study the communication complexity of \alg{LED}, we begin by representing it in terms of communication rounds.  It holds true when iterating through \eqref{local_ed_updates} updates:
	\begin{align*}
		\bm{\Phi}^{r}_{\uptau} &=    \x^{r}-\alpha \sum_{t=0}^{\uptau-1} \nabla \F(\bm{\Phi}^{r}_{t};\bxi_{t}) - \beta  \uptau \y^r,
	\end{align*}
	where $\bm{\Phi}^r_{0}=\x^r$ and, for notational simplicity, we are removing the superscript in $\bxi_{t}^r$.	Substituting the preceding into \eqref{local_ed_comm_round} yields
	\begin{subequations} \label{x_and_y_r_update}
		\begin{align}
			\x^{r+1} &=   \W \left(\x^{r}  - \beta \uptau \y^r - \alpha \sum_{t=0}^{\uptau-1} \nabla \F(\bm{\Phi}^{r}_{t};\bxi_t)  \right) \label{x_r_upd} \\
			\y^{r+1} &= \y^{r}+  (\I-\W)   \left(\x^{r}- \beta \uptau \y^r  - \alpha \sum_{t=0}^{\uptau-1} \nabla \F(\bm{\Phi}^{r}_t;\bxi_t)   \right).
		\end{align}
	\end{subequations}
	Observe that under our initialization, $\y^{0}=(\I-\W) \x^0$, $\y^r$ will always be in the range of $\I-\W$. As a result, we have $(\one\tran \otimes I_m) \y^r =0$ for all $r$. Multiplying both sides of \eqref{x_r_upd} by $(1/N) (\one\tran \otimes I_m)$ on the left and using the definitions in \eqref{notation_proof} , namely $	\overline{\grad f } (\x^r)  = \frac{1}{N} \sum_{i=1}^N  \grad f_i (x_i^r),
\bar{s}_t^r =  \frac{1}{N} \sum_{i=1}^N ( \nabla F_i(\phi^{r}_{i,t};\xi_{i,t}) - \nabla f_i(\phi^{r}_{i,t}) )$],  we get
	\begin{align} 
			\bar{x}^{r+1}&=\bar{x}^r -  \alpha \sum_{t=0}^{\uptau-1} \big( \overline{\nabla f}(\bm{\Phi}^r_t) + \bar{s}_t^r  \big). \label{error_avg_progress}
	\end{align}
	Equation \eqref{error_avg_progress} above describes how the average (centroid) vector evolves in terms of communication rounds, which will be important in our analysis. Now,
	using the definitions in \eqref{notation_proof}, namely $\z^r = \y^r +\frac{\alpha}{\beta}  \grad \f(\bar{\x}^r)$ and $	\s_t^r = \nabla \F(\bm{\Phi}^{r}_t;\bxi_t) - \nabla \f(\bm{\Phi}^{r}_{t})$,   the update \eqref{x_and_y_r_update} can be equivalently described  as
	\begin{subequations} \label{x_and_z_r_update}
		\begin{align}
			\x^{r+1} &=   \W \left(\x^{r}  - \beta \uptau \z^r  - \alpha \sum_{t=0}^{\uptau-1} \big(\nabla \f(\bm{\Phi}^{r}_{t}) - \grad \f(\bar{\x}^r) + \s_t^r\big) \right) \label{x_r_update}   \\
			\z^{r+1} &= [(1- \beta \uptau) \I + \beta \uptau \W]  \z^{r}+  (\I-\W)  \x^{r}  - \alpha  (\I-\W)  \sum_{t=0}^{\uptau-1} \big(\nabla \f(\bm{\Phi}^{r}_{t} ) -\grad \f(\bar{\x}^{r}) + \s_t^r \big)
			\nonumber \\
			& \quad +  \frac{\alpha}{\beta}  \big(\grad \f(\bar{\x}^{r+1})-\grad \f(\bar{\x}^{r}) \big) \label{z_r_update}.  
		\end{align}
	\end{subequations}
\begin{remark}[\sc \small Deviation from Average] \rm
	The introduction of $\z^r$ is inspired from \cite{alghunaim2021unified}. The quantity $\z^r$ can be interpreted as a variable that tracks the average gradient vector $\one \otimes \grad f(\bar{x}^r)$. Observe that by $(\one\tran \otimes I_m) \y^r =0$ and \eqref{z_def}, we have
	\begin{align} \label{avg_z_grad}
		\bar{z}^r= (1/N) (\one\tran \otimes I_m) \z^r = \frac{\alpha  }{\beta N} \sum_{i=1}^N \grad f_i(\bar{x}^r).
	\end{align}
We will next transform the updates \eqref{x_and_z_r_update} into  quantities that measure how far	$\x^{r}$ and $\z^{r}$ deviate from the averages $\bar{\x}^r = \one_N \otimes \bar{x}^r$ and $\bar{\z}^r \define \one_N \otimes \bar{z}^r $, respectively. To do so, we will leverage the structure and properties of the weight matrix $\W$ \eqref{Q_matrix_bound}.  \qd
\end{remark}
 Using the decomposition of $\W$ given in \eqref{W_QQtran}, it holds that $\widehat{\Q}\tran \W=\widehat{\mathbf{\Lambda}} \widehat{\Q}\tran$.
	Therefore, multiplying both sides of \eqref{x_and_z_r_update} by $\widehat{\Q}\tran$, it holds that
	\begin{subequations} \label{trans_updates3r}
		\begin{align}    
			\widehat{\Q}\tran \x^{r+1} &=  \widehat{\mathbf{\Lambda}}  \widehat{\Q}\tran \x^{r}-  \beta \uptau \widehat{\mathbf{\Lambda}}  \widehat{\Q}\tran \z^{r}  -\alpha  \widehat{\mathbf{\Lambda}} \widehat{\Q}\tran   \sum_{t=0}^{\uptau-1} \big(\nabla \f(\bm{\Phi}^{r}_{t}) - \grad \f(\bar{\x}^r) + \s_t^r \big)    \\
			\widehat{\Q}\tran \z^{r+1} &= [(1- \beta \uptau)\I +  \beta \uptau \widehat{\mathbf{\Lambda}} ] \widehat{\Q}\tran \z^{r}+  (\I-\widehat{\mathbf{\Lambda}})  \widehat{\Q}\tran \x^{r} 
			\nonumber \\
			& \quad -\alpha  (\I-\widehat{\mathbf{\Lambda}}) \widehat{\Q}\tran  \sum_{t=0}^{\uptau-1} \big(\nabla \f(\bm{\Phi}^{r}_{t}) - \grad \f(\bar{\x}^r) + \s_t^r \big)
			+ \frac{\alpha}{\beta}
			\widehat{\Q}\tran \big(\grad \f(\bar{\x}^{r+1})-\grad \f(\bar{\x}^{r})\big).
		\end{align}
	\end{subequations}
	Rewriting  	\eqref{trans_updates3r} in matrix notation, we have
	\begin{align}  \label{trans_update_matrix_r}
		\begin{bmatrix}
			\widehat{\Q}\tran	\x^{r+1} \\
			\widehat{\Q}\tran	\z^{r+1}
		\end{bmatrix} 
		&= 
		\begin{bmatrix}
			\widehat{\mathbf{\Lambda}} & -\beta \uptau  \widehat{\mathbf{\Lambda}} \\
			\I-\widehat{\mathbf{\Lambda}} & (1- \beta \uptau)\I+ \beta \uptau \widehat{\mathbf{\Lambda}}
		\end{bmatrix} 
		\begin{bmatrix}
			\widehat{\Q}\tran	\x^{r} \\
			\widehat{\Q}\tran	\z^{r}
		\end{bmatrix} \nonumber \\
		& \quad
		- \alpha \begin{bmatrix}
			\widehat{\mathbf{\Lambda}}  \widehat{\Q}\tran \sum\limits_{t=0}^{\uptau-1} \big(\nabla \f(\bm{\Phi}^{r}_{t}) - \grad \f(\bar{\x}^r) + \s_t^r \big) \\
			(\I-\widehat{\mathbf{\Lambda}})  \widehat{\Q}\tran \sum\limits_{t=0}^{\uptau-1} \big(\nabla \f(\bm{\Phi}^{r}_{t}) - \grad \f(\bar{\x}^r) + \s_t^r \big) - \frac{1}{\beta}  \widehat{\Q}\tran \big(\grad \f(\bar{\x}^{r+1})-\grad \f(\bar{\x}^{r}) \big)
		\end{bmatrix}.
	\end{align}
Note that from \eqref{Q_matrix_bound}, we have $\|\widehat{\Q}\tran \x^r \|^2=\|\x^r - \bar{\x}^r\|^2$ and similarly $\|\widehat{\Q}\tran \z^r \|^2=\|\z^r - \bar{\z}^r\|^2$. Thus, the above describes how the nodes vectors deviates from the averages.	For $\beta=1/\uptau$, we will let
\begin{align} \label{matrix_cons}
 \D \define 	\begin{bmatrix}
		\widehat{\mathbf{\Lambda}} & -\beta \uptau  \widehat{\mathbf{\Lambda}} \\
		\I-\widehat{\mathbf{\Lambda}} & (1- \beta \uptau)\I+ \beta \uptau \widehat{\mathbf{\Lambda}}
	\end{bmatrix} &=\begin{bmatrix}
		\widehat{\mathbf{\Lambda}} & -  \widehat{\mathbf{\Lambda}} \\
		\I-\widehat{\mathbf{\Lambda}} &  ~~\widehat{\mathbf{\Lambda}}
	\end{bmatrix}  = \begin{bmatrix}
	\widehat{\Lambda} & -  \widehat{\Lambda} \\
	I-\widehat{\Lambda} &  ~~\widehat{\Lambda}
\end{bmatrix} \otimes I_m.
\end{align}
If the norm of the matrix $\D$ is less than one $\|\D \| <1$, then the updates \eqref{trans_update_matrix_r} can be used to directly measure the deviation from the averages. However, even though the eigenvalues of $\D$ are less than one, its norm is not guaranteed to satisfy $\|\D\| < 1$; but, we can decompose $\D$ and transform \eqref{trans_update_matrix_r} into a more suitable form for our analysis, as shown in the important result below.
	\begin{lemma}[\sc \small Deviation from average] \label{lemma:transf_avg_hat}
		\rm Suppose that Assumption \ref{assump:network} holds and  $\beta=1/\uptau$, then
			\begin{align}
				\widehat{\d}^{r+1}&=\mathbf{\Delta} \widehat{\d}^{r} - \alpha \widehat{\V}^{-1}   \begin{bmatrix}
					\widehat{\mathbf{\Lambda}}  \widehat{\Q}\tran \sum\limits_{t=0}^{\uptau-1} \big(\nabla \f(\bm{\Phi}^r_t) - \grad \f(\bar{\x}^r) + \s_t^r \big) \\
					\widehat{\B}  \widehat{\Q}\tran \sum\limits_{t=0}^{\uptau-1} \big(\nabla \f(\bm{\Phi}^r_{t}) - \grad \f(\bar{\x}^r) + \s_t^r \big) -  \uptau    \widehat{\Q}\tran \big(\grad \f(\bar{\x}^{r+1})-\grad \f(\bar{\x}^{r})\big)
				\end{bmatrix}, \label{error_check_progress}
			\end{align}
		where 	$\mathbf{\Delta}$ is a matrix with norm $\delta \define  \|\mathbf{\Delta}\|=\sqrt{\lambda} < 1$,  $\widehat{\B} \define \I-\widehat{\mathbf{\Lambda}}$,   and
	\begin{subequations}
			\begin{align} \label{e_hat_def}
			\widehat{\d}^{r} &\define  \widehat{\V}^{-1} \begin{bmatrix}
				\widehat{\Q}\tran\x^{r} \\
				\widehat{\Q}\tran\z^{r}
			\end{bmatrix} \\
		\widehat{\V}^{-1}
		& \define \P \begin{bmatrix}
			\frac{1}{2 } \widehat{\mathbf{\Lambda}}^{-\frac{1}{2}}  & \frac{\jmath}{2} \widehat{\B}^{-\frac{1}{2}}  \\
			\frac{1}{2} \widehat{\mathbf{\Lambda}}^{-\frac{1}{2}}    & -  \frac{\jmath}{2} \widehat{\B}^{-\frac{1}{2}}
		\end{bmatrix} , \quad 
		\widehat{\V} \define \begin{bmatrix}
		\widehat{\mathbf{\Lambda}}^{\frac{1}{2}}  & 	\widehat{\mathbf{\Lambda}}^{\frac{1}{2}} \\
		-\jmath  \widehat{\B}^{\frac{1}{2}}   &   	\jmath  \widehat{\B}^{\frac{1}{2}} 
	\end{bmatrix} \P\tran.  \label{V_structure}
	\end{align}
	\end{subequations}
Here, $\P$ is some permutation matrix and $\jmath=\sqrt{-1}$ is the imaginary number.
	\end{lemma}
	\begin{proof}
		The proof exploits the special structure of the matrix $\D$ given in \eqref{matrix_cons}. First note that if
		\begin{align} \label{diag_A}
			 A  = \begin{bmatrix}
				\diag\{b_i\}_{i=2}^N & \diag\{c_i\}_{i=2}^N \\
				\diag\{d_i\}_{i=2}^N & \diag\{e_i\}_{i=2}^N 
			\end{bmatrix} \define \begin{bmatrix}
				B & C \\
				D & E 
			\end{bmatrix}.
		\end{align}
		where $b_i,c_i,d_i,e_i$ are constants, then there exists a permutation matrix $P$ such that
		\begin{align} \label{bdiag_A_permuted}
	P	A P\tran= \bdiag\{A_i\}_{i=2}^N, \quad A_i=\begin{bmatrix}
			b_i & c_i \\
			d_i & e_i
		\end{bmatrix}.
	\end{align}
			 Since the blocks of $\D$ given in \eqref{matrix_cons} are diagonal matrices, there exists a permutation matrix $\P$ such that 
			\begin{align*} 
			\P \D \P\tran = \bdiag\{D_i \}_{i=2}^N \otimes I_m, \quad 	D_i = \begin{bmatrix}
					\lambda_{i} & -  \lambda_{i}   \\
					1-\lambda_i	 & ~~\lambda_{i}
				\end{bmatrix} \in \real^{2 \times 2}.
			\end{align*}
			The eigenvalues of $D_i$ ($i=2,\ldots,N$) are
			\[
			\begin{aligned}
				\delta_{(1,2),i} &=(1/2) \left[\Tr(D_i) \pm \sqrt{\Tr(D_i)^2-4 \det(D_i)}\right] \\
				&= \lambda_i \pm \sqrt{ \lambda_i^2 - \lambda_i}.
			\end{aligned}
			\]
		Notice that $|\delta_{(1,2),i}|<1$ when $-\frac{1}{3}<\lambda_i<1$, which holds under Assumption \ref{assump:network}	since  $W >0 $, \ie, $0<\lambda_i <1$ ($i=2,\dots,N$).   For $0<\lambda_i <1$,  the eigenvalues of $D_i$ are complex and distinct: 
				\begin{align*} 
					\delta_{(1,2),i} = \lambda_i \pm \jmath \sqrt{\lambda_i-\lambda_i^2}, \qquad |\delta_{(1,2),i} |=\sqrt{\lambda_i}<1,
				\end{align*}
				where $\jmath^2=-1$. Through algebraic multiplication it can be verified that $	D_i=V_i \Delta_i V_{i}^{-1}$ where 
				\begin{align} \label{delta_i_def}
							\Delta_i=
					\begin{bmatrix}
						\lambda_i + \jmath \sqrt{\lambda_i-\lambda_i^2} & 0 \\
						0 & \lambda_i -  \jmath \sqrt{\lambda_i-\lambda_i^2},
					\end{bmatrix}
				\end{align}
				and
				\begin{align} \label{V_i_def}
					V_i=\begin{bmatrix}
						\sqrt{\lambda_i}  & \sqrt{\lambda_i}  \\
						- \jmath \sqrt{1-\lambda_i}   & \jmath  \sqrt{1-\lambda_i}
					\end{bmatrix}, \quad 		V_i^{-1}=
					\begin{bmatrix}
						\frac{1}{2 \sqrt{\lambda_i}}  & \frac{\jmath}{2\sqrt{1-\lambda_i}}  \\
						\frac{1}{2 \sqrt{\lambda_i}}    & -  \frac{\jmath}{2\sqrt{1-\lambda_i}}
					\end{bmatrix}.
				\end{align}	
				We conclude that $\D=\P\tran \V \mathbf{\Delta} \V^{-1} \P$  where $\V=\bdiag\{V_i\}_{i=2}^N \otimes I_m$ and $\mathbf{\Delta} = \bdiag\{\Delta_i\}_{i=2}^N \otimes I_m$. Therefore, left multiplying both sides of \eqref{trans_update_matrix_r} by $\widehat{\V}^{-1}$ where $\widehat{\V} = \P\tran \V$  gives \eqref{error_check_progress}.				Exploiting the structure of $\V^{-1}=\bdiag\{V_i^{-1}\}_{i=2}^N \otimes I_m$ where $V_i^{-1}$ is defined in \eqref{V_i_def} and using \eqref{bdiag_A_permuted}, we get 
					\begin{align*} 
					\widehat{\V}^{-1}= \V^{-1} \P =\P (\P\tran  \V^{-1} \P) 
					&=\P \begin{bmatrix}
						\diag\{\frac{1}{2 \sqrt{\lambda_i}}\}_{i=2}^N   & \diag\{\frac{\jmath}{2\sqrt{1-\lambda_i}}\}_{i=2}^N  \\
						\diag\{\frac{1}{2 \sqrt{\lambda_i}}\}_{i=2}^N    & -  \diag\{\frac{\jmath}{2\sqrt{1-\lambda_i}}\}_{i=2}^N
					\end{bmatrix} \otimes I_m
					\nonumber \\
					& = \P \begin{bmatrix}
						\frac{1}{2 } \widehat{\mathbf{\Lambda}}^{-\frac{1}{2}}  & \frac{\jmath}{2} (\I-\widehat{\mathbf{\Lambda}})^{-\frac{1}{2}}  \\
						\frac{1}{2} \widehat{\mathbf{\Lambda}}^{-\frac{1}{2}}    & -  \frac{\jmath}{2} (\I-\widehat{\mathbf{\Lambda}})^{-\frac{1}{2}}
					\end{bmatrix}.
				\end{align*}
	Using similar arguments for $\V$ gives \eqref{V_structure}.


	\end{proof}
	
\begin{remark}[\sc \small Centralized case] \label{remark:centralized_case} \rm The transformation in Lemma \ref{lemma:transf_avg_hat} is needed for  decentralized network analysis since,  as explained before, the norm of the matrix $\D$ \eqref{matrix_cons} is not necessarily less than one.  However, when the network is fully-connected (centralized case),  we have $\widehat{\mathbf{\Lambda}}=\zero$, and thus, it follows from \eqref{trans_update_matrix_r} that $\widehat{\Q}\tran \x^{r}=\zero$ and when $\beta=1/\uptau$, we have
		\begin{align}    \label{centralize_case_deviation}
		\d^{r+1}_{\rm cen} \define	\widehat{\Q}\tran \z^{r+1} &=   -\alpha   \widehat{\Q}\tran  \sum_{t=0}^{\uptau-1} \big(\nabla \f(\bm{\Phi}^{r}_{t}) - \grad \f(\bar{\x}^r) + \s_t^r \big)
			+ \uptau \alpha
			\widehat{\Q}\tran \big(\grad \f(\bar{\x}^{r+1})-\grad \f(\bar{\x}^{r})\big).
		\end{align}
	In this case, the analysis can be greatly simplified since $\D=\zero$. This specialization covers the method \alg{VRL-SGD} \cite{liang2019variance}. A similar approach can also be adapted to \alg{FedGate} \cite{haddadpour2021federated}, as described in \eqref{fedgate_network}. Doing so eventually leads to the result in Corollary \ref{cor_centralized}. See Appendix \ref{app_server_proof}.  \qd
\end{remark}

		\section{Convergence  analysis}
		In this section, we will prove Theorems \ref{thm:noncvx} and \ref{thm:cvx}. The proof utilizes equations \eqref{error_avg_progress} and \eqref{error_check_progress} derived in the previous section. We want to emphasize that some of the bounds may appear tedious, even though they only involve basic algebra. Such complexity is common in decentralized analysis, as seen in, for example, \cite{alghunaim2021unified,koloskova2019decentralized}.
		
			\subsection{Auxiliary results}
In the proof, we will use the following useful results and facts. (You may overlook and refer to this subsection later in the proofs.)
		
		\begin{itemize}
			
				\item 	 Since the squared norm $\|\cdot\|^2$ is convex, applying Jensen's inequality, it holds that \cite{boyd2004convex}
			\begin{subequations} \label{jensens}
				\begin{align} \label{jensens_a}
					\|{\textstyle \sum\limits_{k=1}^{K}} a_{k}\|^{2} \leq K \sum_{k=1}^{K}\left\|a_{k}\right\|^{2}, 
				\end{align}
				for all  vectors $\{a_k\}_{k=1}^K$ of equal size  and positive integer  $K$.
				Moreover, for any equal-size vectors $a$ and $b$, we have for $\theta \in (0,1)$:
				\begin{align} \label{jensens_b}
					\|a+b\|^2 \leq \frac{1}{\theta} \|a\|^2+\frac{1}{1-\theta}\|b\|^2.
				\end{align}
			\end{subequations}
				Many bounds later in the proofs uses inequality \eqref{jensens_a} without referring to it to avoid redundancies.

		\item 	Taking the squared norm on both sides of \eqref{e_hat_def} and using \eqref{jensens_a}, it holds that:
		\begin{align} \label{dhat_bound}
			\| \widehat{\d}^{r}\|^2 \leq \frac{1}{4} \left\| \begin{bmatrix}
				\widehat{\mathbf{\Lambda}}^{-\frac{1}{2}} \widehat{\Q}\tran\x^{r} + \jmath \widehat{\B}^{-\frac{1}{2}} \widehat{\Q}\tran\z^{r} \\
				\widehat{\mathbf{\Lambda}}^{-\frac{1}{2}} \widehat{\Q}\tran\x^{r}
				- \widehat{\B}^{-\frac{1}{2}} \widehat{\Q}\tran\z^{r}
			\end{bmatrix} \right\|^2  \leq  \|	\widehat{\mathbf{\Lambda}}^{-\frac{1}{2}} \widehat{\Q}\tran\x^{r}\|^2 + \|\widehat{\B}^{-\frac{1}{2}} \widehat{\Q}\tran\z^{r}\|^2 .
		\end{align}

				\item From \eqref{e_hat_def}--\eqref{V_structure}, it holds that
				\begin{align}
					\begin{bmatrix}
						\widehat{\Q}\tran\x^{r} \\
						\widehat{\Q}\tran\z^{r}
					\end{bmatrix} =   \begin{bmatrix}
						\widehat{\mathbf{\Lambda}}^{\frac{1}{2}}  & 	\widehat{\mathbf{\Lambda}}^{\frac{1}{2}} \\
						-\jmath  \widehat{\B}^{\frac{1}{2}}   &   	\jmath  \widehat{\B}^{\frac{1}{2}} 
					\end{bmatrix}	\P\tran	\widehat{\d}^{r} =   \begin{bmatrix}
						\widehat{\mathbf{\Lambda}}^{\frac{1}{2}} (\P_u\tran  + \P_l\tran   )\widehat{\d}^{r} \\
						-	\jmath  \widehat{\B}^{\frac{1}{2}}  (\P_u\tran  - \P_l\tran   ) \widehat{\d}^{r}
					\end{bmatrix},
				\end{align}
				where $\P_u\tran$ and $\P_l\tran$  are the upper and lower blocks of $\P\tran=[\P_u\tran ; \P_l\tran]$.	
				It follows that:
				\begin{subequations} \label{dhat_bound_x_z_dev}
			\begin{align} 
				\|\x^r - \bar{\x}^r\|^2 &= \|\widehat{\Q}\tran \x^{r}\|^2
				= \|\widehat{\mathbf{\Lambda}}^{\frac{1}{2}} (\P_u\tran + \P_l\tran  ) \widehat{\d}^{r} \|^2 \leq 4 \lambda  \|\widehat{\d}^{r}\|^2   \\
				\|\z^r - \bar{\z}^r \|^2 &=	 \|\widehat{\Q}\tran \z^{r}\|^2
				= \| -\jmath  \widehat{\B}^{\frac{1}{2}} (\P_u\tran - \P_l\tran  )  \widehat{\d}^{r} \|^2
				\leq 4 \|\widehat{\B}\|  \|\widehat{\d}^{r}\|^2,
			\end{align}
		\end{subequations}
				where we used $\|\P_u\tran + \P_l\tran\|^2\leq 4$ and $\|\P_u\tran - \P_l\tran\|^2 \leq 4$ since $\P$ is a permutation matrix $\|\P\|=1$.
				
				\item 	
			It holds that
				\begin{align}
				\s^r \define		\widehat{\V}^{-1}   \begin{bmatrix}
						\widehat{\mathbf{\Lambda}}  \widehat{\Q}\tran \sum\limits_{t=0}^{\uptau-1}  \s_t^r  \\
						\widehat{\B}   \widehat{\Q}\tran\sum\limits_{t=0}^{\uptau-1}   \s_t^r 
					\end{bmatrix} 
					 \overset{\eqref{V_structure}}&{=} \P \begin{bmatrix}
						\frac{1}{2 } \widehat{\mathbf{\Lambda}}^{-\frac{1}{2}}  & \frac{\jmath}{2} \widehat{\B}^{-\frac{1}{2}}  \\
						\frac{1}{2} \widehat{\mathbf{\Lambda}}^{-\frac{1}{2}}    & -  \frac{\jmath}{2} \widehat{\B}^{-\frac{1}{2}}
					\end{bmatrix} 
					\begin{bmatrix}
						\widehat{\mathbf{\Lambda}}  \widehat{\Q}\tran \sum\limits_{t=0}^{\uptau-1}  \s_t^r  \\
						\widehat{\B}   \widehat{\Q}\tran \sum\limits_{t=0}^{\uptau-1}   \s_t^r 
					\end{bmatrix} 
					\\
					&= \tfrac{1}{2} \P
					\begin{bmatrix}
						\widehat{\mathbf{\Lambda}}^{\frac{1}{2}}  \widehat{\Q}\tran \sum\limits_{t=0}^{\uptau-1}  \s_t^r + \jmath \widehat{\B}^{\frac{1}{2}} \widehat{\Q}\tran \sum\limits_{t=0}^{\uptau-1}   \s_t^r \\
						\widehat{\mathbf{\Lambda}}^{\frac{1}{2}}  \widehat{\Q}\tran \sum\limits_{t=0}^{\uptau-1}  \s_t^r - \jmath \widehat{\B}^{\frac{1}{2}} \widehat{\Q}\tran \sum\limits_{t=0}^{\uptau-1}   \s_t^r
					\end{bmatrix}.
				\label{s_V_def}
				\end{align}
				Therefore,
				\begin{align}
					\| \s^r \|^2 &\leq \tfrac{1}{4}  \left(\|\widehat{\mathbf{\Lambda}}^{\frac{1}{2}}  \widehat{\Q}\tran \textstyle \sum\limits_{t=0}^{\uptau-1}  \s_t^r + \jmath \widehat{\B}^{\frac{1}{2}} \widehat{\Q}\tran \sum\limits_{t=0}^{\uptau-1}   \s_t^r\|^2 + \|\widehat{\mathbf{\Lambda}}^{\frac{1}{2}}  \widehat{\Q}\tran \sum\limits_{t=0}^{\uptau-1}  \s_t^r - \jmath \widehat{\B}^{\frac{1}{2}} \widehat{\Q}\tran \sum\limits_{t=0}^{\uptau-1}   \s_t^r\|^2 \right) \nonumber \\
					& \leq  2 \|\textstyle \sum\limits_{t=0}^{\uptau-1}   \s_t^r \|^2. \label{vinv_noise_bound}
				\end{align} 
			The last step holds by using Jensen's inequality \eqref{jensens} and  \eqref{Q_matrix_bound}--\eqref{B_matrix_bound}. 				Following similar arguments, it can be shown that the squared norm of
					\begin{align} \label{h_V_def}
				\h^{r+1} & \define \widehat{\V}^{-1}   \begin{bmatrix}
					\widehat{\mathbf{\Lambda}}  \widehat{\Q}\tran \sum\limits_{t} \big(\nabla \f(\bm{\Phi}^r_t) - \grad \f(\bar{\x}^r)  \big) \\
					\widehat{\B}  \widehat{\Q}\tran \sum\limits_{t} \big(\nabla \f(\bm{\Phi}^r_{t}) - \grad \f(\bar{\x}^r)  \big) -  \uptau  \widehat{\Q}\tran \big(\grad \f(\bar{\x}^{r+1})-\grad \f(\bar{\x}^{r})\big)
				\end{bmatrix}.
			\end{align}
		is upper bounded by
				\begin{align}
						\| 	\h^{r+1} \|^2 
					& \leq  4 \|\textstyle \sum\limits_{t=0}^{\uptau-1}  \nabla \f(\bm{\Phi}^r_t) - \grad \f(\bar{\x}^r)\|^2 +  \uptau^2 \|	 \widehat{\B}^{-1}\| \|\grad \f(\bar{\x}^{r+1})-\grad \f(\bar{\x}^{r})\|^2.
					 \label{vinv_grad_bound}
				\end{align}

		\end{itemize}
	
	\subsection{Key bounds}
	In this section, we derive some key bounds that will be used to establish our result for both nonconvex and convex cases.

 The first bound involves the cumulative deviation of the local updates  from the averaged vector at the previous communication round defined as
	\begin{align} \label{drift_def}
		\|\bm{\widehat{\Phi}}^r\|^2 \define \sum\limits_{t=0}^{\uptau-1}  \|\bm{\Phi}^r_{t} - \bar{\x}^r\|^2 =	    \sum_{t=0}^{\uptau-1} \sum_{i =1}^N    \|\phi^{r}_{i,t}-\bar{x}^r\|^2.
	\end{align}
	where $\bm{\widehat{\Phi}}^r \define \col\{\bm{\Phi}^r_{t} - \bar{\x}^r\}_{t=0}^{\uptau-1}$. The term \eqref{drift_def}  will appear frequently in our analysis.
	
	\begin{lemma}[\sc \small Local steps bound] \label{lema_drift}
		\rm  Let Assumptions \ref{assump:noise}--\ref{assump:smoothness} hold, then for $\alpha  \leq \frac{1}{2 \sqrt{2} L \uptau} $ we have
		\begin{align} \label{drift_bound}
			\Ex 	\|\bm{\widehat{\Phi}}^r\|^2 & \leq 64  \uptau  \Ex \|\widehat{\d}^r\|^2+ 16   \alpha^2 \uptau^3 N  \Ex \|\grad f(\bar{x}^r)\|^2 + 4 \alpha^2   \uptau^2 N \sigma^2.
		\end{align}
	\end{lemma}
	\begin{proof}
	The proof extends the techniques from \cite[Lemma 8]{karimireddy2019scaffold}.	When $\uptau=1$, then $\phi_{i, 0}=x_i^r$ for all $i$ and 
		\begin{align*}
			\Ex 	\|\bm{\widehat{\Phi}}^r\|^2=  \Ex \|\x^r - \bar{\x}^r\|^2 \overset{\eqref{dhat_bound_x_z_dev}}{\leq}   4 \lambda \Ex \| \widehat{\d}^{r}\|^2 \leq 4   \Ex \| \widehat{\d}^{r}\|^2. 
		\end{align*}
		Now suppose that $\uptau \geq 2$. Then, using \eqref{local_ed_updates}, it holds that
		\[
		\begin{aligned}
			\Ex\|\phi^{r}_{i,t+1} -\bar{x}^r\|^2 &=\Ex \left\|
			\phi^{r}_{i,t}-\bar{x}^r-\alpha  \grad F_i(\phi^{r}_{i,t};\xi_{i,t})   -  \beta  y_i^{r} \right\|^2 \\
			 \overset{\eqref{noise_bound_eq}}&{\leq} \Ex \left \| \phi^{r}_{i,t}-\bar{x}^r-\alpha  \grad f_i(\phi^{r}_{i,t})   -  \beta  y_i^{r} \right\|^2+\alpha^2 \sigma^2 \\
			 \overset{\eqref{jensens_b}}&{\leq} \left(1+\tfrac{1}{\uptau-1}\right) \Ex\left\| \phi^{r}_{i,t}-\bar{x}^r \right\|^2+\uptau \Ex \left\|\alpha  \grad f_i(\phi^{r}_{i,t})   +  \beta  y_i^{r}  \right\|^2+\alpha^2 \sigma^2  \\
			 \overset{\eqref{z_def}}&{=} \left(1+\tfrac{1}{\uptau-1}\right) \Ex\| \phi^{r}_{i,t}-\bar{x}^r \|^2+\uptau \Ex \left\|\alpha  \grad f_i(\phi^{r}_{i,t})   - \alpha \grad f_i(\bar{x}^r)    + \beta  z_i^{r}  \right\|^2+\alpha^2 \sigma^2 \\
			\overset{\eqref{jensens_a}}&{\leq} \left(1+\tfrac{1}{\uptau-1}\right) \Ex\| \phi^{r}_{i,t}-\bar{x}^r \|^2+2  \alpha^2 \uptau \Ex \left\|  \grad f_i(\phi^{r}_{i,t})   -  \grad f_i(\bar{x}^r) \right\|^2   + 2 \uptau \beta^2 \|  z_i^{r}  \|^2 + \alpha^2 \sigma^2 \\
			\overset{\eqref{smooth_f_eq}}&{\leq} \left(1+\tfrac{1}{\uptau-1}+2\alpha^2 \uptau L^2  \right) \Ex\|\phi^{r}_{i,t}-\bar{x}^r \|^2+\tfrac{2}{ \uptau} \Ex \left\|z_i^r\right\|^2+ \alpha^2 \sigma^2 \\
			& \leq \left(1+\tfrac{5/4}{\uptau-1}\right) \Ex\| \phi^{r}_{i,t}-\bar{x}^r \|^2+\tfrac{2}{ \uptau} \Ex \|z_i^r\|^2+ \alpha^2 \sigma^2.
		\end{aligned}
		\]
	The second inequality uses \eqref{jensens_b} with $\theta=1-\frac{1}{\uptau}$.	The  last inequality holds for $2\alpha^2 \uptau L^2  \leq \frac{1}{4(\uptau-1)}$, which is satisfied if $\alpha \leq \frac{1}{2\sqrt{2} L \uptau}$.		Iterating the inequality above for $t=0,\dots,\uptau-1$:
		\[
		\begin{aligned}
			\Ex\|\phi^{r}_{i,t+1} -\bar{x}^r\|^2 & \leq \left(1+\tfrac{5/4}{(\uptau-1)}\right)^t \Ex \|x_i^r -\bar{x}^r\|^2+ \textstyle \sum\limits_{\ell=0}^{t}\left( \tfrac{2}{\uptau} \Ex \left\| z_i^r \right\|^2+ \alpha^2 \sigma^2 \right)\left(1+\frac{5/4}{(\uptau-1)}\right)^\ell \\
			& \leq \exp\left(\tfrac{(5/4) t}{\uptau-1}\right)  \Ex \|x_i^r -\bar{x}^r\|^2 + \textstyle \sum\limits_{\ell=0}^{t}\left( \tfrac{2}{\uptau} \Ex \left\| z_i^r \right\|^2+ \alpha^2 \sigma^2 \right) \exp\left(\tfrac{(5/4) \ell}{\uptau-1}\right) \\
			& \leq 4  \Ex \|x_i^r -\bar{x}^r\|^2+ \left(\tfrac{2}{\uptau} \Ex\left\| z_i^r \right\|^2+\alpha^2 \sigma^2 \right) 4 \uptau \\
			& \leq 
			 4( \Ex \|x_i^r -\bar{x}^r\|^2+ 2 \Ex \left\| z_i^r \right\|^2 ) +4  \alpha^2 \uptau \sigma^2,
		\end{aligned}
		\]
		where  in the second and third inequalities we used  $(1+\tfrac{a}{\uptau-1})^t \leq \exp(\frac{ a t}{\uptau-1})\leq \exp(a)$ for $t \leq \uptau -1$.	Summing over $i$ and $t$:
		\[
		\begin{aligned}
			\Ex	\|\bm{\widehat{\Phi}}^r\|^2 & \leq 4 \uptau \textstyle \sum\limits_{i=1}^N (\Ex \|x_i^r -\bar{x}^r\|^2+ 2 \Ex \left\| z_i^r \right\|^2 ) + 4 \alpha^2 \uptau^2 N \sigma^2 \\
			& \leq 4 \uptau \textstyle \sum\limits_{i=1}^N (\Ex \|x_i^r -\bar{x}^r\|^2+ 4 \Ex \|z_i^r -\bar{z}^r\|^2 + 4 \Ex \|\bar{z}^r\|^2 ) + 4 \alpha^2 \uptau^2 N \sigma^2 \\
			& \leq 16 \uptau (\Ex \|\x^r -\bar{\x}^r\|^2+  \Ex \|\z^r -\bar{\z}^r\|^2 ) +  16  \uptau N \Ex \|\bar{z}^r\|^2+ 4 \alpha^2 \uptau^2 N \sigma^2.
		\end{aligned}
		\]
	In the second inequality we used $\| z_i^r \|^2 \leq 2 \|z_i^r -\bar{z}^r\|^2 + 2 \Ex \|\bar{z}^r\|^2$, which follows form Jensen's inequality \eqref{jensens_a}.	The result follows by using \eqref{dhat_bound_x_z_dev} and \eqref{avg_z_grad} with $\beta=1/\uptau$.
	\end{proof}
	The next result measures the deviation from the average vector introduced in Lemma \ref{lemma:transf_avg_hat}.
	\begin{lemma}[\sc \small Deviation from average bound] \label{lemma:cons_inequ}
		\rm It holds that
		\begin{align}
			\textstyle \Ex \|\widehat{\d}^{r+1}\|^2 
			& \leq    \textstyle \left(\delta + \tfrac{512  \alpha^2 \uptau^2 L^2      }{ (1-\delta)}  \right)  \Ex \|\widehat{\d}^r\|^2  +  \textstyle \tfrac{128 \alpha^4 \uptau^4  L^2 N }{ (1-\delta)}  \Ex \|\grad f(\bar{x}^r)\|^2 
			+   \tfrac{4 \alpha^4 \uptau^2 L^2 \|	 \widehat{\B}^{-1}\|  N }{ (1-\delta)}  \textstyle \Ex \|  \sum\limits_{t}  \overline{\nabla f}(\bm{\Phi}^r_t) \|^2
			\nonumber \\
			& \quad +  \frac{32 \alpha^4 \uptau^3 L^2  N  \sigma^2 }{ (1-\delta)}      + \frac{4 \alpha^4 \uptau^3 L^2 \|	 \widehat{\B}^{-1}\|  \sigma^2 }{ (1-\delta)}  + 4 \alpha^2 \uptau   N \sigma^2,
			\label{cons_ineq_lemma}
		\end{align}
where $\delta$ and $\widehat{\B}$ were defined in Lemma \ref{lemma:transf_avg_hat}.
	\end{lemma} 
	\begin{proof}
		From now on we use the notation $ \sum\limits_{t} \equiv \sum\limits_{t=0}^{\uptau -1} $ and $ \sum\limits_{i} \equiv \sum\limits_{i=1}^{N} $. 		From \eqref{error_check_progress}, \eqref{s_V_def}, and \eqref{h_V_def}, we have
		\begin{align*}
			\textstyle \Ex_r \|\widehat{\d}^{r+1}\|^2 
			&=  \textstyle \Ex_r  \|\mathbf{\Delta} \widehat{\d}^{r} - \alpha \h^{r+1} \|^2  + \alpha^2 \textstyle \Ex_r  \|\s^r\|^2 - 2\alpha \textstyle \Ex_r  \langle \s^r, \mathbf{\Delta} \widehat{\d}^{r} - \alpha \h^{r+1} \rangle \\
			&=  \textstyle \Ex_r  \|\mathbf{\Delta} \widehat{\d}^{r} - \alpha \h^{r+1} \|^2  + \alpha^2 \textstyle \Ex_r  \|\s^r\|^2 + 2 \alpha^2 \textstyle \Ex_r  \langle \s^r,   \h^{r+1} \rangle \\
			& \leq  \textstyle \delta \Ex_r  \| \widehat{\d}^{r}\|^2 + \tfrac{\alpha^2}{ (1-\delta)} \textstyle \Ex_r \|\h^{r+1} \|^2  + \alpha^2 \textstyle \Ex_r  \|\s^r\|^2 + 2 \alpha^2 \textstyle \Ex_r  \langle \s^r,   \h^{r+1} \rangle,
		\end{align*}
		where $\delta = \|\bm{\Delta}\|=\sqrt{\lambda}<1$. 		The second line follows from the unbiased stochastic gradient condition \eqref{noise_bound_eq_mean} and the last step uses Jensen's inequality \eqref{jensens}.  Using $	2	\langle \s^r,   \h^{r+1} \rangle  \leq  \|\s^r\|^2 + \| \h^{r+1}\|^2$ and $1 \leq 1/(1-\delta)$ gives
		\begin{align}
			\textstyle \Ex_r \|\widehat{\d}^{r+1}\|^2 
			& \leq  \textstyle \delta \Ex_r  \| \widehat{\d}^{r}\|^2 + \tfrac{2\alpha^2}{ (1-\delta)} \textstyle \Ex_r \|\h^{r+1} \|^2  + 2\alpha^2 \textstyle \Ex_r  \|\s^r\|^2.
			\label{Edr_bound}
		\end{align}
		We now bound the terms  $\textstyle \Ex_r \| \s^r\|^2$ and $\textstyle \Ex_r \| \h^{r+1}\|^2$. 	From \eqref{vinv_noise_bound},	the noise term can be bounded by
		\begin{align*}
			\textstyle \Ex_r	\|	\s^r\|^2 
			& \leq 2  \textstyle \Ex_r \|\textstyle \sum\limits_{t}  \s_t^r\|^2  \overset{\eqref{jensens_a}}{\leq} 2 \uptau   N \sigma^2.
		\end{align*}
	 Using \eqref{vinv_grad_bound}, it holds that
		\begin{align*}
			 \textstyle \Ex_r \|	\h^{r+1}\|^2 
			& \leq 4 \textstyle \Ex_r \|\textstyle \sum\limits_{t}  \nabla \f(\bm{\Phi}^r_t) - \grad \f(\bar{\x}^r)\|^2 +  \uptau^2  \|	 \widehat{\B}^{-1}\| \Ex_r \|\grad \f(\bar{\x}^{r+1})-\grad \f(\bar{\x}^{r})\|^2 \\
			\overset{\eqref{smooth_f_eq}, \eqref{jensens_a}}&{\leq} 4 \uptau L^2   \textstyle\sum\limits_{t} \textstyle \Ex_r \|  \bm{\Phi}^r_t -  \bar{\x}^r  \|^2 + \uptau^2 L^2  \|	 \widehat{\B}^{-1}\| N \textstyle \Ex_r \|\bar{x}^{r+1}-\bar{x}^{r}\|^2 \\
			 \overset{\eqref{error_avg_progress}}&{=} 4 \uptau L^2   \textstyle\sum\limits_{t} \textstyle \Ex_r \|  \bm{\Phi}^r_t -  \bar{\x}^r  \|^2 
			+ \uptau^2 L^2  \|	 \widehat{\B}^{-1}\| N\textstyle \Ex_r  \| \alpha \sum\limits_{t} \big( \overline{\nabla f}(\bm{\Phi}^r_t) + \bar{s}_t^r  \big)\|^2 \\		
			& \leq 		 4 \uptau L^2  \textstyle \Ex_r \|\bm{\widehat{\Phi}}^r\|^2 + 2 \alpha^2 \uptau^2  L^2  \|	 \widehat{\B}^{-1}\| N \textstyle \Ex_r \|  \sum\limits_{t}  \overline{\nabla f}(\bm{\Phi}^r_t) \|^2 +2 \alpha^2 \uptau^2 L^2   \|	 \widehat{\B}^{-1}\| N \|  \sum\limits_{t} \bar{s}_t^r  \|^2 \\
			& \leq 		 4 \uptau L^2   \textstyle \Ex_r \|\bm{\widehat{\Phi}}^r\|^2 + 2 \alpha^2 \uptau^2 L^2  \|	 \widehat{\B}^{-1}\|  N \textstyle \Ex_r \|  \sum\limits_{t}  \overline{\nabla f}(\bm{\Phi}^r_t) \|^2 + 2 \alpha^2 \uptau^3 L^2 \|	 \widehat{\B}^{-1}\|  \sigma^2.
		\end{align*}
		Substituting the previous two bounds into \eqref{Edr_bound} and taking expectation gives
		\begin{align}
			\textstyle \Ex \|\widehat{\d}^{r+1}\|^2 
			& \leq  \textstyle \delta \Ex  \| \widehat{\d}^{r}\|^2 +  \textstyle \tfrac{8 \alpha^2 \uptau L^2  }{ (1-\delta)}  \textstyle \Ex  \|\bm{\widehat{\Phi}}^r\|^2 + \tfrac{4 \alpha^4  \uptau^2 L^2 \|	 \widehat{\B}^{-1}\|   N }{ (1-\delta)}  \textstyle \Ex \|  \sum\limits_{t}  \overline{\nabla f}(\bm{\Phi}^r_t) \|^2 \nonumber \\
			& \quad + \frac{4 \alpha^4 \uptau^3 L^2  \|	 \widehat{\B}^{-1}\| \sigma^2 }{ (1-\delta)}  + 4 \alpha^2 \uptau   N \sigma^2 .
		\end{align}
		Substituting \eqref{drift_bound} into the above inequality yields 
		\begin{align}
					\textstyle \Ex \|\widehat{\d}^{r+1}\|^2 
		& \leq  \textstyle \delta \Ex  \| \widehat{\d}^{r}\|^2 
		+  \textstyle \tfrac{8 \alpha^2 \uptau L^2  }{ (1-\delta)} ( 64  \uptau  \Ex \|\widehat{\d}^r\|^2 + 16   \alpha^2 \uptau^3 N  \Ex \|\grad f(\bar{x}^r)\|^2 + 4 \alpha^2   \uptau^2 N \sigma^2)
		\nonumber \\
		& \quad + \tfrac{4 \alpha^4  \uptau^2 L^2  \|	 \widehat{\B}^{-1}\| N }{ (1-\delta)}  \textstyle \Ex \|  \sum\limits_{t}  \overline{\nabla f}(\bm{\Phi}^r_t) \|^2 + \frac{4 \alpha^4 \uptau^3 L^2  \|	 \widehat{\B}^{-1}\| \sigma^2 }{ (1-\delta)}  + 2 \alpha^2 \uptau   N \sigma^2  \nonumber \\
			& =  \textstyle \left(\delta + \tfrac{512  \alpha^2 \uptau^2 L^2      }{ (1-\delta)}  \right)  \Ex \|\widehat{\d}^r\|^2  +  \textstyle \tfrac{128 \alpha^4 \uptau^4  L^2 N }{ (1-\delta)}  \Ex \|\grad f(\bar{x}^r)\|^2 
			+   \tfrac{4 \alpha^4 \uptau^2 L^2 \|	 \widehat{\B}^{-1}\|  N }{ (1-\delta)}  \textstyle \Ex \|  \sum\limits_{t}  \overline{\nabla f}(\bm{\Phi}^r_t) \|^2
			\nonumber \\
			& \quad +  \frac{32 \alpha^4 \uptau^3 L^2  N  \sigma^2 }{ (1-\delta)}      + \frac{4 \alpha^4 \uptau^3 L^2 \|	 \widehat{\B}^{-1}\|  \sigma^2 }{ (1-\delta)}  + 4 \alpha^2 \uptau   N \sigma^2 .
		\end{align}
		
	\end{proof}

	\subsection{Nonconvex case (Theorem \ref{thm:noncvx})}
The nonconvex proof begins with the  following bound for any $L$-smooth function $f$ \cite{nesterov2013introductory}:
	\begin{align}
		f(y) &\leq f(z)+ \langle \grad f(z),~ y-z \rangle+\tfrac{L}{2} \|y-z\|^2, \quad \forall~y,z \in \real^N. \label{bound:L_smooth_function}
	\end{align}
	Recall from \eqref{error_avg_progress} that $\bar{x}^{r+1}=\bar{x}^r -  \frac{\alpha}{N} \sum_{t=0}^{\uptau-1} \sum_{i=1}^N \big( \nabla f_i(\phi^r_{i,t}) + s^r_{i,t} \big)$.	Substituting $y=\bar{x}^{r+1}$ and $z=\bar{x}^{r}$ into  inequality \eqref{bound:L_smooth_function}, and taking conditional expectation, we get
	\begin{align}
		\textstyle	\Ex_r   f(\bar{x}^{r+1}) 
		&\leq f(\bar{x}^{r}) - \alpha \textstyle \Ex_r \big\langle \grad f(\bar{x}^{r}),  \textstyle \tfrac{1}{N} \sum\limits_{t} \sum\limits_{i} \big( \nabla f_i(\phi^r_{i,t}) + s^r_{i,t} \big) \big\rangle  +\tfrac{\alpha^2 L}{2} \Ex_r \| \textstyle \tfrac{1}{N} \sum\limits_{t} \sum\limits_{i} \big( \nabla f_i(\phi^r_{i,t}) + s^r_{i,t}) \|^2  \nonumber \\
		&\leq f(\bar{x}^{r}) - \alpha \textstyle \Ex_r  \textstyle   \big\langle \grad f(\bar{x}^{r}), \frac{1}{N}  \sum\limits_{t} \sum\limits_{i} \nabla f_i(\phi^r_{i,t})  \big\rangle 
		+ \alpha^2 \uptau L  \textstyle \sum\limits_{t}  \| \textstyle \tfrac{1}{N} \sum\limits_{i} \nabla f_i(\phi^r_{i,t})  \|^2 +  	\tfrac{\alpha^2  \uptau L \sigma^2}{N},
	\end{align}
	where $\Ex_{r}$ denote the expectation conditioned on the all iterates up to $r$. The second inequality holds by using  Jensen's inequality and  Assumption \ref{assump:noise}. 	Using $2 \langle a, b \rangle = \|a\|^2 + \|b\|^2 - \|a - b\|^2$, we have
	\begin{align}
		-   &  \textstyle   \big\langle \grad f(\bar{x}^{r}), \frac{1}{N}  \sum\limits_{t} \sum\limits_{i} \nabla f_i(\phi^r_{i,t})  \big\rangle \nonumber \\
		&=	-     \textstyle \sum\limits_{t}    \big\langle \grad f(\bar{x}^{r}), \frac{1}{N}  \sum\limits_{i} \nabla f_i(\phi^r_{i,t})  \big\rangle \nonumber \\
		&=  - \tfrac{\uptau}{2}  \|\grad f(\bar{x}^{r})\|^2 - \textstyle \tfrac{1}{2}  \sum\limits_{t}  \| \textstyle \tfrac{1}{N}  \sum\limits_{i} \nabla f_i(\phi^r_{i,t}) \|^2 + \tfrac{1}{2}  \sum\limits_{t}  \|\textstyle \tfrac{1}{N} \sum\limits_{i} \nabla f_i(\phi^r_{i,t})-\grad f(\bar{x}^{r}) \|^2 \nonumber \\
		&\leq   - \tfrac{\uptau}{2}  \|\grad f(\bar{x}^{r})\|^2 - \textstyle \tfrac{1}{2}  \sum\limits_{t}  \| \textstyle \tfrac{1}{N}  \sum\limits_{i} \nabla f_i(\phi^r_{i,t}) \|^2 + \tfrac{1}{2 N}  \sum\limits_{t} \sum\limits_{i} \|   \nabla f_i(\phi^r_{i,t})-\grad f_i(\bar{x}^{r}) \|^2  \nonumber \\
		&\leq   - \tfrac{\uptau}{2}  \|\grad f(\bar{x}^{r})\|^2 - \textstyle \tfrac{1}{2}  \sum\limits_{t}  \| \textstyle \tfrac{1}{N}  \sum\limits_{i} \nabla f_i(\phi^r_{i,t}) \|^2 + \tfrac{L^2}{2 N}   \| \bm{\widehat{\Phi}}^r \|^2,
	\end{align} 
where the second bound holds from Jensen's inequality \eqref{jensens}. 	Combining the last two equations and taking expectation yields
	\begin{align}
		\Ex   f(\bar{x}^{r+1})  &\leq \Ex f(\bar{x}^{r}) - \tfrac{\alpha \uptau}{2} \Ex \| \grad f(\bar{x}^{r})  \|^2 - \tfrac{\alpha}{2} (1-2\alpha \uptau L )  \textstyle \sum\limits_{t} \Ex \| \tfrac{1}{N}  \sum\limits_{i} \nabla f_i(\phi^r_{i,t})\|^2 \nonumber \\
		& \quad +\tfrac{\alpha L^2}{2 N}   \Ex \| \bm{\widehat{\Phi}}^r \|^2 +	\tfrac{  \alpha^2 \uptau L \sigma^2 }{ N}. \label{2bndsbd}
	\end{align} 
	Substituting the bound \eqref{drift_bound} into inequality \eqref{2bndsbd} and taking expectation yields
	\begin{equation} 	
		\begin{aligned}
			\Ex   f(\bar{x}^{r+1}) &\leq  \Ex  f(\bar{x}^{r}) - \tfrac{\alpha \uptau}{2} (1 - 16 \alpha^2 L^2 \uptau^2) \Ex  \| \grad f(\bar{x}^{r})  \|^2  - \tfrac{\alpha}{2 } (1- 2 \alpha L \uptau) \textstyle \sum\limits_{t}  \| \tfrac{1}{N} \sum\limits_{i} \nabla f_i(\phi^r_{i,t})\|^2 \nonumber \\ 
			& \quad +\tfrac{32 \alpha \uptau L^2   }{N}  \Ex  \|\widehat{\d}^r\|^2   + 2 \alpha^3 \uptau^2 L^2    \sigma^2 +	\tfrac{  \alpha^2 \uptau L \sigma^2 }{ N} .
		\end{aligned}	
	\end{equation}
	When  $\alpha \leq \frac{1}{4 \sqrt{2} \uptau L}$, we can upper bound the previous inequality by
	\begin{equation} \label{descent_inquality}	
		\begin{aligned}
			\Ex   f(\bar{x}^{r+1}) &\leq  \Ex  f(\bar{x}^{r}) - \tfrac{\alpha \uptau}{4}  \Ex  \| \grad f(\bar{x}^{r})  \|^2  - \tfrac{ \alpha}{4}  \textstyle  \sum\limits_{t} \|   \tfrac{1}{N} \sum\limits_{i} \nabla f_i(\phi^r_{i,t})\|^2 \nonumber \\ 
			& \quad +\tfrac{32 \alpha \uptau L^2   }{N}  \Ex  \|\widehat{\d}^r\|^2
			+ 2 \alpha^3 \uptau^2 L^2    \sigma^2 +	\tfrac{  \alpha^2 \uptau L \sigma^2 }{ N} .
		\end{aligned}	
	\end{equation}
		Rearranging  we get
		\begin{align} \label{grad_ineq_noncvx_proof}
		\cE_r 		  &\leq \tfrac{4}{\alpha \uptau} \left( \Ex  \tilde{f}(\bar{x}^{r}) -  \Ex   \tilde{f}(\bar{x}^{r+1}) \right)
			+\tfrac{128  L^2  }{N}  \Ex  \|\widehat{\d}^r\|^2 + 8 \alpha^2 \uptau L^2    \sigma^2 +	\tfrac{ 4 \alpha  L \sigma^2 }{ N},
		\end{align}	
		where  $\cE_r \define  \textstyle   \Ex  \| \grad f(\bar{x}^{r})  \|^2  +  \frac{1}{\uptau} \sum\limits_{t} \|  \tfrac{1}{N}  \sum\limits_{i} \nabla f_i(\phi^r_{i,t})\|^2$ and $\tilde{f}(\bar{x}^{r})\define f(\bar{x}^{r})-f^\star$. Averaging over $r=0,1,\ldots,R-1$  and using $-\tilde{f}(\bar{x}^{r})\leq 0$, it holds that
		\begin{align} \label{sum_grad_ineq_noncvx_proof}
			\frac{1}{R}	 \sum_{r=0}^{R-1} \cE_r  		 
			&\leq \frac{4   \tilde{f}(\bar{x}^{0})}{\alpha \uptau R}	  
			+ \frac{128  L^2    }{ N R}	 \sum_{r=0}^{R-1}   \Ex  \|\widehat{\d}^r\|^2 + 8 \alpha^2 \uptau L^2    \sigma^2 +	\tfrac{ 4 \alpha  L \sigma^2 }{ N}.
		\end{align}		
			We now bound the term $ \sum_{r=0}^{R-1}   \Ex  \|\widehat{\d}^r\|^2$. 	Using $
		\delta + \tfrac{512 \uptau^2 \alpha^2 L^2      }{ (1-\delta)}  \leq \frac{1+ \delta}{2} \define \bar{\delta}$, \ie, 
	\begin{align} \label{ss_noncvx_1}
		\alpha \leq \frac{1-\delta}{16 \sqrt{2} \uptau L }
	\end{align}
		 in \eqref{cons_ineq_lemma}, we have 
				\begin{align}
				\textstyle \Ex \|\widehat{\d}^{r+1}\|^2 
				& \leq   \bar{\delta}  \Ex \|\widehat{\d}^r\|^2  +  \textstyle \tfrac{128 \alpha^4 \uptau^4  L^2 N }{ (1-\delta)}  \Ex \|\grad f(\bar{x}^r)\|^2 
				+   \tfrac{4 \alpha^4 \uptau^2 L^2 \|	 \widehat{\B}^{-1}\|  N }{ (1-\delta)}  \textstyle \Ex \|  \sum\limits_{t}  \overline{\nabla f}(\bm{\Phi}^r_t) \|^2
				\nonumber \\
				& \quad +  \frac{32 \alpha^4 \uptau^3 L^2  N  \sigma^2 }{ (1-\delta)}      + \frac{4 \alpha^4 \uptau^3 L^2 \|	 \widehat{\B}^{-1}\|  \sigma^2 }{ (1-\delta)}  + 4 \alpha^2 \uptau   N \sigma^2 
				\nonumber \\
				& \leq   \bar{\delta}  \Ex \|\widehat{\d}^r\|^2  +  \textstyle \tfrac{128 \alpha^4 \uptau^4  L^2  \|	 \widehat{\B}^{-1}\| N }{ (1-\delta)}  \cE_r
				+  5 \alpha^2 \uptau   N \sigma^2.  
			\end{align}
		The last step uses Jensen's inequality and $ \frac{36 \alpha^2 \uptau^2 L^2 \|	 \widehat{\B}^{-1}\| }{ (1-\delta)}        \leq  1   $, \ie, 
		\begin{align}
			\label{ss_noncvx_2}
			\alpha \leq \frac{ \sqrt{(1-\delta)/\|	 \widehat{\B}^{-1}\|}}{6 \uptau L  } .
		\end{align}
			Iterating gives 
		\begin{align} \label{cons_ineq_noncvx_proof}
			\Ex \|\widehat{\d}^{r}\|^2 &\leq  
			\bar{\delta}^r  \|\widehat{\d}^{0}\|^2   
			+ \tfrac{128 \alpha^4 \uptau^4  L^2  \|	 \widehat{\B}^{-1}\| N }{ (1-\delta)}  \sum_{\ell=0}^{r-1} \left( \tfrac{1+\delta}{2} \right)^{r-1-\ell}     \cE_\ell   +  \frac{10 \alpha^2 \uptau   N \sigma^2}{(1-\delta)}.
		\end{align}  
		Averaging over $r=1,\dots,R$ and using \eqref{ss_noncvx_2}, it holds that
		\begin{align} 
			\frac{1}{R} \sum_{r=1}^R	\Ex \|\widehat{\d}^{r}\|^2 	  &\leq  
			\frac{2 \|\widehat{\d}^{0}\|^2}{(1-\delta)R} 
			+    \tfrac{4 \alpha^2 \uptau^2   N }{  R} \sum_{r=1}^R  \sum_{\ell=0}^{r-1} \left( \tfrac{1+\delta}{2} \right)^{r-1-\ell}   \cE_\ell    +   \frac{10 \alpha^2 \uptau   N \sigma^2}{(1-\delta)} \nonumber \\
			& \leq  
			\frac{2 \|\widehat{\d}^{0}\|^2}{(1-\delta)R} 
			+  \tfrac{8 \alpha^2 \uptau^2   N }{(1-\delta)  R}   \sum_{r=0}^{R-1}    \cE_r   +  \frac{10 \alpha^2 \uptau   N \sigma^2}{(1-\delta)}.
		\end{align}  
		Adding $\frac{\|\widehat{\d}^{0}\|^2}{R}$ to both sides of the previous inequality and using $\frac{\|\widehat{\d}^{0}\|^2}{R} \leq \frac{\|\widehat{\d}^{0}\|^2}{(1-\delta)R}$, we get
		\begin{align}  \label{bound_noncvx_cons_final}
			\frac{1}{R} \sum_{r=0}^{R-1}	\Ex \|\widehat{\d}^{r}\|^2  	  &\leq  
			\frac{3 \|\widehat{\d}^{0}\|^2}{(1-\delta)R} 
			+   \frac{8 \alpha^2 \uptau^2   N }{(1-\delta)  R}   \sum_{r=0}^{R-1}     \cE_r +    \frac{10 \alpha^2 \uptau   N \sigma^2}{(1-\delta)}.
		\end{align}  
	Substituting inequality \eqref{bound_noncvx_cons_final} into \eqref{sum_grad_ineq_noncvx_proof} and rearranging, we obtain 
		\begin{align} 
			\left(1-\tfrac{1024 \alpha^2 \uptau^2 L^2   }{(1-\delta)  }  \right)	\frac{1}{R}	 \sum_{r=0}^{R-1} \cE_r 		 
			&\leq \frac{4   \tilde{f}(\bar{x}^{0})}{\alpha \uptau R}	  
			+ \frac{384 L^2    \|\widehat{\d}^{0}\|^2}{ (1-\delta) N R} \nonumber  \\
			& \quad + \frac{1280 \alpha^2 \uptau  L^2  \sigma^2}{(1-\delta)}
			+ 8  \alpha^2 \uptau L^2    \sigma^2 +	\frac{ 4 \alpha  L \sigma^2 }{ N}.
		\end{align}
		If we set 
		\begin{align} \label{ss_noncvx_3}
			\frac{1}{2} \leq 1-\tfrac{1024 \alpha^2 \uptau^2 L^2   }{(1-\delta)  }   \Longrightarrow \alpha \leq \frac{\sqrt{1-\delta}}{32 \sqrt{2}  \uptau L }, 
		\end{align}	
	then it holds that
		\begin{align} 
			\frac{1}{R}	 \sum_{r=0}^{R-1} \cE_r 		 
		&\leq \frac{8   \tilde{f}(\bar{x}^{0})}{\alpha \uptau R}	  
		+ \frac{768 L^2    \|\widehat{\d}^{0}\|^2}{ (1-\delta) N R} + \frac{2560 \alpha^2 \uptau  L^2  \sigma^2}{(1-\delta)}
		+ 16  \alpha^2 \uptau L^2    \sigma^2 +	\frac{ 8 \alpha  L \sigma^2 }{ N}.
		\end{align}
	Now from  \eqref{dhat_bound}, we can bound $\|\widehat{\d}^{0} \|^2$ by
		\begin{align} \label{0_d_hat_bound}
			\|\widehat{\d}^{0}\|^2 &\leq  \|	\widehat{\mathbf{\Lambda}}^{-\frac{1}{2}} \widehat{\Q}\tran\x^{0}\|^2 + \|\widehat{\B}^{-\frac{1}{2}} \widehat{\Q}\tran\z^{0}\|^2  \nonumber \\
			&\leq   \frac{1}{\underline{\lambda}} \|\x^{0}-\bar{\x}^0\|^2 + 2 \|\widehat{\B}^{-1}\| \|\y^{0}-\bar{\y}^0\|^2 + 2 \alpha^2 \uptau^2 \|\widehat{\B}^{-1}\| \|\grad \f(\bar{\x}^0) - \one \otimes \grad f(\bar{\x}^0) \|^2 \nonumber \\
			&\leq   \frac{1}{\underline{\lambda}} \|\x^{0}-\bar{\x}^0\|^2 + \frac{2 \alpha^2 \uptau^2}{1-\lambda} \|\grad \f(\bar{\x}^0) - \one \otimes \grad f(\bar{\x}^0) \|^2,
		\end{align}
	where the second inequality we used \eqref{z_def} ($\z^0 = \y^0 +\alpha \uptau  \grad \f(\bar{\x}^0) $) and Jensen's inequality. The last inequality holds under initialization $\y^0=\zero$.
		We conclude that 
			\begin{align} \label{last_noncvx_thm_proof}
			\frac{1}{R}	 \sum_{r=0}^{R-1} \cE_r 		 
			&\leq \frac{8   \tilde{f}(\bar{x}^{0})}{\alpha \uptau R}	  
			+ \frac{768 L^2      }{ (1-\delta) \underline{\lambda} N R}
			  \|\x^{0}-\bar{\x}^0\|^2  + \frac{1536 \alpha^2 \uptau^2 L^2    \varsigma_0^2 }{ (1-\delta) (1-\lambda)  R} \nonumber \\
			 & \quad +  \frac{2560 \alpha^2 \uptau  L^2  \sigma^2}{(1-\delta)}
			+ 16 \alpha^2 \uptau L^2    \sigma^2 +	\frac{ 8 \alpha  L \sigma^2 }{ N},
		\end{align}
		where $\varsigma_0^2 = \frac{1}{N} \|\grad \f(\bar{\x}^0) - \one \otimes \grad f(\bar{\x}^0) \|^2$.
		
\subsection{Convex cases (Theorem \ref{thm:cvx})}
	 Recall from \eqref{error_avg_progress} that $	\bar{x}^{r+1}=\bar{x}^r -  \frac{\alpha}{N} \sum_{t=0}^{\uptau-1} \sum_{i=1}^N \big( \nabla f_i(\phi^r_{i,t}) + s^r_{i,t} \big)$. Thus, it holds that
	\begin{align} 
		\textstyle	\Ex_{r} \|\bar{x}^{r+1}-x^\star \|^2
		&=  \|\bar{x}^r-x^\star\|^2
		- \tfrac{2\alpha}{N}\textstyle  \Ex_r  \langle (\bar{x}^r-x^\star),  \textstyle \sum\limits_{t} \sum\limits_{i} \big( \nabla f_i(\phi^r_{i,t}) + s^r_{i,t} \big)  \rangle 
		\nonumber 	\\
		& \quad + \alpha^2 \textstyle \Ex_{r} \| \tfrac{1}{N} \textstyle \sum\limits_{t} \sum\limits_{i} \big( \nabla f_i(\phi^r_{i,t}) + s^r_{i,t} \big)\|^2 \nonumber \\
		&\leq  \|\bar{x}^r-x^\star\|^2
		- \tfrac{2\alpha}{N}\textstyle  \Ex_r \langle  (\bar{x}^r-x^\star),  \textstyle  \sum\limits_{i }   \sum\limits_{t}  \grad f_i(\phi^{r}_{i,t})  \rangle 
		\nonumber 	\\
		& \quad + 2 \alpha^2 \textstyle \Ex_{r} \| \tfrac{1}{N} \textstyle \sum\limits_{i}   \sum\limits_{t}  \grad f_i(\phi^{r}_{i,t})\|^2 + 2 \alpha^2 \Ex \|\tfrac{1}{N} \textstyle \sum\limits_{t} \sum\limits_{i}  s^r_{i,t}\|^2 \nonumber \\ 
		&\leq  \|\bar{x}^r-x^\star\|^2
		- \tfrac{2\alpha}{N} \textstyle \Ex_r \langle  (\bar{x}^r-x^\star),  \textstyle  \sum\limits_{i }   \sum\limits_{t}  \grad f_i(\phi^{r}_{i,t})  \rangle 
		\nonumber 	\\
		& \quad + 2 \alpha^2  \textstyle \Ex_{r} \| \tfrac{1}{N} \textstyle \sum\limits_{i}   \sum\limits_{t}  \grad f_i(\phi^{r}_{i,t})\|^2 + \frac{2 \alpha^2 \uptau \sigma^2 }{N}. \label{avg_progress_cvx0}
	\end{align} 
	The first inequality we used the zero mean condition \eqref{noise_bound_eq_mean} and Jensen's inequality. 	The second  inequality holds from bounded noise variance condition  from Assumption \ref{assump:noise} and Jensen's inequality. We now bound the cross term by using the bound \cite[Lemma 5]{karimireddy2019scaffold}:
	\begin{align} \label{cvx_bnd}
		\big\langle (z-y), \grad g(x) \big\rangle \geq g(z)-g(y)+\frac{\mu}{4} \|y-z\|^2 - L \|z-x\|^2, \quad \forall~x,y,z \in \real^N,
	\end{align}
 for any $L$-smooth and $\mu$-strongly convex function $g$. 	Using \eqref{cvx_bnd}, we can bound the cross term as follows
	\begin{align*}
	 -\tfrac{2\alpha}{N}    \textstyle \sum\limits_{i}   \sum\limits_{t}  \left\langle (\bar{x}^r-x^\star), \grad f_i(\phi^{r}_{i,t}) \right\rangle 
		& \leq \tfrac{2\alpha}{N} \textstyle  \sum\limits_{i}  \sum\limits_{t}   \left(f_i(x^\star)-f_i(\bar{x}^r)-\frac{\mu}{4} \|\bar{x}^r-x^\star\|^2 + L  \|\bar{x}^r-\phi^{r}_{i,t}\|^2 \right) \nonumber \\
		& = -2 \alpha \uptau \left(f(\bar{x}^r)-f(x^\star)+\tfrac{\mu}{4} \|\bar{x}^r-x^\star\|^2 \right) +   \tfrac{2\alpha L }{N} \|\bm{\widehat{\Phi}}^r\|^2,
	\end{align*}
where $
\|\bm{\widehat{\Phi}}^r\|^2 = 	  \sum_{i =1}^N  \sum_{t=0}^{\uptau-1}   \|\phi^{r}_{i,t}-\bar{x}^r\|^2$. 	Substituting the previous  bound into  \eqref{avg_progress_cvx0} and taking expectation gives
\begin{align} \label{avg_progress_cvx00}
	\Ex \|\bar{x}^{r+1}-x^\star \|^2 
	&\leq  (1-\tfrac{\mu \uptau \alpha }{2} ) \Ex \|\bar{x}^r-x^\star\|^2 - 2 \alpha \uptau \Ex \big(f(\bar{x}^r)-f(x^\star) \big)     \nonumber \\ 
	& \quad +    \tfrac{2\alpha L}{N} \Ex \|\bm{\widehat{\Phi}}^r\|^2 + 
	2 \alpha^2  \textstyle \Ex \|  \sum\limits_{t}  \overline{\nabla f}(\bm{\Phi}^r_t) \|^2  + \tfrac{2 \alpha^2 \uptau \sigma^2 }{N},
\end{align} 
where $ \overline{\nabla f}(\bm{\Phi}^r_t)= \tfrac{1}{N} \textstyle \sum\limits_{i}     \grad f_i(\phi^{r}_{i,t})$.	Note  that
	\begin{align}
 \textstyle  \|  \sum\limits_{t}  \overline{\nabla f}(\bm{\Phi}^r_t) \|^2=	\| \tfrac{1}{N} \textstyle \sum\limits_{i}   \sum\limits_{t}  \grad f_i(\phi^{r}_{i,t})\|^2 &=	\| \tfrac{1}{N} \textstyle \sum\limits_{i}   \sum\limits_{t}  \grad f_i(\phi_i^{t})-\grad f_i(\bar{x}^r)+\grad f_i(\bar{x}^r) \|^2  \nonumber \\
		& \leq  2  \left\| \tfrac{1}{N} \textstyle \sum\limits_{i}   \sum\limits_{t}  \grad f_i(\phi^{r}_{i,t})-\grad f_i(\bar{x}^r) \right\|^2 +2  \uptau^2  \left\| \tfrac{1}{N} \textstyle \sum\limits_{i}       \grad f_i(\bar{x}^r) \right\|^2 \nonumber \\
		& \leq  \tfrac{2 \uptau}{N} \textstyle \sum\limits_{i}   \sum\limits_{t}  \| \grad f_i(\phi^{r}_{i,t})-\grad f_i(\bar{x}^r)\|^2 +2  \uptau^2 \| \grad f(\bar{x}^r)\|^2 \nonumber \\
		& \leq  \tfrac{2 \uptau L^2 }{N}  \|\bm{\widehat{\Phi}}^r\|^2 +2  \uptau^2 \| \grad f(\bar{x}^r)\|^2.
		\label{grad_drift_bound}
	\end{align}
	Substituting the previous  bound into  \eqref{avg_progress_cvx00}  gives 	
	\begin{align} 
		\Ex \|\bar{x}^{r+1}-x^\star \|^2 
		&\leq  (1-\tfrac{\mu \uptau \alpha }{2} ) \Ex \|\bar{x}^r-x^\star\|^2 - 2 \alpha \uptau \Ex \big(f(\bar{x}^r)-f(x^\star) \big)     \nonumber \\ 
		& \quad +4 \alpha^2 \uptau^2  \Ex \|\grad f(\bar{x}^r)\|^2  +  \tfrac{2 \alpha L}{N}  (1+  2 \alpha \uptau L) \Ex \|\bm{\widehat{\Phi}}^r\|^2 + \frac{2 \alpha^2 \uptau \sigma^2 }{N} \nonumber \\
			&\leq  (1-\tfrac{\mu \uptau \alpha }{2} ) \Ex \|\bar{x}^r-x^\star\|^2 - 2 \alpha \uptau \Ex \big(f(\bar{x}^r)-f(x^\star) \big)     \nonumber \\ 
		& \quad +4 \alpha^2 \uptau^2  \Ex \|\grad f(\bar{x}^r)\|^2  +  \tfrac{3 \alpha L}{N}   \Ex \|\bm{\widehat{\Phi}}^r\|^2 + \frac{2 \alpha^2 \uptau \sigma^2 }{N}.
	\end{align} 
The last inequality holds for $1+  2 \alpha \uptau L \leq 3/2$ or
\begin{align} \label{ss_cvx_1}
	\alpha \leq \frac{1}{4 \uptau L}.
\end{align}
 	 Substituting the  bound  \eqref{drift_bound} into the above inequality gives 
	\begin{align} 
			\Ex \|\bar{x}^{r+1}-x^\star \|^2 
		&\leq  (1-\tfrac{\mu \uptau \alpha }{2} ) \Ex \|\bar{x}^r-x^\star\|^2 - 2 \alpha \uptau \Ex \big(f(\bar{x}^r)-f(x^\star) \big)     \nonumber \\ 
		& \quad + (4 \alpha^2 \uptau^2 + 48 \alpha^3 \uptau^3  L  )  \Ex \|\grad f(\bar{x}^r)\|^2  +     \tfrac{192 \alpha \uptau L }{N}  \Ex \|\widehat{\d}^r\|^2+         12 \alpha^3 \uptau^2 L^2 \sigma^2 + \frac{2 \alpha^2 \uptau \sigma^2 }{N} \nonumber \\
		&\leq  (1-\tfrac{\mu \uptau \alpha }{2} ) \Ex \|\bar{x}^r-x^\star\|^2 - (2 \alpha \uptau -8  \alpha^2 \uptau^2 L  - 96 \alpha^3 \uptau^3  L^2   ) \Ex \big(f(\bar{x}^r)-f(x^\star) \big)     \nonumber \\ 
		& \quad   +     \tfrac{192 \alpha \uptau L  }{N}  \Ex \|\widehat{\d}^r\|^2+         12 \alpha^3 \uptau^2 L^2 \sigma^2 + \frac{2 \alpha^2 \uptau \sigma^2 }{N}.
	 \nonumber
	\end{align} 
In the last step, we used  $\|\grad f(\bar{x}^r)\|^2 \leq 2 L (f(\bar{x}^r)-f(x^\star))$, which is satisfied under Assumption \ref{assump:cvx} \cite{nesterov2013introductory}. Using $(2 \alpha \uptau -8 L \alpha^2 \uptau^2 - 96 \alpha^3 \uptau^3  L^2   ) \leq \alpha \uptau$, \ie,
\begin{align} \label{ss_cvx_2}
	\alpha \leq \frac{1}{32 L \uptau},
\end{align}
we get
	\begin{align} 
	\Ex \|\bar{x}^{r+1}-x^\star \|^2 
	&\leq  (1-\tfrac{\mu \uptau \alpha }{2} ) \Ex  \|\bar{x}^r-x^\star\|^2 -\alpha \uptau  \Ex [f(\bar{x}^r)-f(x^\star)]      \nonumber \\ 
	& \quad   + \tfrac{192 \alpha \uptau L  }{N}  \Ex \|\widehat{\d}^r\|^2  +    12 \alpha^3 \uptau^2  L      \sigma^2 + \frac{2\alpha^2 \uptau \sigma^2 }{N}.
	\label{avg_bound_cvx_proof}
\end{align} 
 	Starting from \eqref{cons_ineq_lemma}, we have
		\begin{align}
			\textstyle \Ex \|\widehat{\d}^{r+1}\|^2 
			& \leq         \textstyle \left(\delta + \tfrac{512  \alpha^2 \uptau^2 L^2      }{ (1-\delta)}  \right)  \Ex \|\widehat{\d}^r\|^2  +  \textstyle \tfrac{128 \alpha^4 \uptau^4  L^2 N }{ (1-\delta)}  \Ex \|\grad f(\bar{x}^r)\|^2 
			+   \tfrac{4 \alpha^4 \uptau^2 L^2 \|	 \widehat{\B}^{-1}\|  N }{ (1-\delta)}  \textstyle \Ex \|  \sum\limits_{t}  \overline{\nabla f}(\bm{\Phi}^r_t) \|^2
			\nonumber \\
			& \quad +  \frac{32 \alpha^4 \uptau^3 L^2  N  \sigma^2 }{ (1-\delta)}      + \frac{4 \alpha^4 \uptau^3 L^2 \|	 \widehat{\B}^{-1}\|  \sigma^2 }{ (1-\delta)}  + 4 \alpha^2 \uptau   N \sigma^2 
			\nonumber \\ 
				\overset{\eqref{grad_drift_bound}}&{\leq}  
			\textstyle \left(\delta + \tfrac{512  \alpha^2 \uptau^2 L^2      }{ (1-\delta)}  \right)   \Ex \|\widehat{\d}^r\|^2  +  \textstyle \tfrac{136 \alpha^4 \uptau^4 L^2 \|	 \widehat{\B}^{-1}\| N  }{ (1-\delta)}  \Ex \|\grad f(\bar{x}^r)\|^2 
			+   \tfrac{8 \alpha^4 \uptau^3  L^4 \|	 \widehat{\B}^{-1}\|   }{ (1-\delta)}\Ex \|\bm{\widehat{\Phi}}^r\|^2 
			\nonumber \\
			& \quad   + \frac{32 \alpha^4 \uptau^3 L^2  N  \sigma^2 }{ (1-\delta)}      + \frac{4 \alpha^4 \uptau^3 L^2 \|	 \widehat{\B}^{-1}\|  \sigma^2 }{ (1-\delta)}  + 4 \alpha^2 \uptau   N \sigma^2 \nonumber \\
				& \leq 
			\textstyle \left(\delta + \tfrac{512  \alpha^2 \uptau^2 L^2      }{ (1-\delta)}  \right)   \Ex \|\widehat{\d}^r\|^2  +  \textstyle \tfrac{136 \alpha^4 \uptau^4 L^2 \|	 \widehat{\B}^{-1}\| N  }{ (1-\delta)}  \Ex \|\grad f(\bar{x}^r)\|^2 
			+   \tfrac{8 \alpha^4 \uptau^3  L^4 \|	 \widehat{\B}^{-1}\|   }{ (1-\delta)}\Ex \|\bm{\widehat{\Phi}}^r\|^2 
			\nonumber \\
			& \quad   + 5 \alpha^2 \uptau   N \sigma^2,
		\end{align}
where the last inequality holds when	$ \frac{36 \alpha^2 \uptau^2 L^2 \|	 \widehat{\B}^{-1}\| }{ (1-\delta)}        \leq  1   $, \ie,
\begin{align} \label{ss_cvx_3}
	\alpha \leq  \frac{\sqrt{(1-\delta)/ \|	 \widehat{\B}^{-1}\| }}{6  \uptau L}.
\end{align}
	Substituting the bound \eqref{drift_bound} into the above inequality yields 
			\begin{align}
		\textstyle \Ex \|\widehat{\d}^{r+1}\|^2 
		& \leq      \left(\delta +\tfrac{512 \uptau^2 \alpha^2 L^2      }{ (1-\delta)}  +  \tfrac{512 \alpha^4 \uptau^4   L^4 \|	 \widehat{\B}^{-1}\|   }{ (1-\delta)}  \right)  \Ex \|\widehat{\d}^r\|^2
		\nonumber  \\
		& \quad  +  \textstyle \tfrac{136 \alpha^4 \uptau^4  L^2   \|	 \widehat{\B}^{-1}\|  N}{ (1-\delta)}  \Ex \|\grad f(\bar{x}^r)\|^2
		  +  \textstyle \tfrac{128 \alpha^6 \uptau^6 L^4  \|	 \widehat{\B}^{-1}\|  N  }{ (1-\delta)}  \Ex \|\grad f(\bar{x}^r)\|^2 
		\nonumber \\
		& \quad  +     \frac{32 \alpha^6 \uptau^5  L^4  \|	 \widehat{\B}^{-1}\|   N  \sigma^2 }{ (1-\delta)}      + 5 \alpha^2 \uptau   N \sigma^2 \nonumber \\
			\overset{\eqref{ss_cvx_2}}&{\leq}       \left(\delta + \tfrac{512 \uptau^2 \alpha^2 L^2      }{ (1-\delta)}  +  \tfrac{512 \alpha^4 \uptau^4   L^4 \|	 \widehat{\B}^{-1}\|   }{ (1-\delta)}  \right) \Ex \|\widehat{\d}^r\|^2 +  \textstyle \tfrac{137 \alpha^4 \uptau^4  L^2   \|	 \widehat{\B}^{-1}\|  N}{ (1-\delta)}  \Ex \|\grad f(\bar{x}^r)\|^2
		\nonumber  \\
		& \quad  
	 +     \frac{ \alpha^4 \uptau^3  L^2  \|	 \widehat{\B}^{-1}\|   N  \sigma^2 }{ (1-\delta)}      + 5 \alpha^2 \uptau   N \sigma^2
	 \nonumber \\
	 \overset{ \eqref{ss_cvx_3}}&{\leq}       \left(\delta +\tfrac{527 \uptau^2 \alpha^2 L^2      }{ (1-\delta)}   \right)  \Ex \|\widehat{\d}^r\|^2 +  \textstyle (4 \alpha^2 \uptau^2   N)  \Ex \|\grad f(\bar{x}^r)\|^2
   + 6 \alpha^2 \uptau   N \sigma^2.
	\end{align} 
	Using the condition $\delta +\tfrac{527 \uptau^2 \alpha^2 L^2      }{ (1-\delta)}   \leq \frac{1+\delta}{2} \define \bar{\delta}$, which holds when
	\begin{align} \label{ss_cvx_4}
		\alpha \leq \frac{1-\delta}{33 \uptau L},
	\end{align}
	 the right hand side can be upper bounded by
		\begin{align}
	\textstyle \Ex \|\widehat{\d}^{r+1}\|^2 
	& \leq     \bar{\delta}  \Ex \|\widehat{\d}^r\|^2 +  \textstyle (4 \alpha^2 \uptau^2   N)    \Ex \|\grad f(\bar{x}^r)\|^2
	+ 6 \alpha^2 \uptau   N \sigma^2.
	\label{cvx_dhat_bound}
\end{align}

	\subsubsection{Convex case $\mu=0$}
	For convex but not strongly-convex, we have $\mu=0$ and equation \eqref{avg_bound_cvx_proof} becomes
	\begin{align} 
		\Ex \|\bar{x}^{r+1}-x^\star \|^2 
		&\leq    \Ex  \|\bar{x}^r-x^\star\|^2 -\alpha \uptau  \Ex [f(\bar{x}^r)-f(x^\star)]    + \tfrac{192 \alpha \uptau L   }{N}  \Ex \|\widehat{\d}^r\|^2    \nonumber \\ 
		& \quad   +    12 \alpha^3 \uptau^2 L       \sigma^2 + \frac{2 \alpha^2  \uptau \sigma^2 }{N}.
	\end{align}
	Rearranging gives
\begin{align}
	\cE_r 		  &\leq \tfrac{1}{\alpha \uptau} \left( \Ex \|\bar{x}^{r}-x^\star \|^2 - \Ex \|\bar{x}^{r+1}-x^\star \|^2 \right)
	+\tfrac{192  L  }{N}  \Ex  \|\widehat{\d}^r\|^2 + 12 \alpha^2 \uptau L    \sigma^2 +	\tfrac{ 2 \alpha   \sigma^2 }{ N},
\end{align}	
where $\cE_r \define  \textstyle    \Ex [f(\bar{x}^r)-f(x^\star)]  $. The above equation is similar to the nonconvex equation \eqref{grad_ineq_noncvx_proof} with the only  difference being the error criteria (and constants).  Therefore, the analysis follows using similar arguments used in the nonconvex case. Averaging $r=0,1,\ldots,R-1$  and using $-\tilde{f}(\bar{x}^{r})\leq 0$, it holds that 
\begin{align} \label{sum_grad_ineq_cvx_proof}
	\frac{1}{R}	 \sum_{r=0}^{R-1} \cE_r  		 
	&\leq \frac{    \|\bar{x}^{0}-x^\star \|^2}{\alpha \uptau R}	  
	+ \tfrac{192  L   }{N R}	 \sum_{r=0}^{R-1}   \Ex  \|\widehat{\d}^r\|^2 +  12 \alpha^2 L   \uptau \sigma^2 +	\tfrac{ 2 \alpha   \sigma^2 }{ N}.
\end{align}		
Plugging $\|\grad f(\bar{x}^r)\|^2 \leq 2 L (f(\bar{x}^r)-f(x^\star))$ into \eqref{cvx_dhat_bound} gives
		\begin{align}
		\textstyle \Ex \|\widehat{\d}^{r+1}\|^2
		& \leq    \bar{\delta} \Ex \|\widehat{\d}^r\|^2  +  (8 \alpha^2 \uptau^2  L N)     \cE_r
		+ 6 \alpha^2 \uptau   N \sigma^2.
	\end{align}
	Iterating and averaging over $r=1,\dots,R$
\begin{align} 
	\frac{1}{R} \sum_{r=1}^R	\Ex \|\widehat{\d}^{r}\|^2 	  &\leq  
	\frac{2 \|\widehat{\d}^{0}\|^2}{(1-\delta)R} 
	+    \tfrac{8 \alpha^2 \uptau^2  L N  }{ R}  \sum_{r=1}^R  \sum_{\ell=0}^{r-1} \left( \tfrac{1+\delta}{2} \right)^{r-1-\ell}   \cE_\ell    +  \frac{12 \alpha^2 \uptau  N \sigma^2}{(1-\delta)} \nonumber \\
	& \leq  
	\frac{2 \|\widehat{\d}^{0}\|^2}{(1-\delta)R} 
	+  \tfrac{16 \alpha^2 \uptau^2  L N}{(1-\delta)  R}  \sum_{r=0}^{R-1}    \cE_r   +  \frac{12 \alpha^2 \uptau  N \sigma^2}{(1-\delta)} .
\end{align}  
Adding $\frac{\|\widehat{\d}^{0}\|^2}{R}$ to both sides of the previous inequality and using $\frac{\|\widehat{\d}^{0}\|^2}{R} \leq \frac{\|\widehat{\d}^{0}\|^2}{(1-\delta)R}$, we get
\begin{align}  \label{bound_cvx_cons_final}
	\frac{1}{R} \sum_{r=0}^{R-1}	\Ex \|\widehat{\d}^{r}\|^2  	  &\leq  
	\frac{3 \|\widehat{\d}^{0}\|^2}{(1-\delta)R} 
	+   \frac{16 \alpha^2 \uptau^2  L N}{(1-\delta)  R}  \sum_{r=0}^{R-1}    \cE_r   +  \frac{12 \alpha^2 \uptau  N \sigma^2}{(1-\delta)}.
\end{align}   
Substituting inequality \eqref{bound_cvx_cons_final} into \eqref{sum_grad_ineq_cvx_proof} and rearranging, we obtain 
\begin{align} 
	\left(1-\tfrac{3072 \alpha^2  \uptau^2 L^2    }{(1-\delta)  } \right)	\frac{1}{R}	 \sum_{r=0}^{R-1} \cE_r 		 
	&\leq \frac{    \|\bar{x}^{0}-x^\star \|^2}{\alpha \uptau R}	  
	+ \frac{576 L     \|\widehat{\d}^{0}\|^2}{ (1-\delta) N R} \nonumber  \\
	& \quad + \frac{2304 \alpha^2 \uptau   L \sigma^2}{(1-\delta)}
+  12 \alpha^2 \uptau L    \sigma^2 +	\tfrac{ 2 \alpha   \sigma^2 }{ N}.
\end{align}
If we set 
\begin{align} \label{ss_cvx_5_cvx}
	\frac{1}{2} \leq 1-\frac{3072 \alpha^2  \uptau^2 L^2    }{(1-\delta)  }, \quad \Rightarrow \quad \alpha \leq \frac{\sqrt{1-\delta}}{100  \uptau L },
\end{align}	
then it holds that 
\begin{align} 
	\frac{1}{R}	 \sum_{r=0}^{R-1} \cE_r 		 
	&\leq  \frac{  2  \|\bar{x}^{0}-x^\star \|^2}{\alpha \uptau R}	  
	+ \frac{1152 L     \|\widehat{\d}^{0}\|^2}{ (1-\delta) N R}
	+  \frac{4608 \alpha^2 \uptau  L  \sigma^2}{(1-\delta)}
	+   24 \alpha^2 \uptau L    \sigma^2 +	\frac{ 4 \alpha   \sigma^2 }{ N}.
\end{align}
Plugging the bound \eqref{0_d_hat_bound} we get
\begin{align} \label{last_cvx_thm_proof}
	\frac{1}{R}	 \sum_{r=0}^{R-1} \cE_r 		 
	&\leq  \frac{  2  \|\bar{x}^{0}-x^\star \|^2}{\alpha \uptau R}	  
	+ \frac{1152 L     }{ (1-\delta) \underline{\lambda} N R}
 \|\x^{0}-\bar{\x}^0\|^2  + \frac{2304 \alpha^2 \uptau^2 L^2    \varsigma_0^2 }{ (1-\delta) (1-\lambda) R} \nonumber \\
	& \quad + 	  \frac{3072 \alpha^2 \uptau  L  \sigma^2}{(1-\delta)}
	+   24 \alpha^2 \uptau L    \sigma^2 +	\frac{ 4 \alpha   \sigma^2 }{ N},
\end{align}
where $\varsigma_0^2 = \frac{1}{N} \|\grad \f(\bar{\x}^0) - \one \otimes \grad f(\bar{\x}^0) \|^2$.

\subsubsection{Strongly-convex case $\mu > 0$}
From \eqref{avg_bound_cvx_proof} and  \eqref{cvx_dhat_bound}, it holds that
	\begin{align} 
	\Ex \|\bar{x}^{r+1}-x^\star \|^2 
	&\leq   (1-\tfrac{\mu \uptau \alpha }{2} ) \Ex  \|\bar{x}^r-x^\star\|^2  
	 + \tfrac{192 \alpha \uptau L  }{N}  \Ex \|\widehat{\d}^r\|^2      \nonumber \\ 
	& \quad   +    12 \alpha^3 \uptau^2  L      \sigma^2 + \frac{2\alpha^2 \uptau \sigma^2 }{N}.
\end{align}
and
	\begin{align}
	\textstyle \Ex \|\widehat{\d}^{r+1}\|^2 
	& \leq    \bar{\delta}  \Ex \|\widehat{\d}^r\|^2 +  \textstyle (4 \alpha^2 \uptau^2   L^2 N)   	\Ex \|\bar{x}^{r}-x^\star \|^2 
	+ 6 \alpha^2 \uptau   N \sigma^2,
\end{align}
where the last inequality follows from $\|\grad f(\bar{x}^r)\|^2 \leq L^2  \|\bar{x}^{r}-x^\star \|^2 $.	 It follows that
	\begin{align} \label{linear_dynamical_error2}
		\begin{bmatrix} 
			  \Ex  \|\bar{x}^{r+1}-x^\star\|^2 \\
	 \frac{1}{N}	\Ex \|\widehat{\d}^{r+1}\|^2  
		\end{bmatrix}
		\leq 
		\underbrace{\begin{bmatrix}
				1- \frac{\mu \uptau \alpha}{2}  \vspace{0.5mm} 		 
				&
			 192 \alpha \uptau L      \vspace{0.5mm} \\
			4 \alpha^2 \uptau^2 L^2   
				& 
				\tfrac{1+\delta}{2}
		\end{bmatrix}}_{\define A}
		\begin{bmatrix} 
			  \Ex  \|\bar{x}^r-x^\star\|^2 \\
		 \frac{1}{N} \Ex \|\widehat{\d}^{r}\|^2  
		\end{bmatrix} 
		+ \underbrace{\begin{bmatrix}
			 12 \alpha^3 \uptau^2 L   \sigma^2 + \frac{2\alpha^2 \uptau \sigma^2 }{N} \\
		6 \alpha^2 \uptau    \sigma^2
		\end{bmatrix}}_{\define b}.
	\end{align} 
	 Note that
	\begin{align} \label{rho_H}
		\rho(A) \leq \|A\|_1= \max \left\{
		1- \frac{\mu \uptau \alpha}{2}
		+ 4 \alpha^2 \uptau^2 L^2 
		, ~
		\tfrac{1+\delta}{2}+ 192 \alpha \uptau L    
		\right\} \leq 1-\tfrac{\mu \uptau \alpha}{4} .
	\end{align}
	where the  last inequality holds under the  step size condition:
	\begin{align} \label{ss_cvx_5_strongcvx}
		\alpha \leq \min \left\{\frac{\mu}{8 \uptau L^2},  \frac{1-\delta}{2 \uptau (192 L  + \mu/4)}\right\}.
\end{align}
	Since $\rho(A) <1$, we can iterate inequality \eqref{linear_dynamical_error2} to get
	\begin{align} 
		\begin{bmatrix} 
		\Ex  \|\bar{x}^{r}-x^\star\|^2 \\
	 \frac{1}{N}	\Ex \|\widehat{\d}^{r}\|^2  
	\end{bmatrix}
		& \leq  
		A^r
		\begin{bmatrix} 
		\Ex  \|\bar{x}^{0}-x^\star\|^2 \\
		 \frac{1}{N} \Ex \|\widehat{\d}^{0}\|^2  
	\end{bmatrix}
		+ \sum_{\ell=0}^{r-1} A^\ell b \nonumber \\ 
		& \leq  
		A^r
		\begin{bmatrix} 
		\Ex  \|\bar{x}^{0}-x^\star\|^2 \\
		 \frac{1}{N} \Ex \|\widehat{\d}^{0}\|^2  
	\end{bmatrix} 
		+(I- A)^{-1} b.
	\end{align}
	Taking the $1$-induced-norm and using   properties of the (induced) norms, it holds that
	\begin{align} \label{tran_SC_0}
	\Ex  \|\bar{x}^{r}-x^\star\|^2  +  \tfrac{1}{N}	\Ex \|\widehat{\d}^{r}\|^2  
		& \leq  
		\|A^r\|_1  a_0
		+ \left\|(I- A)^{-1} b \right\|_1 \leq  
		\|A\|^r_1 a_0
		+ \left\|(I- A)^{-1} b \right\|_1.
	\end{align}
	where $a_0  \define \Ex  \|\bar{x}^{0}-x^\star\|^2 + \frac{1}{N} \Ex \|\widehat{\d}^{0}\|^2 $.  We now bound the last term by  noting that  
	\begin{align*}
		(I-A)^{-1} b &= \begin{bmatrix}
			 \frac{\mu \uptau \alpha}{2}  \vspace{0.5mm} 		 
			&
		-	192 \uptau \alpha L      \vspace{0.5mm} \\
			-4 \alpha^2 \uptau^2 L^2 
			& 
			\tfrac{1-\delta}{2}
		\end{bmatrix}^{-1}  b
		=
		\frac{1}{\det(I-A)}
		 \begin{bmatrix}
		 \tfrac{1-\delta}{2} \vspace{0.5mm} 		 
			&
				192 \uptau \alpha L    \vspace{0.5mm} \\
			4 \alpha^2 \uptau^2 L^2 
			& 
				\frac{\mu \uptau \alpha}{2}
		\end{bmatrix} b \\
		& = 
		\frac{1}{\alpha \uptau \mu (1-\delta) (\frac{1}{4} -768 \alpha^3 \uptau^3 L^3  ) }
	 \begin{bmatrix}
		\tfrac{1-\delta}{2} \vspace{0.5mm} 		 
		&
		192 \uptau \alpha L    \vspace{0.5mm} \\
		4 \alpha^2 \uptau^2 L^2 
		& 
		\frac{\mu \uptau \alpha}{2}
	\end{bmatrix}
		\begin{bmatrix}
			12 \alpha^3 \uptau^2 L   \sigma^2 + \frac{2\alpha^2 \uptau \sigma^2 }{N} \\
			6 \alpha^2 \uptau    \sigma^2
		\end{bmatrix}
		\nonumber \\
		& \leq  \frac{8}{ \alpha \uptau \mu (1-\delta)}  
		\begin{bmatrix}
			  (1-\delta) \alpha^2 \uptau \sigma^2 /N + 6 (1-\delta) \alpha^3 \uptau^2 L   \sigma^2 + 1152 \alpha^3 \uptau^2  L   \sigma^2  \vspace{2mm}
			\\
		8 \alpha^4 \uptau^3  L^3  \sigma^2 (1/N + 6 \alpha \uptau L )+	3 \alpha^3 \uptau^2 \mu  \sigma^2 
		\end{bmatrix}.
	\end{align*}
The last step holds for $\frac{1}{4} -768 \alpha^3 \uptau^3 L^3   \geq \frac{1}{8} $ or $768 \alpha^3 \uptau^3 L^3 \leq \frac{1}{8}$, which holds under condition \eqref{ss_cvx_1}.
	Therefore,
	\begin{align*}
		& \left\|(I- A)^{-1} b \right\|_1 \nonumber \\
		& \leq     \frac{8 \alpha \sigma^2 }{\mu N} + \dfrac{ 48 (1-\delta) \alpha^2 \uptau L       \sigma^2 + 6144  \alpha^2 \uptau L  \sigma^2  }{  \mu (1-\delta)} + \dfrac{	16 \alpha^3 \uptau^2  L^3   \sigma^2 (1/N +6 \alpha \uptau L)+	3 \alpha^2 \uptau \mu  \sigma^2}{ \mu (1-\delta) }.
	\end{align*}
	Substituting the above into \eqref{tran_SC_0} and using \eqref{rho_H}, we obtain
	\begin{align} \label{last_strong_cvx_thm_proof}
	&		\Ex  \|\bar{x}^{r}-x^\star\|^2  +  \tfrac{1}{N}	\Ex \|\widehat{\d}^{r}\|^2  
		\nonumber \\
		& \leq  (1-\tfrac{\alpha \uptau \mu}{4})^r a_0
		+     \tfrac{8 \alpha \sigma^2 }{\mu N} + \tfrac{ 48 (1-\delta) \alpha^2 \uptau L       \sigma^2 + 6147  \alpha^2 \uptau L  \sigma^2  }{  \mu (1-\delta)} + \tfrac{	16 \alpha^3 \uptau^2  L^3   \sigma^2 (1/N +6 \alpha \uptau L)}{ \mu (1-\delta) }.
	\end{align}

\section{Proof of Corollary \ref{coro:rates}} \label{app:corro:rate}
The final rate can be obtained by tuning the stepsize in a way similar to  \cite{stich2019unified,karimireddy2019scaffold,koloskova2020unified}.

\paragraph{Nonconvex  case} If  all nodes use equal initialization, then equation \eqref{last_noncvx_thm_proof} (\eqref{eq:thm1:noncvx} from Theorem \ref{thm:noncvx}) reduces to
\begin{align} \label{cvx_stepsize_final}
\frac{1}{R}	 \sum_{r=0}^{R-1} \cE_r \leq \underbrace{\frac{c_0}{\alpha  R}+ c_1 \alpha   +c_2 \alpha^2 }_{\define \Psi_R} + \frac{a_{0} \alpha^2 }{R} ,
\end{align}
where $	\cE_r \define  \textstyle   \Ex  \| \grad f(\bar{x}^{r})  \|^2  +  \frac{1}{\uptau} \sum\limits_{t} \|  \tfrac{1}{N}  \sum\limits_{i} \nabla f_i(\phi^r_{i,t})\|^2 $ and
\begin{subequations} \label{parameters_noncvx}
	\begin{align}
c_0 & =   8   \frac{f(\bar{x}^{0})-f(\bar{x}^{\star})}{\uptau}, \quad
	c_1 = 	\frac{ 8   L \sigma^2 }{ N} \\
		c_2 &=  \uptau L^2 \sigma^2 \left(\frac{1536   }{(1-\sqrt{\lambda})}
	+ 16 \right), \quad a_{0} = \frac{2560  \uptau^2 L^2      \varsigma_0^2 }{ (1-\sqrt{\lambda}) (1-\lambda)  }  .
\end{align}
\end{subequations}
Note that the above holds under the condition:
\begin{align} \label{ss_noncvx_all}
	\alpha \leq  \frac{1}{\underline{\alpha}} \define 
		 \min \left\{ \frac{1-\sqrt{\lambda}}{16 \sqrt{2} \uptau L }, \frac{ \sqrt{(1-\sqrt{\lambda})(1-\lambda)}}{6 \uptau L }, \frac{\sqrt{1-\sqrt{\lambda}}}{32 \sqrt{2}  \uptau L } \right\} = O\left(\frac{1-\lambda}{\uptau L} \right),
\end{align}
 where $\frac{1}{\underline{\alpha}}$ satisfies all stepsize conditions used to derive \eqref{eq:thm1:noncvx}. Setting $\alpha =\min \left\{\left(\frac{c_0}{c_1 R}\right)^{\frac{1}{2}},\left(\frac{c_0}{c_2 R}\right)^{\frac{1}{3}}, \frac{1}{\underline{\alpha}} \right\} \leq  \frac{1}{\underline{\alpha}}$. Then we have three cases.

\begin{itemize}
	\item  When $\alpha=\frac{1}{\underline{\alpha}}$ and is smaller than both $\left(\frac{c_0}{c_1 R}\right)^{\frac{1}{2}}$ and $\left(\frac{c_0}{c_2 R}\right)^{\frac{1}{3}}$, then
	\begin{align*}
		\Psi_R =\frac{c_0}{\alpha  R}+ c_1 \alpha   +c_2 \alpha^2= \frac{\underline{\alpha} c_0}{ R}+ \frac{c_1}{\underline{\alpha}} + \frac{c_2}{\underline{\alpha}^2} \leq
		\frac{\underline{\alpha} c_0}{ R}
		+ c_1^{\frac{1}{2}} \left(\frac{c_0}{R}\right)^{\frac{1}{2}}+c_2^{\frac{1}{3}}\left(\frac{c_0}{R}\right)^{\frac{2}{3}}.
	\end{align*}
	
	\item When	$\alpha=\left(\frac{c_0}{c_1 R}\right)^{\frac{1}{2}} \leq \left(\frac{c_0}{ c_2  R}\right)^{\frac{1}{3}}$, then
	\begin{align*}
		\Psi_R \leq 2 c_1^{\frac{1}{2}} \left(\frac{ c_0 }{R}\right)^{\frac{1}{2}}+c_2 \left(\frac{c_0}{c_1 R}\right) \leq 2 c_1^{\frac{1}{2}} \left(\frac{ c_0 }{R}\right)^{\frac{1}{2}}+c_2^{\frac{1}{3}}\left(\frac{c_0}{R}\right)^{\frac{2}{3}}.
	\end{align*}
	
	\item  When $\alpha=\left(\frac{c_0}{c_2  R}\right)^{\frac{1}{3}} \leq \left(\frac{c_0}{c_1 R}\right)^{\frac{1}{2}}$, then
	\begin{align*}
		\Psi_R \leq 2 c_2^{\frac{1}{3}}\left(\frac{c_0}{ R}\right)^{\frac{2}{3}}+c_1 \left(\frac{c_0}{c_2 R}\right)^{\frac{1}{3}} \leq 2 c_2^{\frac{1}{3}}\left(\frac{c_0}{R}\right)^{\frac{2}{3}}+c_1^{\frac{1}{2}} \left(\frac{ c_0}{R}\right)^{\frac{1}{2}}.
	\end{align*}
\end{itemize}
Combining the above three cases together it holds that
\begin{align*}
\Psi_R=	\frac{c_0}{\alpha R}+c_1 \alpha+ c_2 \alpha^{2} \leq 2 c_1^{\frac{1}{2}} \left(\frac{ c_0}{R}\right)^{\frac{1}{2}}+2 c_2^{\frac{1}{3}}\left(\frac{c_0}{R}\right)^{\frac{2}{3}}+\frac{\underline{\alpha} c_0}{R}.
\end{align*}
Substituting the above into \eqref{cvx_stepsize_final}, we conclude that
\begin{align*} 
\frac{1}{R}	 \sum_{r=0}^{R-1} \cE_r
	& \leq 2 c_1^{\frac{1}{2}} \left(\frac{ c_0}{R}\right)^{\frac{1}{2}}+2 c_2^{\frac{1}{3}}\left(\frac{c_0}{R}\right)^{\frac{2}{3}}
	+\frac{(\underline{\alpha} c_0+ a_{0}/\underline{\alpha}^2 )}{R}.
\end{align*}
The rate \eqref{rate_noncvx} follows by plugging the parameters \eqref{parameters_noncvx} and using \eqref{ss_noncvx_all}.

\paragraph{Convex case} If we start from equal initialization then the convex bound \eqref{last_cvx_thm_proof} also satisfies \eqref{cvx_stepsize_final} under condition
	\begin{align}
	\alpha \leq \frac{1}{\underline{\alpha}} \define \min \left\{
	\frac{\sqrt{(1-\sqrt{\lambda}) (1-\lambda) }}{6  \uptau L},  \frac{1-\sqrt{\lambda}}{33 \uptau L},\frac{\sqrt{1-\sqrt{\lambda}}}{100  \uptau L }  \right\} = \cO\left(\frac{1-\lambda}{\uptau L}\right)
\end{align}
  with
\begin{align*}
	\cE_r &\define   \Ex [f(\bar{x}^r) - f(x^\star)] \\
 c_0 &=  \frac{  2  \|\bar{x}^{0}-x^\star \|^2}{ \uptau }, \quad
	c_1 = 	\frac{ 4   L \sigma^2 }{ N} \\
	 	c_2 &=    \frac{3072 \alpha^2 \uptau  L  \sigma^2}{(1-\sqrt{\lambda})}
	 	+   24 \alpha^2 \uptau L    \sigma^2 , \quad 	a_0  = \frac{2304  \uptau^2 L^2    \varsigma_0^2 }{ (1-\sqrt{\lambda}) (1-\lambda)  } .
\end{align*}
Therefore, the rate can be obtained by following the same arguments used for the noncovex case.

\paragraph{Strongly convex case} Using the stepsize condition used to derive Theorem \ref{thm:cvx}, namely,
	\begin{align} \label{all_strong_cvx_step}
	\alpha \leq \frac{1}{\underline{\alpha}} \define \min \left\{ \frac{\sqrt{(1-\sqrt{\lambda}) (1-\lambda) }}{6  \uptau L}, \frac{\sqrt{1-\sqrt{\lambda}}}{100  \uptau L } ,\frac{\mu}{8 \uptau L^2},  \frac{1-\sqrt{\lambda}}{2 \uptau (192 L  + \mu/4)} \right\} 
	= \cO \left(\frac{\mu (1-\lambda)}{L^2 \uptau}\right),
\end{align}
and starting from equal initialization, inequality \eqref{last_strong_cvx_thm_proof} (\eqref{eq:thm2:strongcvx}) can be upper bounded by
	\begin{align} 
		\Ex  \|\bar{x}^{R}-x^\star\|^2  +  \tfrac{1}{N}	\Ex \|\widehat{\d}^{R}\|^2 
			& \leq  (1-\tfrac{\alpha \uptau \mu}{4})^R a_0
		+     \tfrac{8 \alpha \sigma^2 }{\mu N} + \tfrac{ 48 (1-\delta) \alpha^2 \uptau L       \sigma^2 + 6147  \alpha^2 \uptau L  \sigma^2  }{  \mu (1-\delta)} + \tfrac{	16 \alpha^3 \uptau^2  L^3   \sigma^2 (1/N +6 \alpha \uptau L)}{ \mu (1-\delta) },
		\nonumber \\ 
	& \leq  (1-\tfrac{\alpha \uptau \mu}{4})^R a_0
	+   c_1 \alpha + c_2 \alpha^2 \nonumber  \\
	& \leq \exp(-\tfrac{\alpha \uptau \mu}{2} R) (c_0+ \alpha^2 b_0) +   c_1 \alpha + c_2 \alpha^2, \label{rsvx}
\end{align}
where
	\begin{align*}
c_0 &=   \|\bar{x}^{0}-x^\star\|^2, \quad b_0 = \frac{2  \uptau^2 \varsigma_0^2}{1-\lambda}    \\
c_1 &=  \frac{8 \sigma^2}{\mu N}, \quad 
c_2 = 	\frac{ 48   \uptau L       \sigma^2 + 6147  \uptau L  \sigma^2  +   \uptau  L^2   \sigma^2}{  \mu (1-\sqrt{\lambda})} .
\end{align*}
Now we select $\alpha=\min \left\{\frac{\ln \left(\max \left\{1, \mu \uptau (c_0 + b_0/\underline{\alpha}^2) R / c_1\right\}\right)}{\mu \uptau R}, \frac{1}{\underline{\alpha}}\right\} \leq \frac{1}{\underline{\alpha}}$ to get the following cases.
\begin{itemize}
	\item  If $\alpha=\frac{\ln \left(\max \left\{1, \mu (c_0 + b_0/\underline{\alpha}^2) R / c_1\right\}\right)}{\mu \uptau R} \leq \frac{1}{\underline{\alpha}} $ then  
	\begin{align*}
	\exp(-\tfrac{\alpha \uptau \mu}{2} R) (c_0+ \alpha^2 b_0) &\leq	\tilde{\mathcal{O}}\left( (c_0 + \frac{b_0}{\underline{\alpha}^2})  \exp \left[-\ln \left(\max \left\{1, \mu \uptau (c_0 + \frac{b_0}{\underline{\alpha}^2}) R / c_1 \right\}\right)\right]\right) \\
		&=\mathcal{O}\left(\frac{c_1}{\mu \uptau R}\right).
	\end{align*}
	
	\item  Otherwise $\alpha = \frac{1}{\underline{\alpha}} \leq \frac{\ln \left(\max \left\{1, \mu \uptau (c_0+b_0/\underline{\alpha}^2)  / c_1 \right\}\right)}{\mu \uptau R}$ and
	\begin{align*}
			\exp(-\tfrac{\alpha \uptau \mu}{2} R) (c_0+ \alpha^2 b_0) &=	\tilde{\mathcal{O}}\left(   \exp \left[-\frac{ \uptau \mu R}{2 \underline{\alpha}}\right] (c_0+\frac{b_0}{\underline{\alpha}^2})
			 \right).
	\end{align*}
\end{itemize}
Collecting these cases together into \eqref{rsvx}, we obtain
		\begin{align} 
			\Ex  \|\bar{x}^{R}-x^\star\|^2  +  \tfrac{1}{N}	\Ex \|\widehat{\d}^{R}\|^2  
			& \leq \exp(-\tfrac{\alpha \uptau \mu}{2} R) (c_0+ \alpha^2 b_0) +   c_1 \alpha + c_2 \alpha^2 \nonumber \\
			& \leq \tilde{\mathcal{O}}\left(\frac{c_1}{\mu \uptau R}\right) 
			+ \tilde{\mathcal{O}}\left(\frac{c_2}{\mu^2 \uptau^2 R^2}\right) 
			+	\tilde{\mathcal{O}}\left(   \exp \left[-\frac{ \uptau \mu R}{2 \underline{\alpha}}\right] (c_0+\frac{b_0}{\underline{\alpha}^2})
			\right). 
		\end{align}
Plugging in the parameters and using \eqref{all_strong_cvx_step} gives the final rate \eqref{rate_strongcvx}.

\section{LED analysis in the centralized server-workers setup} \label{app_server_proof}
In this section, we will analyze \alg{LED} within the server-workers setup. Specifically, we will examine the algorithm listed in \ref{alg:fed-ed}. It can be verified that this algorithm is equivalent to Algorithm \ref{alg:led} for the fully connected network case, $W=(1/N)\one \one\tran$, when $\gamma=1$. Here, $\gamma$ is an additional parameter that allows us to derive tighter bounds. To make this section self-contained, we will revisit steps similar to those in the decentralized case but specialized for the centralized case, leading to simpler steps.

\begin{algorithm}[t] 
	\caption{\sc  LED in the server-workers setup}
	\textbf{node $i$ input:} $x^0$,   $\alpha>0$, $\beta>0$, $\gamma>0$, and $\uptau$.
	
	\textbf{initialize} $y_i^0=0$.
	
	\textbf{repeat for} $r=0,1,2,\dots$
	\begin{enumerate}\itemsep=-5pt
		\item 	 \texttt{Local updates:} In parallel each node (worker) $i$  do ($\phi_{i,0}^r=x^{r}$):  
		\begin{subequations}
			\begin{align}  \label{local_fed_per_node}
				\phi_{i,t+1}^r &=    \phi_{i,t}^r-\alpha  \grad F_i(\phi_{i,t}^r;\xi_{i,t}^r)   -  \beta  y_i^{r} , \quad t=0,\dots,\uptau-1 .
			\end{align}
			
			\item  \texttt{Communication:} Server receives $\{\phi_{j,\uptau}^{r}\}$ from all workers, computes the average $\frac{1}{N} \sum_{j=1}^N  \phi_{j,\uptau}^{r}$, and send it back to all nodes.

			\item \texttt{Estimates update:} Each node $i$ do
			\begin{align} 
					x^{r+1} &= (1-\gamma) x^r + \gamma \frac{1}{N} \sum_{j=1}^N  \phi_{j,\uptau}^{r}. \label{local_fed_comm_round_per_node} \\
				\label{local_fed_dual_per_node}
				y_i^{r+1} &= y_i^{r} +  \phi_{i,\uptau}^{r} -\frac{1}{N} \sum_{j=1}^N  \phi_{j,\uptau}^{r}.
			\end{align}
			
		\end{subequations}
	\end{enumerate}
	\label{alg:fed-ed} 
\end{algorithm} 

\paragraph{Network description}	 We start by defining 
\begin{subequations}
	\begin{align}
		\A&= \tfrac{1}{N} \one \one\tran \otimes I_m \in \real^{mN \times mN} \\
		\bar{\x}^r &=\col\{x^r,\dots,x^r\} \in \real^{mN} \\
		\x &=\col\{x_1,\dots,x_N\} \in \real^{mN}, ~ x_i \in \real^m \\	
		\bm{\upphi}^{r}_{t}&=\col\{\phi_{1,t}^r,\dots,\phi_{N,t}^r\} \in \real^{mN} \\
		\f(\x)&=\sum_{i=1}^N f_i(x_i) \\
		\grad \f(\x)&=\col\{\grad f_1(x_1),\dots,\grad F_N(x_N)\} \in \real^{mN} \\
		\grad \F(\x; \bxi)&=\col\{\grad F_1(x_1;\xi_1),\dots,\grad F_N(x_N;\xi_N)\} \in \real^{mN} \\
		\y &=\col\{y_1,\dots,y_N\} \in \real^{mN}.
	\end{align}
\end{subequations}
Using the above notation, Algorithm \ref{alg:fed-ed} can be described in compact form as follows: Set $\bm{\upphi}_{0}^r=\bar{\x}^{r}$ and do: 
	\begin{subequations}  \label{fed_updates_network}
		\begin{align}  \label{local_fed_updates}
			\bm{\upphi}^{r}_{t+1} &=    \bm{\upphi}^{r}_{t}-\alpha  \grad \F(\bm{\upphi}^{r}_{t};\bxi_t^r)   -  \beta  \y^{r}, \quad t=0,\dots,\uptau-1 \\
			\bar{\x}^{r+1} &= (1-\gamma) \bar{\x}^r + \gamma  \A \bm{\upphi}^r_{\uptau}  \label{local_fed_comm_round}
			\\
			\y^{r+1} &= \y^{r}+  (\I-\A) \bm{\upphi}^r_{\uptau}.
		\end{align}
	\end{subequations}
For analysis purposes, we also introduce the notation:
\begin{subequations} \label{notation_proof_fed}
	\begin{align}
		\overline{\grad f } (\x^r)  &\define \frac{1}{N} \sum_{i=1}^N  \grad f_i (x_i^r) \\
		\bar{s}_t^r &\define \frac{1}{N} \sum_{i=1}^N \left( \nabla F_i(\phi^{r}_{i,t};\xi_{i,t}) - \nabla f_i(\phi^{r}_{i,t}) \right), \quad
		\s_t^r \define \nabla \F(\bm{\Phi}^{r}_t;\bxi_t) - \nabla \f(\bm{\Phi}^{r}_{t}) \label{noise_def_fed} \\
		\z^r &\define \y^r +\frac{\alpha}{\beta}  \grad \f(\bar{\x}^r), \quad \bar{\z}^r \define \one_N \otimes \bar{z}^r, \quad \bar{z}^r  \define \frac{1}{N}\sum_{i=1}^N z_i^r  \label{z_def_fed}.
	\end{align}
\end{subequations}

\subsection{Centroid and gradient deviation}
Iterating  \eqref{local_fed_updates} updates:
\begin{align*}
	\bm{\Phi}^{r}_{\uptau} &=    \bar{\x}^{r}-\alpha \sum_{t=0}^{\uptau-1} \nabla \F(\bm{\Phi}^{r}_{t};\bxi_{t}) - \beta  \uptau \y^r,
\end{align*}
where $\bm{\Phi}^r_{0}=\x^r$ and (for simplicity) we are removing the superscript in $\bxi_{t}^r$.	Substituting the preceding into \eqref{local_fed_comm_round} yields
\begin{subequations} \label{x_and_y_r_update0_fed}
	\begin{align}
		\bar{\x}^{r+1} &=  (1-\gamma) \bar{\x}^{r}  + \gamma  \A \left(\bar{\x}^{r}  - \beta \uptau \y^r - \alpha \sum_{t=0}^{\uptau-1} \nabla \F(\bm{\Phi}^{r}_{t};\bxi_t)  \right) \label{x_r_upd0} \\
		\y^{r+1} &= \y^{r}+  (\I-\A)   \left(\bar{\x}^{r}- \beta \uptau \y^r  - \alpha \sum_{t=0}^{\uptau-1} \nabla \F(\bm{\Phi}^{r}_t;\bxi_t)   \right).
	\end{align}
\end{subequations}
When $\y^{0}=\zero$, the iterates $\{\y^r\}$ will always be in the range of $\I-\A$, consequently, $(\one\tran \otimes I_m) \y^r =0$ for all $r$. Using this and the fact $\A \bar{\x}^{r}=\bar{\x}^{r}$, the updates \eqref{x_and_y_r_update0_fed} become
\begin{subequations} \label{x_and_y_r_update_fed}
	\begin{align}
		\bar{\x}^{r+1} &=   \bar{\x}^{r}  + \alpha \gamma  \A   \sum_{t=0}^{\uptau-1} \nabla \F(\bm{\Phi}^{r}_{t};\bxi_t)   \label{x_r_upd_fed} \\
		\y^{r+1} &=  (1-\beta \uptau) \y^{r} - \alpha  (\I-\A)  \sum_{t=0}^{\uptau-1} \nabla \F(\bm{\Phi}^{r}_t;\bxi_t) 
		 \label{y_r_upd_fed}
	\end{align}
\end{subequations}
Using the definitions in \eqref{notation_proof_fed} into \eqref{y_r_upd_fed}, we have
	\begin{align} 
		\z^{r+1} &= [(1- \beta \uptau)\I+\beta \uptau \A ]  \z^{r}  - \alpha  (\I-\A)  \sum_{t=0}^{\uptau-1} \big(\nabla \f(\bm{\Phi}^{r}_{t} ) -\grad \f(\bar{\x}^{r}) + \s_t^r \big)
		+  \frac{\alpha}{\beta}  \big(\grad \f(\bar{\x}^{r+1})-\grad \f(\bar{\x}^{r}) \big) \label{z_r_update_fed}.  
	\end{align}
Let $\bar{\z}^r = \A \z^r$, then we have for $\beta=1/\uptau$:
\begin{align*}
	\bar{\z}^{r+1} &=   \A \bar{\z}^{r}  - \alpha  \A (\I-\A)  \sum_{t=0}^{\uptau-1} \big(\nabla \f(\bm{\Phi}^{r}_{t} ) -\grad \f(\bar{\x}^{r}) + \s_t^r \big)
	+  \A \frac{\alpha}{\beta}  \big(\grad \f(\bar{\x}^{r+1})-\grad \f(\bar{\x}^{r}) \big) 
	\\
	&=   \bar{\z}^{r}   
	+  \A \frac{\alpha}{\beta}  \big(\grad \f(\bar{\x}^{r+1})-\grad \f(\bar{\x}^{r}) \big)
\end{align*}
where in the last step we used $\A \bar{\z}^{r}=\bar{\z}^r$ and $\A(\I-\A)=\zero$.
Therefore,
\begin{align} \label{z_track0_fed}
	\z^{r+1} -\bar{\z}^{r+1}  &=  - \alpha  (\I-\A)  \sum_{t=0}^{\uptau-1} \big(\nabla \f(\bm{\Phi}^{r}_{t} ) -\grad \f(\bar{\x}^{r}) + \s_t^r \big)
	+  \alpha \uptau (\I-\A)  \big(\grad \f(\bar{\x}^{r+1})-\grad \f(\bar{\x}^{r}) \big).
\end{align}
It follows from \eqref{x_r_upd_fed} and \eqref{z_track0_fed} and  that
	\begin{subequations} 
		\begin{align}
			x^{r+1}&=x^r -  \alpha \gamma \sum_{t=0}^{\uptau-1} \big( \overline{\nabla f}(\bm{\Phi}^r_t) + \bar{s}_t^r  \big)   \label{error_fed_avg_progress}  \\
			\z^{r+1} -\bar{\z}^{r+1}  &=  - \alpha  (\I-\A)  \sum_{t=0}^{\uptau-1} \big(\nabla \f(\bm{\Phi}^{r}_{t} ) -\grad \f(\bar{\x}^{r}) + \s_t^r \big) 
			+  \alpha \uptau (\I-\A)  \big(\grad \f(\bar{\x}^{r+1})-\grad \f(\bar{\x}^{r}) \big)  \label{z_track_fed}
		\end{align}
	\end{subequations}

\subsection{Auxiliary bounds}
Define
\begin{align} \label{drift_def_fed}
	\|\bm{\widehat{\Phi}}^r\|^2 \define \sum\limits_{t=0}^{\uptau-1}  \|\bm{\Phi}^r_{t} - \bar{\x}^r\|^2 =	    \sum_{t=0}^{\uptau-1} \sum_{i =1}^N    \|\phi^{r}_{i,t}-x^r\|^2.
\end{align}
where $\bm{\widehat{\Phi}}^r \define \col\{\bm{\Phi}^r_{t} - \bar{\x}^r\}_{t=0}^{\uptau-1}$.

\begin{lemma}[\sc \small Local drift bound] \label{lema_drift_cen_fed}
	\rm  Let Assumptions \ref{assump:noise}--\ref{assump:smoothness} hold, then for $\alpha  \leq \frac{1}{2 \sqrt{2} L \uptau} $ we have
	\begin{align} \label{drift_bound_fed}
		\Ex 	\|\bm{\widehat{\Phi}}^r\|^2 & \leq 16 \uptau  \Ex \|\z^r -\bar{\z}^r\|^2 + 16   \alpha^2 \uptau^3 N  \Ex \|\grad f(x^{r})\|^2 + 4 \alpha^2   \uptau^2 N \sigma^2.
	\end{align}
\end{lemma}
\begin{proof} 
	The proof follows similar steps to Lemma \ref{lema_drift} specialized to the centralized scenario.	When $\uptau=1$, then $\phi_{i, 0}=x^r$ for all $i$ and 
	$
	\Ex 	\|\bm{\widehat{\Phi}}^r\|^2=  0. $
	Now suppose that $\uptau \geq 2$. Then, using \eqref{local_fed_updates}, it holds that
	\[
	\begin{aligned}
		\Ex\|\phi^{r}_{i,t+1} -x^{r}\|^2 &=\Ex \left\|
		\phi^{r}_{i,t}-x^{r}-\alpha  \grad F_i(\phi^{r}_{i,t};\xi_{i,t})   -  \beta  y_i^{r} \right\|^2 \\
		&\leq  \Ex \left \| \phi^{r}_{i,t}-x^{r}-\alpha  \grad f_i(\phi^{r}_{i,t})   -  \beta  y_i^{r} \right\|^2+\alpha^2 \sigma^2 \\
		&\leq  \left(1+\tfrac{1}{\uptau-1}\right) \Ex\left\| \phi^{r}_{i,t}-x^{r} \right\|^2+\uptau \Ex \left\|\alpha  \grad f_i(\phi^{r}_{i,t})   +  \beta  y_i^{r}  \right\|^2+\alpha^2 \sigma^2  \\
		&= \left(1+\tfrac{1}{\uptau-1}\right) \Ex\| \phi^{r}_{i,t}-x^{r} \|^2+\uptau \Ex \left\|\alpha  \grad f_i(\phi^{r}_{i,t})   - \alpha \grad f_i(x^{r})    + \beta  z_i^{r}  \right\|^2+\alpha^2 \sigma^2 \\
		&\leq \left(1+\tfrac{1}{\uptau-1}\right) \Ex\| \phi^{r}_{i,t}-x^{r} \|^2+2  \alpha^2 \uptau \Ex \left\|  \grad f_i(\phi^{r}_{i,t})   -  \grad f_i(x^{r}) \right\|^2   + 2 \uptau \beta^2 \Ex \|  z_i^{r}  \|^2 + \alpha^2 \sigma^2 \\
		&\leq \left(1+\tfrac{1}{\uptau-1}+2\alpha^2 \uptau L^2  \right) \Ex\|\phi^{r}_{i,t}-x^{r} \|^2+\tfrac{2}{ \uptau} \Ex \left\|z_i^r\right\|^2+ \alpha^2 \sigma^2 \\
		& \leq \left(1+\tfrac{5/4}{\uptau-1}\right) \Ex\| \phi^{r}_{i,t}-x^{r} \|^2+\tfrac{2}{ \uptau} \Ex \|z_i^r\|^2+ \alpha^2 \sigma^2.
	\end{aligned}
	\]
	The  last inequality holds for $2\alpha^2 \uptau L^2  \leq \frac{1}{4(\uptau-1)}$, which is satisfied if $\alpha \leq \frac{1}{2\sqrt{2} L \uptau}$.		Iterating the inequality above for $t=0,\dots,\uptau-1$:
	\[
	\begin{aligned}
		\Ex\|\phi^{r}_{i,t+1} -x^{r}\|^2 & \leq \textstyle \sum\limits_{\ell=0}^{t}\left( \tfrac{2}{\uptau} \Ex \left\| z_i^r \right\|^2+ \alpha^2 \sigma^2 \right)\left(1+\frac{5/4}{(\uptau-1)}\right)^\ell \\
		& \leq \textstyle \sum\limits_{\ell=0}^{t}\left( \tfrac{2}{\uptau} \Ex \left\| z_i^r \right\|^2+ \alpha^2 \sigma^2 \right) \exp\left(\tfrac{(5/4) \ell}{\uptau-1}\right) \\
		& \leq  \left(\tfrac{2}{\uptau} \Ex\left\| z_i^r \right\|^2+\alpha^2 \sigma^2 \right) 4 \uptau \\
		& = 
		8 \Ex \left\| z_i^r \right\|^2  +4  \alpha^2 \uptau \sigma^2,
	\end{aligned}
	\]
	where  in the second and third inequalities we used  $(1+\tfrac{a}{\uptau-1})^t \leq \exp(\frac{ a t}{\uptau-1})\leq \exp(a)$ for $t \leq \uptau -1$.	Summing over $i$ and $t$:
	\[
	\begin{aligned}
		\Ex	\|\bm{\widehat{\Phi}}^r\|^2 & \leq 8 \uptau \textstyle \sum\limits_{i=1}^N \Ex \left\| z_i^r \right\|^2  + 4 \alpha^2 \uptau^2 N \sigma^2 \\
		& \leq 16 \uptau \textstyle \sum\limits_{i=1}^N ( \Ex \|z_i^r -\bar{z}^r\|^2 +  \Ex \|\bar{z}^r\|^2 ) + 4 \alpha^2 \uptau^2 N \sigma^2 \\
		& \leq 16 \uptau  \Ex \|\z^r -\bar{\z}^r\|^2  +  16  \uptau N \Ex \|\bar{z}^r\|^2+ 4 \alpha^2 \uptau^2 N \sigma^2.
	\end{aligned}
	\]
	The result follows by using $\bar{z}^r= \alpha \uptau \frac{1}{N} \sum_{i=1}^N \grad f_i(x^r)$.
\end{proof}
\begin{lemma}[\sc \small Gradient deviation bound] \label{lemma:cons_inequ_led}
	\rm For $\alpha \leq \frac{1}{2 \sqrt{2} L \uptau}$, it holds that
	\begin{align}
		\Ex	\|	\z^{r+1} -\bar{\z}^{r+1} \|  
		& \leq
		48 \alpha^2 \uptau^2 L^2    \Ex \|\z^r -\bar{\z}^r\|^2 + 48   \alpha^4 \uptau^4 L^2 N  \Ex \|\grad f(x^{r})\|^2 
		+ 6 \alpha^4 \gamma^2 \uptau^2 L^2    N \textstyle \Ex \|  \sum\limits_{t}  \overline{\nabla f}(\bm{\Phi}^r_t) \|^2 \nonumber \\
		& \quad  + 12 \alpha^4   \uptau^3 L^2 N \sigma^2  + 6 \alpha^4 \gamma^2 \uptau^3 L^2   \sigma^2
		+   3 \alpha^2 \uptau   N \sigma^2.
		\label{cons_ineq_lemma_fed}
	\end{align}
\end{lemma} 
\begin{proof}
	From now on we use the notation $ \sum\limits_{t} \equiv \sum\limits_{t=0}^{\uptau -1} $ and $ \sum\limits_{i} \equiv \sum\limits_{i=1}^{N} $. 	 From \eqref{z_track_fed}
	\begin{align*}
		\Ex_r 	\|	\z^{r+1} -\bar{\z}^{r+1} \|  &=  \Ex_r  \left\|- \alpha  (\I-\A)  \sum_{t=0}^{\uptau-1} \big(\nabla \f(\bm{\Phi}^{r}_{t} ) -\grad \f(\bar{\x}^{r}) + \s_t^r \big) 
		+  \alpha \uptau (\I-\A)  \big(\grad \f(\bar{\x}^{r+1})-\grad \f(\bar{\x}^{r}) \big) \right\|^2 \\
		& \leq 3 \alpha^2 \Ex_r  \left\|  \sum_{t=0}^{\uptau-1} \big(\nabla \f(\bm{\Phi}^{r}_{t} ) -\grad \f(\bar{\x}^{r}) \big) 
		\right\|^2 +  3 \alpha^2 \uptau^2  \Ex_r  \left\| \grad \f(\bar{\x}^{r+1})-\grad \f(\bar{\x}^{r})  \right\|^2  +  3 \alpha^2 \Ex_r \left\|  \sum_{t=0}^{\uptau-1}  \s_t^r \right\|^2 \\
		& \leq 3 \alpha^2 \Ex_r  \left\|  \sum_{t=0}^{\uptau-1} \big(\nabla \f(\bm{\Phi}^{r}_{t} ) -\grad \f(\bar{\x}^{r}) \big) 
		\right\|^2 +  3 \alpha^2 \uptau^2  \Ex_r  \left\| \grad \f(\bar{\x}^{r+1})-\grad \f(\bar{\x}^{r})  \right\|^2  +  3 \alpha^2 \uptau   N \sigma^2.
	\end{align*}
	Observe that
	\begin{align*}
		&   \textstyle \Ex_r \|\textstyle \sum\limits_{t}  \nabla \f(\bm{\Phi}^r_t) - \grad \f(\bar{\x}^r)\|^2 +  \uptau^2  \Ex_r \|\grad \f(\bar{\x}^{r+1})-\grad \f(\bar{\x}^{r})\|^2 \\
		&\leq   \uptau L^2   \textstyle\sum\limits_{t} \textstyle \Ex_r \|  \bm{\Phi}^r_t -  \bar{\x}^r  \|^2 + \uptau^2 L^2  N \textstyle \Ex_r \|x^{r+1}-x^{r}\|^2 \\
		\overset{\eqref{error_fed_avg_progress}}&{=}  \uptau L^2   \textstyle\sum\limits_{t} \textstyle \Ex_r \|  \bm{\Phi}^r_t -  \bar{\x}^r  \|^2 
		+ \uptau^2 L^2 N\textstyle \Ex_r  \| \alpha \gamma \sum\limits_{t} \big( \overline{\nabla f}(\bm{\Phi}^r_t) + \bar{s}_t^r  \big)\|^2 \\		
		& \leq 		  \uptau L^2  \textstyle \Ex_r \|\bm{\widehat{\Phi}}^r\|^2 + 2 \alpha^2 \gamma^2 \uptau^2  L^2   N \textstyle \Ex_r \|  \sum\limits_{t}  \overline{\nabla f}(\bm{\Phi}^r_t) \|^2 +2 \gamma^2 \alpha^2 \uptau^2 L^2    N \|  \sum\limits_{t} \bar{s}_t^r  \|^2 \\
		& \leq 		  \uptau L^2   \textstyle \Ex_r \|\bm{\widehat{\Phi}}^r\|^2 + 2 \alpha^2 \gamma^2 \uptau^2 L^2    N \textstyle \Ex_r \|  \sum\limits_{t}  \overline{\nabla f}(\bm{\Phi}^r_t) \|^2 + 2 \alpha^2 \gamma^2 \uptau^3 L^2   \sigma^2.
	\end{align*}
	Combining the last two bounds
	\begin{align*}
		\Ex	\|	\z^{r+1} -\bar{\z}^{r+1} \|  
		& \leq 3 \alpha^2  \Ex \left\|  \sum_{t=0}^{\uptau-1} \big(\nabla \f(\bm{\Phi}^{r}_{t} ) -\grad \f(\bar{\x}^{r}) \big) 
		\right\|^2 +  3 \alpha^2 \uptau^2  \Ex \left\| \grad \f(\bar{\x}^{r+1})-\grad \f(\bar{\x}^{r})  \right\|^2  +  3 \alpha^2 \uptau   N \sigma^2 \\
		& \leq 3 \alpha^2 \uptau L^2   \textstyle \Ex \|\bm{\widehat{\Phi}}^r\|^2 + 6 \alpha^4 \gamma^2 \uptau^2 L^2    N \textstyle \Ex \|  \sum\limits_{t}  \overline{\nabla f}(\bm{\Phi}^r_t) \|^2 + 6 \alpha^4 \gamma^2 \uptau^3 L^2   \sigma^2
		+   3 \alpha^2 \uptau   N \sigma^2.
	\end{align*}
	Substituting \eqref{drift_bound_fed} into the above inequality yields 
	\begin{align*}
		\Ex	\|	\z^{r+1} -\bar{\z}^{r+1} \|  
		& \leq 3 \alpha^2 \uptau L^2   (16 \uptau  \Ex \|\z^r -\bar{\z}^r\|^2 + 16   \alpha^2 \uptau^3 N  \Ex \|\grad f(x^{r})\|^2 + 4 \alpha^2   \uptau^2 N \sigma^2) \\
		& \quad  + 6 \alpha^4 \gamma^2 \uptau^2 L^2    N \textstyle \Ex \|  \sum\limits_{t}  \overline{\nabla f}(\bm{\Phi}^r_t) \|^2 + 6 \alpha^4 \gamma^2 \uptau^3 L^2   \sigma^2
		+   3 \alpha^2 \uptau   N \sigma^2 \\
		& = 48 \alpha^2 \uptau^2 L^2    \Ex \|\z^r -\bar{\z}^r\|^2 + 48   \alpha^4 \uptau^4 L^2 N  \Ex \|\grad f(x^{r})\|^2 
		+ 6 \alpha^4 \gamma^2 \uptau^2 L^2    N \textstyle \Ex \|  \sum\limits_{t}  \overline{\nabla f}(\bm{\Phi}^r_t) \|^2 \\
		& \quad  + 12 \alpha^4   \uptau^3 L^2 N \sigma^2  + 6 \alpha^4 \gamma^2 \uptau^3 L^2   \sigma^2
		+   3 \alpha^2 \uptau   N \sigma^2.
	\end{align*}

\end{proof}

\subsection{Nonconvex case}
Recall from \eqref{error_fed_avg_progress} that $x^{r+1}=x^{r} -  \frac{\alpha \gamma}{N} \sum_{t=0}^{\uptau-1} \sum_{i=1}^N \big( \nabla f_i(\phi^r_{i,t}) + s^r_{i,t} \big)$.	Substituting $y=x^{r+1}$ and $z=x^{r}$ into  inequality \eqref{bound:L_smooth_function}, and taking conditional expectation, we get
\begin{align}
	\textstyle	\Ex_r   f(x^{r+1}) 
	&\leq f(x^{r}) - \alpha \gamma \textstyle \Ex_r \big\langle \grad f(x^{r}),  \textstyle \tfrac{1}{N} \sum\limits_{t} \sum\limits_{i} \big( \nabla f_i(\phi^r_{i,t}) + s^r_{i,t} \big) \big\rangle  +\tfrac{\alpha^2 \gamma^2 L}{2} \Ex_r \| \textstyle \tfrac{1}{N} \sum\limits_{t} \sum\limits_{i} \big( \nabla f_i(\phi^r_{i,t}) + s^r_{i,t}) \|^2  \nonumber \\
	&\leq f(x^{r}) - \alpha \gamma \textstyle \Ex_r  \textstyle   \big\langle \grad f(x^{r}), \frac{1}{N}  \sum\limits_{t} \sum\limits_{i} \nabla f_i(\phi^r_{i,t})  \big\rangle 
	+ \alpha^2 \gamma^2 \uptau L  \textstyle \sum\limits_{t}  \| \textstyle \tfrac{1}{N} \sum\limits_{i} \nabla f_i(\phi^r_{i,t})  \|^2 +  	\tfrac{\alpha^2 \gamma^2  \uptau L \sigma^2}{N},
\end{align}
where $\Ex_{r}$ denote the expectation conditioned on the all iterates up to $r$. Note that
\begin{align}
	-   &  \textstyle   \big\langle \grad f(x^{r}), \frac{1}{N}  \sum\limits_{t} \sum\limits_{i} \nabla f_i(\phi^r_{i,t})  \big\rangle \nonumber \\
	&=	-     \textstyle \sum\limits_{t}    \big\langle \grad f(x^{r}), \frac{1}{N}  \sum\limits_{i} \nabla f_i(\phi^r_{i,t})  \big\rangle \nonumber \\
	&=  - \tfrac{\uptau}{2}  \|\grad f(x^{r})\|^2 - \textstyle \tfrac{1}{2}  \sum\limits_{t}  \| \textstyle \tfrac{1}{N}  \sum\limits_{i} \nabla f_i(\phi^r_{i,t}) \|^2 + \tfrac{1}{2}  \sum\limits_{t}  \|\textstyle \tfrac{1}{N} \sum\limits_{i} \nabla f_i(\phi^r_{i,t})-\grad f(x^{r}) \|^2 \nonumber \\
	&\leq   - \tfrac{\uptau}{2}  \|\grad f(x^{r})\|^2 - \textstyle \tfrac{1}{2}  \sum\limits_{t}  \| \textstyle \tfrac{1}{N}  \sum\limits_{i} \nabla f_i(\phi^r_{i,t}) \|^2 + \tfrac{1}{2 N}  \sum\limits_{t} \sum\limits_{i} \|   \nabla f_i(\phi^r_{i,t})-\grad f_i(x^{r}) \|^2  \nonumber \\
	&\leq   - \tfrac{\uptau}{2}  \|\grad f(x^{r})\|^2 - \textstyle \tfrac{1}{2}  \sum\limits_{t}  \| \textstyle \tfrac{1}{N}  \sum\limits_{i} \nabla f_i(\phi^r_{i,t}) \|^2 + \tfrac{L^2}{2 N}   \| \bm{\widehat{\Phi}}^r \|^2,
\end{align} 
where the second bound holds from Jensen's inequality. 	Combining the last two equations and taking expectation yields
\begin{align}
	\Ex   f(x^{r+1})  &\leq \Ex f(x^{r}) - \tfrac{\alpha \gamma \uptau}{2} \Ex \| \grad f(x^{r})  \|^2 - \tfrac{\alpha \gamma}{2} (1-2\alpha \gamma \uptau L )  \textstyle \sum\limits_{t} \Ex \| \tfrac{1}{N}  \sum\limits_{i} \nabla f_i(\phi^r_{i,t})\|^2 \nonumber \\
	& \quad +\tfrac{\alpha \gamma L^2}{2 N}   \Ex \| \bm{\widehat{\Phi}}^r \|^2 +	\tfrac{  \alpha^2 \gamma^2 \uptau L \sigma^2 }{ N}. \label{2bndsbd_fed}
\end{align} 
Substituting the bound \eqref{drift_bound_fed} into inequality \eqref{2bndsbd_fed} and taking expectation yields
\begin{align}
	\Ex   f(x^{r+1})  &\leq \Ex f(x^{r}) - \tfrac{\alpha \gamma \uptau}{2} \Ex \| \grad f(x^{r})  \|^2 - \tfrac{\alpha \gamma}{2} (1-2\alpha \gamma \uptau L )  \textstyle \sum\limits_{t} \Ex \| \tfrac{1}{N}  \sum\limits_{i} \nabla f_i(\phi^r_{i,t})\|^2 \nonumber \\
	& \quad +\tfrac{\alpha \gamma L^2}{2 N}   (16 \uptau  \Ex \|\z^r -\bar{\z}^r\|^2 + 16   \alpha^2 \uptau^3 N  \Ex \|\grad f(x^{r})\|^2 + 4 \alpha^2   \uptau^2 N \sigma^2) +	\tfrac{  \alpha^2 \gamma^2 \uptau L \sigma^2 }{ N} \nonumber \\
	&=  \Ex f(x^{r}) - \tfrac{\alpha \gamma \uptau}{2} (1-16 \alpha^2  \uptau^2 L^2) \Ex \| \grad f(x^{r})  \|^2 - \tfrac{\alpha \gamma}{2} (1-2\alpha \gamma \uptau L )  \textstyle \sum\limits_{t} \Ex \| \tfrac{1}{N}  \sum\limits_{i} \nabla f_i(\phi^r_{i,t})\|^2 \nonumber \\
	& \quad +\tfrac{8\alpha  \gamma \uptau L^2}{ N}    \Ex \|\z^r -\bar{\z}^r\|^2  + 2 \alpha^3 \gamma   \uptau^2 L^2  \sigma^2 +	\tfrac{  \alpha^2 \gamma^2 \uptau L \sigma^2 }{ N}.
\end{align} 
When  $\alpha \gamma \leq \frac{1}{4 \sqrt{2}  \uptau L}$, we can upper bound the previous inequality by
\begin{equation} \label{descent_inquality_fed}	
	\begin{aligned}
		\Ex   f(x^{r+1}) &\leq  \Ex  f(x^{r}) - \tfrac{\alpha \gamma \uptau}{4}  \Ex  \| \grad f(x^{r})  \|^2  - \tfrac{ \alpha \gamma}{4}  \textstyle  \sum\limits_{t} \|   \tfrac{1}{N} \sum\limits_{i} \nabla f_i(\phi^r_{i,t})\|^2 \nonumber \\ 
		& \quad +\tfrac{8\alpha  \gamma \uptau L^2}{ N}    \Ex \|\z^r -\bar{\z}^r\|^2  + 2 \alpha^3 \gamma   \uptau^2 L^2  \sigma^2 +	\tfrac{  \alpha^2 \gamma^2 \uptau L \sigma^2 }{ N}.
	\end{aligned}	
\end{equation}
It follows that
\begin{align} \label{grad_ineq_noncvx_proof_fed}
	\cE_r 		  &\leq \tfrac{4}{\alpha \gamma \uptau} \left( \Ex  \tilde{f}(x^{r}) -  \Ex   \tilde{f}(x^{r+1}) \right)
	+\tfrac{32 L^2}{ N}    \Ex \|\z^r -\bar{\z}^r\|^2 + 8 \alpha^2 \uptau L^2    \sigma^2 +	\tfrac{ 4 \alpha \gamma  L \sigma^2 }{ N},
\end{align}	
where  $\cE_r \define  \textstyle   \Ex  \| \grad f(x^{r})  \|^2  +  \frac{1}{\uptau} \sum\limits_{t} \|  \tfrac{1}{N}  \sum\limits_{i} \nabla f_i(\phi^r_{i,t})\|^2$ and $\tilde{f}(x^{r})\define f(x^{r})-f^\star$. Averaging over $r=0,1,\ldots,R-1$  and using $-\tilde{f}(x^{r})\leq 0$, we get
\begin{align} \label{sum_grad_ineq_noncvx_proof_fed}
	\frac{1}{R}	 \sum_{r=0}^{R-1} \cE_r  		 
	&\leq \frac{4   \tilde{f}(\bar{x}^{0})}{\alpha \gamma \uptau R}	  
	+ \frac{32  L^2    }{ N R}	 \sum_{r=0}^{R-1}   \Ex \|\z^r -\bar{\z}^r\|^2  + 8 \alpha^2 \uptau L^2    \sigma^2 +	\tfrac{ 4 \alpha \gamma  L \sigma^2 }{ N}.
\end{align}		
We now bound the term $ \sum_{r=0}^{R-1}   \Ex \|\z^r -\bar{\z}^r\|^2 $. 	Using $	48 \alpha^2 \uptau^2 L^2 \leq 1/2$
in \eqref{cons_ineq_lemma_fed}, it holds that
\begin{align}
	\Ex	\|	\z^{r+1} -\bar{\z}^{r+1} \|^2  
	& \leq
	\tfrac{1}{2}    \Ex \|\z^r -\bar{\z}^r\|^2 +   \alpha^2 \uptau^2  N  \Ex \|\grad f(x^{r})\|^2 
	+  \alpha^2 \gamma^2     N \textstyle \Ex \|  \sum\limits_{t}  \overline{\nabla f}(\bm{\Phi}^r_t) \|^2 \nonumber \\
	& \quad +  6 \alpha^2 \gamma^2 \uptau    \sigma^2
	+ 4 \alpha^2 \uptau   N \sigma^2.
\end{align}
Iterating yields 
\begin{align} \label{cons_ineq_noncvx_proof_fed}
	\Ex	\|	\z^{r} -\bar{\z}^{r} \|^2  
	& \leq
	(\tfrac{1}{2})^r    \Ex \|\z^0 -\bar{\z}^0\|^2 +   \alpha^2 \uptau^2  N   \sum_{\ell=0}^{r-1} \left( \tfrac{1}{2} \right)^{r-1-\ell}     \cE_\ell 
	+  12 \alpha^2 \gamma^2 \uptau    \sigma^2
	+ 8 \alpha^2 \uptau   N \sigma^2.
\end{align}
Averaging over $r=1,\dots,R$ 
\begin{align} 
	\frac{1}{R} \sum_{r=1}^R		\Ex	\|	\z^{r} -\bar{\z}^{r} \|^2  	  &\leq  
	\frac{2   \|\z^0 -\bar{\z}^0\|^2}{R} 
	+    \tfrac{ \alpha^2 \uptau^2   N }{  R} \sum_{r=1}^R  \sum_{\ell=0}^{r-1} \left( \tfrac{1}{2} \right)^{r-1-\ell}   \cE_\ell    	+  12 \alpha^2 \gamma^2 \uptau    \sigma^2
	+ 8 \alpha^2 \uptau   N \sigma^2. \nonumber \\
	& \leq  
	\frac{2   \|\z^0 -\bar{\z}^0\|^2}{R} 
	+  \tfrac{2 \alpha^2 \uptau^2   N }{  R}   \sum_{r=0}^{R-1}    \cE_r   	+  12 \alpha^2 \gamma^2 \uptau    \sigma^2
	+ 8 \alpha^2 \uptau   N \sigma^2.
\end{align}  
Hence,
\begin{align}  \label{bound_noncvx_cons_final_fed}
	\frac{1}{R} \sum_{r=0}^{R-1}		\Ex	\|	\z^{r} -\bar{\z}^{r} \|  	
	& \leq  
	\frac{3   \|\z^0 -\bar{\z}^0\|^2}{R} 
	+  \tfrac{2 \alpha^2 \uptau^2   N }{  R}   \sum_{r=0}^{R-1}    \cE_r   	+  12 \alpha^2 \gamma^2 \uptau    \sigma^2
	+ 8 \alpha^2 \uptau   N \sigma^2.
\end{align}  
Substituting inequality \eqref{bound_noncvx_cons_final_fed} into \eqref{sum_grad_ineq_noncvx_proof_fed} and rearranging, we obtain 
\begin{align} 
	\left(1-64 \alpha^2 \uptau^2 L^2     \right)	\frac{1}{R}	 \sum_{r=0}^{R-1} \cE_r  		 
	&\leq \frac{4   \tilde{f}(\bar{x}^{0})}{\alpha \gamma \uptau R}	+ 	 \frac{96 L^2     \|\z^0 -\bar{\z}^0\|^2}{  N R}  
	\nonumber \\
	& \quad + \frac{384 \alpha^2 \gamma^2 \uptau   L^2 \sigma^2     }{ N }	
	+ 264 \alpha^2 \uptau L^2    \sigma^2  +	\frac{ 4 \alpha \gamma  L \sigma^2 }{ N}.
\end{align}		
If we set $1-64 \alpha^2 \uptau^2 L^2  \geq  1/2$, then it holds that
\begin{align} 
	\frac{1}{R}	 \sum_{r=0}^{R-1} \cE_r  		 
	&\leq \frac{8   \tilde{f}(\bar{x}^{0})}{\alpha \gamma \uptau R}	+ 	 \frac{192 L^2     \|\z^0 -\bar{\z}^0\|^2}{  N R}  
	+ \frac{768 \alpha^2 \gamma^2 \uptau   L^2 \sigma^2     }{ N }	
	+ 528 \alpha^2 \uptau L^2    \sigma^2  +	\frac{ 8 \alpha \gamma  L \sigma^2 }{ N} \nonumber \\
		&\leq \frac{8   \tilde{f}(\bar{x}^{0})}{\alpha \gamma \uptau R}	+ 	 \frac{192 L^2  \alpha^2 \uptau^2  \varsigma^2}{  R}  
	+ \frac{768 \alpha^2 \gamma^2 \uptau   L^2 \sigma^2     }{ N }	
	+ 528 \alpha^2 \uptau L^2    \sigma^2  +	\frac{ 8 \alpha \gamma  L \sigma^2 }{ N}.
\end{align}	
where  $\varsigma_{0}^2=(1/N) \sum\limits_{i=1}^N \|\grad f_i(\bar{x}^0) - \grad f(\bar{x}^0)\|^2$ and the last step holds from \eqref{z_def_fed} and the fact that $\y^0=\zero$.
 Let $\tilde{\alpha}=\alpha \gamma \uptau$, then
 \begin{align} 
 	\frac{1}{R}	 \sum_{r=0}^{R-1} \cE_r  		 
 	&\leq \frac{8   \tilde{f}(\bar{x}^{0})}{\tilde{\alpha} R}	+ 	 \frac{192 L^2  \tilde{\alpha}^2  \varsigma_{0}^2}{ \gamma^2  R}  
 	+ \frac{768 \tilde{\alpha}^2   L^2 \sigma^2     }{\uptau N }	
 	+ 528 \frac{\tilde{\alpha}^2 L^2    \sigma^2}{\uptau \gamma^2}  +	\frac{ 8 \tilde{\alpha}  L \sigma^2 }{\uptau N}.
 \end{align}
 Setting $\gamma=\sqrt{N}$, we obtain
  \begin{align} 
 	\frac{1}{R}	 \sum_{r=0}^{R-1} \cE_r  		 
 	&\leq \frac{8   \tilde{f}(\bar{x}^{0})}{\tilde{\alpha} R}	+ 	 \frac{192 L^2  \tilde{\alpha}^2  \varsigma_{0}^2}{ N  R}  
 	+ \frac{768 \tilde{\alpha}^2   L^2 \sigma^2     }{\uptau N }	
 	+ 528 \frac{\tilde{\alpha}^2 L^2    \sigma^2}{\uptau N}  +	\frac{ 8 \tilde{\alpha}  L \sigma^2 }{\uptau N} \nonumber \\
 		&\leq \frac{c   \tilde{f}(\bar{x}^{0})}{\tilde{\alpha} R}	+ 	 \frac{c L^2  \tilde{\alpha}^2  \varsigma_{0}^2}{ N  R}  
   +	\frac{ c \tilde{\alpha}  L \sigma^2 }{\uptau N},
 \end{align}
 where we used  $\alpha \gamma \uptau = \tilde{\alpha} \leq \frac{1}{4 \sqrt{2} L}$ and $C$ is a constant.	
 	The final rate can be obtained by tuning the stepsize in a way similar to \cite{karimireddy2019scaffold}. Specifically, letting $\alpha \gamma \uptau = \tilde{\alpha} = \cO \left( \min\{\frac{1}{L},\frac{ \tilde{f}(\bar{x}^{0})}{\sigma} \sqrt{\frac{N \uptau}{R}}\} \right)$ and using $\uptau \leq R$ yields:
 		\begin{align} 
 		\frac{1}{R}	 \sum_{r=0}^{R-1} \cE_r  		 
 		&\leq \cO \left(	\frac{ \sigma }{ \sqrt{N R \uptau}} +  \frac{   \tilde{f}(\bar{x}^{0}) + \varsigma_{0}^2}{  R } .
 	 \right)
 	\end{align}
The proof for the convex cases can also be specialized for the centralized case to obtain the rate given in Table \ref{table_rates}.	
			
\end{document}